\documentclass[10pt]{elsarticle}
\usepackage{amssymb, amsmath, mathrsfs}
\usepackage[margin=1.5in]{geometry}
\usepackage{bm}
\usepackage{indentfirst}
\usepackage{latexsym}
\usepackage{graphicx}
\usepackage{subfig}
\usepackage{caption}
\usepackage[colorlinks]{hyperref}
\usepackage{pdfsync}
\usepackage{tabularx}
\usepackage{multirow}
\graphicspath{{./fig/}}
\usepackage{color, xcolor}

\newtheorem{theorem}{Theorem}
\newenvironment{proof}{\textbf{Proof.}\,}{\par\hfill $\square$\par}
\newtheorem{remark}{Remark}
\newtheorem{lemma}{Lemma}
\allowdisplaybreaks[4]
\begin{document}
    \captionsetup[figure]{labelformat={default}, labelsep=period, name={Fig.}}

    \begin{frontmatter}
        \title{Efficient finite element schemes for a phase field model of two-phase incompressible flows with different densities}
        \author[a]{Jiancheng Wang}
        \ead{202311110403@std.uestc.edu.cn} %lacus19980908@163.com}
        \author[a]{Maojun Li}
        \ead{limj@uestc.edu.cn}
        \author[b]{Cheng Wang\corref{cor1}}
        \ead{cwang1@umassd.edu}
        \address[a]{School of Mathematical Sciences, University of Electronic Science and Technology of China, Sichuan, 611731, P.R. China }
        \address[b]{Mathematics Department, University of Massachusetts, North Dartmouth, MA 02747, USA}
        \cortext[cor1]{Corresponding author.}

        \begin{abstract}
            In this paper, we present two multiple scalar auxiliary variable (MSAV)-based, finite element numerical schemes for the Abels-Garcke-Gr{\"u}n (AGG) model, which is a thermodynamically consistent phase field model of two-phase incompressible flows with different densities. Both schemes are decoupled, linear, second-order in time, and the numerical implementation turns out to be straightforward. The first scheme solves the Navier-Stokes equations in a saddle point formulation, while the second one employs the artificial compressibility method, leading to a fully decoupled structure with a time-independent pressure update equation. In terms of computational cost, only a sequence of independent elliptic or saddle point systems needs to be solved at each time step. At a theoretical level, the unique solvability and unconditional energy stability (with respect to a modified energy functional) of the proposed schemes are established. In addition, comprehensive numerical simulations are performed to verify the effectiveness and robustness of the proposed schemes.
        \end{abstract}

    \end{frontmatter}

    \section{Introduction}
        Two-phase incompressible flows are fundamental components of numerous real-life physical phenomena and play a pivotal role in multiphysics simulations, with diverse industrial applications, such as bubble column modeling in chemical and biological engineering processes \cite{2011_gross}, and bubble dynamics modeling in metallurgy processes \cite{2014_zhang}. Meanwhile, accurately capturing the dynamic interface is of paramount importance in numerical simulations of two-phase incompressible flows, as it enables a comprehensive understanding of the intricate dynamics exhibited by such a system. Among various interface modeling methods (such as the volume of fluid method \cite{1981_hirt}, level set method \cite{1994_sussman}, arbitrary Lagrange-Eulerian method \cite{2022_duan}), we are particularly interested in the phase field approach, due to its ability to circumvent difficulties associated with simulating discontinuities in physical quantities across interfaces and its adherence to an energy law.

        The classical phase field model of two-phase incompressible flows (the so-called Model H) originates from the seminal work of Hohenberg and Halperin \cite{1977_hohenberg}, and was subsequently refined in a thermodynamically consistent manner by Gurtin et al. \cite{1996_gurtin}. There have been enormous outstanding achievements for Model H, see \cite{2023_cai, chen24b, chen22b, 2017_diegel, 2006_feng, 2023_feng, 2015_han, 2017_han} and the references therein. On the other hand, such a model is based on a critical assumption that the density ratio is equal to one. Although the Boussinesq approximation \cite{2003_liu} is applicable for different densities, it does not hold in the presence of a large density ratio. Furthermore, incorporating density variation directly into the momentum equation is inappropriate, as it is crucial to ensure the model is thermodynamically consistent in the case of a large density ratio. Therefore, it is not suitable to employ Model H for simulating many engineering problems, such as the behavior of rising gas bubbles in liquid metal flows under external magnetic fields \cite{2014_zhang}. Various thermodynamically consistent phase field models (Cahn-Hilliard-Navier-Stokes system) have been proposed in the literature; see \cite{2012_abels, 2014_aki, 2022_haddad, 2021_guo, 2015_liu, 1998_lowengrub, 2018_roudbari, 2023_ten}, and a comprehensive exploration can be found in \cite{2023_ten}. Among these available phase field models, our specific focus lies on the so-called Abels-Garcke-Gr{\"u}n (AGG) model, proposed by Abels et al. \cite{2012_abels}. This model employs volume-averaged velocity to ensure solenoidal velocity, while refining the conventional continuity equation and momentum equation by incorporating a diffusive mass flux determined by the density ratio and chemical potential. In turn, gravitational energy cannot be taken into consideration as the conventional continuity equation is no longer valid, and such an approach leads to an energy dissipation law in the absence of external gravity forces. On the other hand, designing efficient numerical schemes for the AGG model is much easier, in comparison with other models mentioned above, since it is a generalization of Model H in unmatched density cases. The theoretical results (such as the global well-posedness and regularity analysis) of this model have been exhaustively investigated; see \cite{2023_abels} and the references therein. Also, the AGG model has been generalized to many complex multi-physics problems \cite{2012_campillo, 2018_dong, 2016_grun, 2021_rohde}, etc. 

        Now, our literature review will focus exclusively on the numerical schemes for the AGG model. Gr{\"u}n and Klingbeil \cite{2014_grun} proposed a second-order in time and unconditionally energy stable scheme, with its convergence established in \cite{2013_grun}. However, their scheme is nonlinear and coupled, so a theoretical justification of unique solvability is not available in their work. Additionally, the convergence rate analysis is absent. Shen and Yang \cite{2015_shen} employed the pressure stabilization method \cite{2009_guermond, 2011_guermond} to construct a fully decoupled, linear, first-order in time, and unconditionally energy stable semi-implicit scheme; also see \cite{2016_chen}. Gong et al. \cite{2018_gong} introduced a novel energy quadratization technique to derive a linear and second-order in time scheme, and unconditional energy stability was proved, subject to periodic boundary conditions. Yang and Dong \cite{2019_yang} employed a scalar auxiliary variable (SAV) approach, which combines the kinetic energy and the potential free energy to reformulate the model, and proposed a second-order in time and unconditionally energy stable scheme. In fact, the momentum equation is transformed into a curl form in their work, and the vorticity variable is introduced to decouple the velocity and pressure gradient. Such an effort leads to a fully decoupled structure, and only a sequence of time-independent coefficient linear systems and one nonlinear algebraic equation need to be solved at each time step. However, a very small time step is needed in their scheme to obtain a reliable simulation result. Khanwale et al. \cite{2020_khanwale} proposed an unconditionally energy stable Crank-Nicolson type scheme, while a residual-based variational multiscale method was used in the Navier-Stokes part for pressure stabilization, and a massively parallel adaptive octree-based meshing framework was also developed. Moreover, the pressure projection idea was employed and a projection-based semi-implicit scheme was derived in \cite{2023_khanwale}. Meanwhile, a linear, second-order in time and unconditionally energy stable scheme was proposed by Fu and Han \cite{2021_fu}, and a classical residual-based stabilization was applied in the advection-dominated regime. On the other hand, this is a coupled numerical scheme, which leads to a higher computational cost. And they failed to utilize the ``zero-energy-contribution" (ZEC) approach \cite{2021_yang_1} in constructing a decoupled scheme, as an order one constancy error was observed in their numerical experiments. Nevertheless, the ZEC approach is still very useful in obtaining a decoupled structure. For example, following the work~\cite{2015_shen}, Chen and Yang \cite{2021_chen} proposed a fully decoupled, linear, and unconditionally energy stable scheme, while the pressure stabilization method, Strang operator splitting method, and ZEC approach were utilized to decouple the system; also see the related work~\cite{2022_chen}. The primary shortcoming of this numerical effort is associated with the fact that their schemes were only limited to first-order accuracy in time.

        Inspired by the previous studies, our objective in this paper is to develop decoupled, linear, second-order in time, and unconditionally energy stable, fully discrete finite element numerical schemes for the AGG model. In particular, such an effort is based on decoupling methods for incompressible flows, such as projection and quasi-compressibility approaches. This numerical effort is undoubtedly very challenging, since the densities and dynamic viscosities in the momentum equation are determined by the order parameter, and more coupled terms are involved than in Model H. Moreover, a theoretical justification of energy stability also poses a formidable challenge, especially for second-order in time, fully decoupled schemes. In fact, the numerical scheme presented in \cite{2019_yang} is the only one that fulfills all the aforementioned properties. However, this algorithm encounters significant challenges in numerical simulation and lacks scalability for more intricate two-phase incompressible flows. In order to accomplish these objectives, the following strategies are implemented:
        \begin{itemize}
            \item The conventional numerical techniques, such as semi-implicit methods, are not ideal to achieve the objectives outlined above. To facilitate the numerical design and theoretically obtain a discrete energy law, we reformulate the dimensionless AGG model by some suitable identical transformations. Afterwards, the MSAV approach \cite{2018_cheng, 2022_li} is employed to rewrite the system into an equivalent form in the continuous sense. Of course, the nonlinear potential in the Cahn-Hilliard equation can be linearized by this approach. In comparison with the standard SAV approach \cite{2018_shen_2}, the MSAV formulation can effectively capture the disparate evolution process in the total energy \cite{2018_cheng}.
            \item To ensure unconditional energy stability in developing a linear scheme, we incorporate the ZEC feature into the scalar auxiliary variables of the MSAV approach, and treat all the coupled terms explicitly. In addition, extrapolation formulas are utilized to linearize all variables related to the order parameter in the momentum equation, such as the density and dynamic viscosity.
            \item The backward difference formula (BDF) is used as the temporal discretization. The first scheme is derived by a direct application of Galerkin finite element spatial approximation. To further develop a fully decoupled finite element scheme with a time-independent pressure update equation, we employ the artificial compressibility method \cite{2021_Ern_2} and follow the ideas in \cite{2017_decaria, 2022_gao, 2011_guermond} to deal with the Navier-Stokes equations. The second proposed scheme is based on a combination of these numerical ideas.
        \end{itemize}

        The rest of this paper is organized as follows. In Section \ref{section_2}, we introduce the dimensionless governing equations and employ the MSAV approach for an equivalent reformulation. In Section \ref{section_3}, we propose the associated numerical schemes, and prove the unique solvability and unconditional energy stability. The numerical implementation is outlined as well. In Section \ref{section_4}, various numerical simulation results are presented to validate the robustness of the proposed schemes, and the impact of different mobility forms is investigated. Finally, some concluding remarks are made in Section \ref{section_5}, and the related future works are mentioned.

    \section{Governing equations and MSAV reformulation}
    \label{section_2}
        \subsection{Governing equations}
            Let $\Omega\subset\mathbb{R}^d (d=2,3)$ be a bounded and connected domain with Lipschitz continuous boundary $\partial\Omega$. The following dimensionless thermodynamically consistent phase field model of two-phase incompressible flows is considered, with different densities (for the dimensional AGG model, see \cite{2012_abels} for details):
            \begin{subequations}
                \label{OCHNS}
                \begin{align}
                    \phi_t+\nabla\cdot\left(\bm{u}\phi\right)=&\frac{1}{\mathrm{Pe}}\nabla\cdot\left(m\nabla\mu\right),\\
                    -\mathrm{Cn}^2\Delta\phi+F'(\phi)=&\mu,\\
                    \rho_t+\nabla\cdot\left(\rho\bm{u}+\bm{J}\right)=&0,\label{OCHNS_rho}\\
                    \rho\bm{u}_t+\left(\rho\bm{u}+\bm{J}\right)\cdot\nabla\bm{u}=&\frac{1}{\mathrm{Re}}\nabla\cdot\left(2\eta\mathbb{D}(\bm{u})\right)-\nabla{p}-\frac{\mathrm{Cn}}{\mathrm{We}}\nabla\cdot\left(\nabla\phi\nabla\phi\right)+\frac{1}{\mathrm{Fr}}\bm{f},\label{OCHNS_u}\\
                    \nabla\cdot\bm{u}=&0.
                \end{align}
            \end{subequations}
            In the above equations, $\phi$ is the order parameter that labels the two immiscible fluids such that
            \begin{equation}
                \phi(\bm{x},t)=
                \begin{cases}
                    1,  & \text{fluid 1},\\
                    -1, & \text{fluid 2}.
                \end{cases}
            \end{equation}
            Moreover, $\mu$ is the chemical potential, $m=m(\phi)\ge{0}$ is the mobility, $F(\phi)=\frac{1}{4}(\phi^2-1)^2$ is the Ginzubrg-Landau double-well potential, $\rho$ and $\eta$ are the density and dynamic viscosity defined by
            \begin{equation}
                \rho:=\frac{\rho_1-\rho_2}{2\rho_r}\phi+\frac{\rho_1+\rho_2}{2\rho_r}\quad\text{and}\quad\eta:=\frac{\eta_1-\eta_2}{2\eta_r}\phi+\frac{\eta_1+\eta_2}{2\eta_r},
            \end{equation}
            where $\rho_1$ (respective to $\rho_2$) and $\eta_1$ (respective to $\eta_2$) are the density and dynamic viscosity of fluid 1 (respective to fluid 2), and $\rho_r,\eta_r$ stand for the reference density and reference dynamic viscosity. The variable $\bm{u}$ is the velocity, $p$ is the pressure, $\mathbb{D}(\bm{u})=\frac{1}{2}(\nabla\bm{u}+\nabla\bm{u}^T)$ is the deformation tensor, $\bm{f}=\bm{f}(\rho)$ is the external gravity force.
            In equation \eqref{OCHNS_rho}, $\bm{J}$ is the diffusion flux related to the density difference,  given by
            \begin{equation}
                \bm{J}=-\frac{\rho_1-\rho_2}{2\rho_r\mathrm{Pe}}m\nabla\mu.
            \end{equation}

            For the details of non-dimensional form, let $L_r,u_r,\epsilon,\lambda,M,g$ denote the reference length, reference velocity, interface thickness, surface tension, reference mobility and gravity acceleration, respectively. In turn, the dimensionless numbers are given by the following parameters:
            \begin{itemize}
                \item Cahn number $\mathrm{Cn}=\displaystyle\frac{\epsilon}{L_r}$
                \item (effective) P\'{e}clet number $\mathrm{Pe}=\displaystyle\frac{2\sqrt{2}\epsilon{L}_ru_r}{3\lambda{M}}$
                \item Reynolds number $\mathrm{Re}=\displaystyle\frac{L_r\rho_ru_r}{\eta_r}$
                \item Weber number $\mathrm{We}=\displaystyle\frac{2\sqrt{2}L_r\rho_ru_r^2}{3\lambda}$
                \item Froude number $\mathrm{Fr}=\displaystyle\frac{u_r^2}{gL_r}$
            \end{itemize}
            \begin{remark}
                Here we give the details to determine these dimensionless numbers. In fact, the Froude number $\mathrm{Fr}$ is set to one, and this fixes the reference velocity to be $u_r=\sqrt{gL_r}$. Subsequently, $\mathrm{Re}$ and $\mathrm{We}$ can be determined by setting appropriate $L_r,\rho_r,\eta_r$, with given $\sigma$. In general, $\mathrm{Pe}$ is assumed to be a function of $\mathrm{Cn}$. In the case of constant mobility, the scaling law $1/\mathrm{Pe}=3\mathrm{Cn}$ is preferred in the literature \cite{2020_khanwale, 2023_khanwale, 2013_magaletti}; while in the degenerate mobility case (as a typical choice of $(\phi^2-1)^2$), the scaling law $1/\mathrm{Pe}=$ constant is preferred in the literature \cite{2012_aland, 2021_fu, 2019_wang}. Following the above scaling laws, we are able to complete the parameter setting by choosing the interface thickness $\epsilon.$
            \end{remark}

            To facilitate the numerical design, we first make use of the identity
            \begin{equation}
                \nabla\cdot\left(\nabla\phi\nabla\phi\right)=\frac{1}{2}\nabla\left|\nabla\phi\right|^2+\Delta\phi\nabla\phi,
            \end{equation}
            and transform the right-hand side of \eqref{OCHNS_u} into
            \begin{equation}
                \mathrm{RHS}=\frac{1}{\mathrm{Re}}\nabla\cdot\left(2\eta\mathbb{D}(\bm{u})\right)-\nabla{P}-\frac{1}{\mathrm{We}\mathrm{Cn}}\phi\nabla\mu+\frac{1}{\mathrm{Fr}}\bm{f} ,
            \end{equation}
            in which the so-called effective pressure has been introduced:
            \begin{equation}
                P:=p+\frac{\mathrm{Cn}}{2\mathrm{We}}|\nabla\phi|^2+\frac{1}{\mathrm{We}\mathrm{Cn}}\left(F(\phi)-\mu\phi\right).
            \end{equation}
            In turn, we multiply \eqref{OCHNS_rho} by $\frac{1}{2}\bm{u}$ and add it to \eqref{OCHNS_u}. By introducing a new variable $\sigma=\sqrt{\rho},$ the left-hand side of \eqref{OCHNS_u} turns out to be
            \begin{equation}
                \mathrm{LHS}=\sigma(\sigma\bm{u})_t+(\rho\bm{u}+\bm{J})\cdot\nabla\bm{u}+\frac{1}{2}\nabla\cdot\left(\rho\bm{u}+\bm{J}\right)\bm{u}.
            \end{equation}
            Hence, system \eqref{OCHNS} could be rewritten as the following form:
            \begin{subequations}
                \label{CHNS}
                \begin{align}
                    \phi_t+\nabla\cdot\left(\bm{u}\phi\right)=&\frac{1}{\mathrm{Pe}}\nabla\cdot\left(m\nabla\mu\right),\\
                    -\mathrm{Cn}^2\Delta\phi+F'(\phi)=&\mu,\\
                    \sigma(\sigma\bm{u})_t+\left(\rho\bm{u}+\bm{J}\right)\cdot\nabla\bm{u}+\frac{1}{2}\nabla\cdot\left(\rho\bm{u}+\bm{J}\right)\bm{u}=&\frac{1}{\mathrm{Re}}\nabla\cdot\left(2\eta\mathbb{D}(\bm{u})\right)-\nabla{P}-\frac{1}{\mathrm{We}\mathrm{Cn}}\phi\nabla\mu+\frac{1}{\mathrm{Fr}}\bm{f},\\
                    \nabla\cdot\bm{u}=&0 .
                \end{align}
            \end{subequations}
            To form a closed PDE system, we take the following boundary conditions:
            \begin{equation}
                \label{bc}
                \begin{aligned}
                    \bm{n}\cdot\nabla\phi|_{\partial\Omega}=&0,\\
                    \bm{n}\cdot\nabla\mu|_{\partial\Omega}=&0,\\
                    \bm{u}|_{\partial\Omega}=&0,
                \end{aligned}
            \end{equation}
            and the initial conditions:
            \begin{equation}
                \label{ic}
                \begin{aligned}
                    \phi(0)=&\phi^0,\\
                    \bm{u}(0)=&u^0.
                \end{aligned}
            \end{equation}
            Notice that other appropriate physical boundary conditions could also be adopted, such as periodic and free-slip ones.

            By taking $\bm{f}=0$, system \eqref{CHNS} with boundary condition \eqref{bc} and initial condition \eqref{ic} satisfies the following energy dissipation law:
            \begin{equation}
                \frac{\mathrm{d}}{\mathrm{d}t}\mathscr{E}(t)=-\mathscr{P}(t)\le{0},
            \end{equation}
            or in an equivalent integral form:
            \begin{equation}
                \sup_{0\le{t}\le{T}}\mathscr{E}(t)+\int_0^T\mathscr{P}(t)\mathrm{d}t\le{C},
            \end{equation}
            where $T\in(0,\infty)$ is the final time and $C$ is a generic constant dependent on the initial data and dimensionless numbers. The dimensionless total energy and the energy dissipation density are given by
            \begin{equation}
                \label{true_energy}
                \begin{aligned}
                    \mathscr{E}&=\int_\Omega\left[\frac{1}{2}|\sigma\bm{u}|^2+\frac{1}{\mathrm{We}\mathrm{Cn}}\left(\frac{\mathrm{Cn}^2}{2}|\nabla\phi|^2+F(\phi)\right)\right]\mathrm{d}\bm{x},\\
                    \mathscr{P}&=\int_\Omega\left(\frac{1}{\mathrm{Pe}\mathrm{We}\mathrm{Cn}}m|\nabla\mu|^2+\frac{2}{\mathrm{Re}}\eta|\mathbb{D}(\bm{u})|^2\right)\mathrm{d}\bm{x}.
                \end{aligned}
            \end{equation}

        \subsection{The MSAV reformulation}
            For the convenience of numerical design, we introduce the following scalar auxiliary variables \cite{2022_li}:
            \begin{subequations}
                \begin{align}
                    R(t)=&U(t)=\sqrt{\int_\Omega{G}(\phi)\mathrm{d}\bm{x}+S}, \quad
                    G(\phi)=F(\phi)-\frac{s}{2}\phi^2 , \label{SAV}\\
                    \xi_1=&R/U,\\
                    Q(t)=&e^{-t/T},\\
                    \xi_2=&e^{t/T}Q,
                \end{align}
            \end{subequations}
            where $s,S$ are positive constants. In turn, system \eqref{CHNS} can be reformulated as
            \begin{subequations}
                \label{MCHNS}
                \begin{align}
                    \phi_t+\xi_1\nabla\cdot\left(\bm{u}\phi\right)=&\frac{1}{\mathrm{Pe}}\nabla\cdot\left(m\nabla\mu\right),\label{MCHNS_phi}\\
                    -\mathrm{Cn}^2\Delta\phi+s\phi+\xi_1G'(\phi)=&\mu,\label{MCHNS_mu}\\
                    R_t=&\frac{1}{2U}\int_\Omega{G}'(\phi)\phi_t\mathrm{d}\bm{x},\label{MCHNS_R}\\
                    \sigma(\sigma\bm{u})_t+\xi_2\left[(\rho\bm{u}+\bm{J})\cdot\nabla\bm{u}+\frac{1}{2}\nabla\cdot\left(\rho\bm{u}+\bm{J}\right)\bm{u}\right]=&\frac{1}{\mathrm{Re}}\nabla\cdot\left(2\eta\mathbb{D}(\bm{u})\right)-\nabla{P}-\frac{\xi_1}{\mathrm{We}\mathrm{Cn}}\phi\nabla\mu+\frac{1}{\mathrm{Fr}}\bm{f},\label{MCHNS_u}\\
                    \nabla\cdot\bm{u}=&0,\label{MCHNS_rho}\\
                    Q_t=&-\frac{1}{T}Q.\label{MCHNS_Q}
                \end{align}
            \end{subequations}
            \begin{remark}
                In fact, we have introduced the term $\frac{s}{2}\phi^2$ to simplify the analysis \cite{2022_li, 2018_shen_1}, which will be explored in our future work.
            \end{remark}

            It is obvious that, the reformulated system \eqref{MCHNS}, together with the boundary and initial conditions \eqref{bc}, \eqref{ic}, supplemented with two extra initial conditions for $R(t)$ and $Q(t),$ i.e.,
            \begin{equation}
                \begin{aligned}
                    R(0)=&\sqrt{\int_\Omega{G}(\phi^0)\mathrm{d}\bm{x}+S},\\
                    Q(0)=&1,
                \end{aligned}
            \end{equation}
            is equivalent to the original PDE system.

            Multiplying \eqref{MCHNS_R} by $2R$ and \eqref{MCHNS_Q} by $Q,$ we can easily obtain an equivalent energy dissipation law:
            \begin{equation}
                \frac{\mathrm{d}}{\mathrm{d}t}\mathscr{E}(t)=-\mathscr{P}(t)\le{0},
            \end{equation}
            or in an equivalent integral form:
            \begin{equation}
                \sup_{0\le{t}\le{T}}\mathscr{E}(t)+\int_0^T\mathscr{P}(t)\mathrm{d}t\le{C},
            \end{equation}
            where
            \begin{equation}
                \begin{aligned}
                    \mathscr{E}&=\frac{1}{2}\int_\Omega\left(|\sigma\bm{u}|^2+\frac{\mathrm{Cn}}{\mathrm{We}}|\nabla\phi|^2+\frac{s}{\mathrm{We}\mathrm{Cn}}|\phi|^2\right)\mathrm{d}\bm{x}+\frac{1}{\mathrm{We}\mathrm{Cn}}R^2+\frac{1}{2}Q^2,\\
                    \mathscr{P}&=\int_\Omega\left(\frac{1}{\mathrm{Pe}\mathrm{We}\mathrm{Cn}}m|\nabla\mu|^2+\frac{2}{\mathrm{Re}}\eta|\mathbb{D}(\bm{u})|^2\right)\mathrm{d}\bm{x}+\frac{Q^2}{T}.
                \end{aligned}
            \end{equation}

            In the subsequent section, we will present fully discrete finite element schemes for system \eqref{MCHNS} with provable unconditional energy stability (in terms of a modified energy) for both schemes. Some related works of SAV-type numerical methods that preserve the original energy have been reported in  \cite{2022_jiang, 2022_zhang}, while this effort has been beyond the scope of this paper, and we only focus on the modified energy stability in this work. For simplicity, the mobility $m$ is set to be a constant, while a degenerate mobility (such as $m=(\phi^2-1)^2$) is also applicable. In particular, we will investigate their impact on interface dynamics through several numerical simulations in Section \ref{section_5}, since a degenerate mobility could reduce the non-physical effect of bulk diffusion in two-phase flows, as mentioned in \cite{2021_fu}.

    \section{Fully discrete schemes and their properties}
    \label{section_3}
        Let $H^k(\Omega)$ be the Sobolev spaces equipped with inner product $(\cdot,\cdot)_k$ and norm $\|\cdot\|_k$; of course, it becomes the Lebesgue space $L^2(\Omega)$ with inner product $(\cdot,\cdot)$ and norm $\|\cdot\|$ if $k=0$. Then we introduce the following spaces:
        \begin{equation}
            \begin{aligned}
                V:=&H^1(\Omega),\\
                \bm{X}:=&H_0^1(\Omega)^d,\\
                N:=&L^2(\Omega)\cap{H}^1(\Omega).
            \end{aligned}
        \end{equation}
        In the spatial discretization, let $\mathcal{J}_h$ be a quasi-uniform triangulation of $\Omega$ with mesh size $h,$ and $V_h\subset{V},\bm{X}_h\subset\bm{X},N_h\subset{N},M_h=N_h\cap{L}_0^2(\Omega)$ be the finite element spaces. It is assumed that the pair $\bm{X}_h\times{N}_h$ satisfies the corresponding LBB condition \cite{2021_Ern_1}: there exists a positive constant $C$ which is independent of $h,$ such that
        \begin{equation}
            \inf_{q\in{N}_h}\sup_{\bm{v}\in{\bm{X}}_h}\frac{(\nabla\cdot\bm{v},q)}{\|\bm{v}\|_1\|q\|}\ge{C}.
        \end{equation}
        In terms of the temporal discretization, let $t_n,n=0,1,\cdots,N$ be a uniform partition of the time interval $[0,T]$ with time step size $\tau=\frac{T}{N},$ and the following BDF operator is defined:
        \begin{equation}
            \delta_\tau\chi^{n+1}=\frac{\gamma_0\chi^{n+1}-\hat{\chi}}{\tau},
        \end{equation}
        where $\chi$ is a generic variable and
        \begin{equation}
            \hat{\chi}=
            \begin{cases}
                \chi^n    ,                     & J=1,\\
                2\chi^n-\frac{1}{2}\chi^{n-1} , & J=2,
            \end{cases}
            ; \quad \gamma_0=
            \begin{cases}
                1   , & J=1,\\
                3/2 , & J=2.
            \end{cases}
        \end{equation}
        We also recall the extrapolation formulas
        \begin{equation}
            \widetilde{\chi}^{n+1}=
            \begin{cases}
                \chi^n     ,         & J=1,\\
                2\chi^n-\chi^{n-1} , & J=2,
            \end{cases}
        \end{equation}
        and use the following notations in the numerical design:
        \begin{equation}
            \begin{aligned}
                \mathscr{D}\chi^{n+1}&=\chi^{n+1}-\chi^n,\\
                \mathscr{D}^2\chi^{n+1}&=\chi^{n+1}-2\chi^n+\chi^{n-1},\\
                \mathscr{D}^3\chi^{n+1}&=\chi^{n+1}-3\chi^n+3\chi^{n-1}-\chi^{n-2}.
            \end{aligned}
        \end{equation}

        \subsection{The standard Galerkin scheme}
            The Galerkin finite element method can be directly applied to system \eqref{MCHNS} for spatial discretization, i.e., solving velocity and pressure in a saddle point formulation. Its combination with the BDF temporal discretization results in the following fully discrete finite element scheme: find $(\phi_h^{n+1},\mu_h^{n+1})\in{V}_h\times{V}_h,(\bm{u}_h^{n+1},P_h^{n+1})\in\bm{X}_h\times{M}_h,R^{n+1}\in\mathbb{R}^+,Q^{n+1}\in\mathbb{R}^+,$ such that for all $(\psi_h,\omega_h)\in{V}_h\times{V}_h,(\bm{v}_h,q_h)\in\bm{X}_h\times{M}_h,$
            \begin{subequations}
                \label{MCHNSt_1}
                \begin{align}
                    (\delta_\tau\phi_h^{n+1},\omega_h)=&\xi_1^{n+1}(\widetilde{\phi}_h^{n+1}\widetilde{\bm{u}}_h^{n+1},\nabla\omega_h)-\frac{1}{\mathrm{Pe}}(\nabla\mu_h^{n+1},\nabla\omega_h),\label{CHNS_phit_1}\\
                    (\mu_h^{n+1},\psi_h)=&\mathrm{Cn}^2(\nabla\phi_h^{n+1},\nabla\psi_h)+s(\phi_h^{n+1},\psi_h)+\xi_1^{n+1}(G'(\widetilde{\phi}_h^{n+1}),\psi_h),\label{CHNS_mut_1}\\
                    \delta_\tau{R}^{n+1}=&\frac{1}{2\widetilde{U}^{n+1}}\left((G'(\widetilde{\phi}_h^{n+1}),\delta_\tau\phi_h^{n+1})-(\widetilde{\phi}_h^{n+1}\widetilde{\bm{u}}_h^{n+1},\nabla\mu_h^{n+1})+(\widetilde{\phi}_h^{n+1}\nabla\widetilde{\mu}_h^{n+1},\bm{u}_h^{n+1})\right),\label{CHNS_Rt_1}\\
                    (\widetilde{\sigma}_h^{n+1}\delta_\tau(\widetilde{\sigma}\bm{u})_h^{n+1}&,\bm{v}_h)+\frac{2}{\mathrm{Re}}\left(\widetilde{\eta}_h^{n+1}\mathbb{D}(\bm{u}_h^{n+1}),\mathbb{D}(\bm{v}_h)\right)-(P_h^{n+1},\nabla\cdot\bm{v}_h)+\left(\nabla\cdot\bm{u}_h^{n+1},q_h\right)\notag\\
                    =&-\frac{\xi_2^{n+1}}{2}\left[\left((\widetilde{\rho}_h^{n+1}\widetilde{\bm{u}}_h^{n+1}+\widetilde{\bm{J}}_h^{n+1})\cdot\nabla\widetilde{\bm{u}}_h^{n+1},\bm{v}_h\right)-\left((\widetilde{\rho}_h^{n+1}\widetilde{\bm{u}}_h^{n+1}+\widetilde{\bm{J}}_h^{n+1})\cdot\nabla\bm{v}_h,\widetilde{\bm{u}}_h^{n+1}\right)\right]\notag\\
                    &-\frac{\xi_1^{n+1}}{\mathrm{We}\mathrm{Cn}}(\widetilde{\phi}_h^{n+1}\nabla\widetilde{\mu}_h^{n+1},\bm{v}_h)+\frac{1}{\mathrm{Fr}}(\widetilde{\bm{f}}_h^{n+1},\bm{v}_h),\label{CHNS_rhout_1}\\
                    \delta_\tau{Q}^{n+1}=&-\frac{Q^{n+1}}{T}+\frac{1}{2}e^{t^{n+1}/T}\left[\left((\widetilde{\rho}_h^{n+1}\widetilde{\bm{u}}_h^{n+1}+\widetilde{\bm{J}}_h^{n+1})\cdot\nabla\widetilde{\bm{u}}_h^{n+1},\bm{u}_h^{n+1}\right)\right.\notag\\
                    &\left.-\left((\widetilde{\rho}_h^{n+1}\widetilde{\bm{u}}_h^{n+1}+\widetilde{\bm{J}}_h^{n+1})\cdot\nabla\bm{u}_h^{n+1},\widetilde{\bm{u}}_h^{n+1}\right)\right].\label{CHNS_Qt_1}
                \end{align}
            \end{subequations}
            In the above equations, we denote $\widetilde{\chi}^{n+1}=\chi(\widetilde{\phi}_h^{n+1})$ for simplicity.
            \begin{remark}
                In fact, the ZEC approach has been utilized in the construction of \eqref{CHNS_Rt_1} and \eqref{CHNS_Qt_1}, to ensure the energy stability property.
            \end{remark}
            \begin{remark}
                Although $\phi$ is bounded in the continuous PDE analysis, the discrete solution $\phi_h$ may not obey this bound. The dynamics of $\phi$ would not be influenced by this deviation, while the strict positivity of the associated physical parameters may be destroyed, which is especially pivotal here since $\sigma=\sqrt{\rho}$ is introduced. To avoid such a numerical error caused by an excursion of $\phi_h$, the cut-off approach is applied:
                \begin{equation}
                    \phi_h^\mathscr{C}=
                    \begin{cases}
                        \phi_h,                & \text{if} \, \, \, |\phi_h|\le{1}, \\
                        \mathrm{sign}(\phi_h), & \text{otherwise},
                    \end{cases}
                \end{equation}
                and $\rho_h,\eta_h$ could be evaluated by
                \begin{equation}
                    \rho_h=\frac{\rho_1-\rho_2}{2\rho_r}\phi_h^\mathscr{C}+\frac{\rho_1+\rho_2}{2\rho_r}\quad\text{and}\quad\eta_h:=\frac{\eta_1-\eta_2}{2\eta_r}\phi_h^\mathscr{C}+\frac{\eta_1+\eta_2}{2\eta_r} , 
                \end{equation}
                with suitable time discretization for $\phi_h$. In practice, we first use extrapolations to get $\widetilde{\phi}_h^{n+1}$, then the cut-off variables, $\widetilde{\rho}_h^{n+1}$ and $\widetilde{\eta}_h^{n+1}$,  can be updated afterwards. Via this method, the bound of $\rho$ and $\eta$ will be inherited by discrete solutions. In addition, the theoretical proof of discrete energy law will not be affected, since it is employed only to make $\widetilde{\rho}_h^{n+1}$ and $\widetilde{\eta}_h^{n+1}$ bounded in the bulk by $\rho_k$ and $\eta_k$, $k=1, 2$. 
            \end{remark}

            Now we look at the numerical implementation process. With the notations
            \begin{equation}
                \label{update}
                \begin{cases}
                    \phi_h^{n+1}=\phi_0^{n+1}+\xi_1^{n+1}\phi_1^{n+1},\\
                    \mu_h^{n+1}=\mu_0^{n+1}+\xi_1^{n+1}\mu_1^{n+1},\\
                    \bm{u}_h^{n+1}=\bm{u}_0^{n+1}+\xi_1^{n+1}\bm{u}_1^{n+1}+\xi_2^{n+1}\bm{u}_2^{n+1},\\
                    P_h^{n+1}=P_0^{n+1}+\xi_1^{n+1}P_1^{n+1}+\xi_2^{n+1}P_2^{n+1},
                \end{cases}
            \end{equation}
            we are able to decompose \eqref{MCHNSt_1} into the following steps.

            \textit{Step} 1:
            Find $(\phi_0^{n+1},\mu_0^{n+1})\in{V}_h\times{V}_h,$ such that for all $(\psi_h,\omega_h)\in{V}_h\times{V}_h,$
            \begin{equation}
                \label{phase_0}
                \begin{aligned}
                    \gamma_0(\phi_0^{n+1},\omega_h)+\frac{\tau}{\mathrm{Pe}}(\nabla\mu_0^{n+1},\nabla\omega_h)=&(\hat{\phi},\omega_h),\\
                    (\mu_0^{n+1},\psi_h)-\mathrm{Cn}^2(\nabla\phi_0^{n+1},\nabla\psi_h)-s(\phi_0^{n+1},\psi_h)=&0.
                \end{aligned}
            \end{equation}
            In addition, we solve for $(\phi_1^{n+1},\mu_1^{n+1})\in{V}_h\times{V}_h,$ such that for all $(\psi_h,\omega_h)\in{V}_h\times{V}_h,$
            \begin{equation}
                \label{phase_1}
                \begin{aligned}
                    \gamma_0(\phi_1^{n+1},\omega_h)+\frac{\tau}{\mathrm{Pe}}(\nabla\mu_1^{n+1},\nabla\omega_h)=&\tau(\widetilde{\phi}_h^{n+1}\widetilde{\bm{u}}_h^{n+1},\nabla\omega_h),\\
                    (\mu_1^{n+1},\psi_h)-\mathrm{Cn}^2(\nabla\phi_1^{n+1},\nabla\psi_h)-s(\phi_1^{n+1},\psi_h)=&(G'(\widetilde{\phi}_h^{n+1}),\psi_h).
                \end{aligned}
            \end{equation}

            \textit{Step} 2:
            Obtain $(\bm{u}_0^{n+1},P_0^{n+1})\in\bm{X}_h\times{M}_h,$ such that for all $(\bm{v}_h,q_h)\in\bm{X}_h\times{M}_h,$
            \begin{equation}
                \label{flow_u0}
                \begin{aligned}
                    \gamma_0((\widetilde{\rho}_h\bm{u}_0)^{n+1},\bm{v}_h)+\frac{2\tau}{\mathrm{Re}}\left(\widetilde{\eta}_h^{n+1}\mathbb{D}(\bm{u}_0^{n+1}),\mathbb{D}(\bm{v}_h)\right)-&\tau(P_0^{n+1},\nabla\cdot\bm{v}_h)\\
                    +&\tau(\nabla\cdot\bm{u}_0^{n+1},q_h)=(\widetilde{\sigma}_h^{n+1}\widehat{\widetilde{\sigma}\bm{u}},\bm{v}_h)+\frac{\tau}{\mathrm{Fr}}(\widetilde{\bm{f}}_h^{n+1},\bm{v}_h).
                \end{aligned}
            \end{equation}
            Similarly, we get $(\bm{u}_1^{n+1},P_1^{n+1}) , \, (\bm{u}_2^{n+1},P_2^{n+1}) \in\bm{X}_h\times{M}_h,$ such that for all $(\bm{v}_h,q_h)\in\bm{X}_h\times{M}_h$,
            \begin{equation}
                \label{flow_u1}
                \begin{aligned}
                    \gamma_0((\widetilde{\rho}_h\bm{u}_1)^{n+1},\bm{v}_h)+\frac{2\tau}{\mathrm{Re}}\left(\widetilde{\eta}_h^{n+1}\mathbb{D}(\bm{u}_1^{n+1}),\mathbb{D}(\bm{v}_h)\right)-&\tau(P_1^{n+1},\nabla\cdot\bm{v}_h)\\
                    +&\tau(\nabla\cdot\bm{u}_1^{n+1},q_h)=-\frac{\tau}{\mathrm{We}\mathrm{Cn}}(\widetilde{\phi}_h^{n+1}\nabla\widetilde{\mu}_h^{n+1},\bm{v}_h) ,
                \end{aligned}
            \end{equation}
            \begin{equation}
                \label{flow_u2}
                \begin{aligned}
                    \gamma_0((\widetilde{\rho}_h\bm{u}_2)^{n+1},\bm{v}_h)+\frac{2\tau}{\mathrm{Re}}\left(\widetilde{\eta}_h^{n+1}\mathbb{D}(\bm{u}_2^{n+1}),\mathbb{D}(\bm{v}_h)\right)-&\tau(P_2^{n+1},\nabla\cdot\bm{v}_h)+\tau(\nabla\cdot\bm{u}_2^{n+1},q_h)\\
                    =&-\frac{\tau}{2}\left[\left((\widetilde{\rho}_h^{n+1}\widetilde{\bm{u}}_h^{n+1}+\widetilde{\bm{J}}_h^{n+1})\cdot\nabla\widetilde{\bm{u}}_h^{n+1},\bm{v}_h\right)\right.\\
                    &\left.-\left((\widetilde{\rho}_h^{n+1}\widetilde{\bm{u}}_h^{n+1}+\widetilde{\bm{J}}_h^{n+1})\cdot\nabla\bm{v}_h,\widetilde{\bm{u}}_h^{n+1}\right)\right].
                \end{aligned}
            \end{equation}

            \textit{Step} 3:
            Once $\phi_i^{n+1},\mu_i^{n+1},\bm{u}_i^{n+1}$ are obtained, we are able to determine $\xi_i^{n+1}$ by
            \begin{equation}
                \label{CHNS_xi}
                \begin{aligned}
                    \begin{pmatrix}
                        A_1 & A_2 \\
                        B_1 & B_2
                    \end{pmatrix}
                    \begin{pmatrix}
                        \xi_1^{n+1} \\
                        \xi_2^{n+1}
                    \end{pmatrix}
                    =
                    \begin{pmatrix}
                        A_0 \\
                        B_0
                    \end{pmatrix}
                    ,
                \end{aligned}
            \end{equation}
            where
            \begin{equation}
                \left\{
                \begin{aligned}
                    A_0=&\hat{R}+\frac{1}{2\widetilde{U}^{n+1}}\left((G'(\widetilde{\phi}_h^{n+1}),\gamma_0\phi_0^{n+1}-\hat{\phi})-\tau(\widetilde{\phi}_h^{n+1}\widetilde{\bm{u}}_h^{n+1},\nabla\mu_0^{n+1})+\tau(\widetilde{\phi}_h^{n+1}\nabla\widetilde{\mu}_h^{n+1},\bm{u}_0^{n+1})\right),\\
                    A_1=&\gamma_0\widetilde{U}^{n+1}-\frac{1}{2\widetilde{U}^{n+1}}\left((G'(\widetilde{\phi}_h^{n+1}),\gamma_0\phi_1^{n+1})-\tau(\widetilde{\phi}_h^{n+1}\widetilde{\bm{u}}_h^{n+1},\nabla\mu_1^{n+1})+\tau(\widetilde{\phi}_h^{n+1}\nabla\widetilde{\mu}_h^{n+1},\bm{u}_1^{n+1})\right),\\
                    A_2=&-\frac{\tau}{2\widetilde{U}^{n+1}}(\widetilde{\phi}_h^{n+1}\nabla\widetilde{\mu}_h^{n+1},\bm{u}_2^{n+1}),\\
                    B_0=&\hat{Q}+\frac{\tau}{2}{e}^{t^{n+1}/T}\left[\left((\widetilde{\rho}_h^{n+1}\widetilde{\bm{u}}_h^{n+1}+\widetilde{\bm{J}}_h^{n+1})\cdot\nabla\widetilde{\bm{u}}_h^{n+1},\bm{u}_0^{n+1}\right)-\left((\widetilde{\rho}_h^{n+1}\widetilde{\bm{u}}_h^{n+1}+\widetilde{\bm{J}}_h^{n+1})\cdot\nabla\bm{u}_0^{n+1},\widetilde{\bm{u}}_h^{n+1}\right)\right],\\
                    B_1=&-\frac{\tau}{2}{e}^{t^{n+1}/T}\left[\left((\widetilde{\rho}_h^{n+1}\widetilde{\bm{u}}_h^{n+1}+\widetilde{\bm{J}}_h^{n+1})\cdot\nabla\widetilde{\bm{u}}_h^{n+1},\bm{u}_1^{n+1}\right)-\left((\rho_h^{n+1}\widetilde{\bm{u}}_h^{n+1}+\widetilde{\bm{J}}_h^{n+1})\cdot\nabla\bm{u}_1^{n+1},\widetilde{\bm{u}}_h^{n+1}\right)\right],\\
                    B_2=&e^{-t^{n+1}/T}\left(\gamma_0+\frac{\tau}{T}\right)-\frac{\tau}{2}{e}^{t^{n+1}/T}\left[\left((\widetilde{\rho}_h^{n+1}\widetilde{\bm{u}}_h^{n+1}+\widetilde{\bm{J}}_h^{n+1})\cdot\nabla\widetilde{\bm{u}}_h^{n+1},\bm{u}_2^{n+1}\right)\right.\\
                    &\left.-\left((\widetilde{\rho}_h^{n+1}\widetilde{\bm{u}}_h^{n+1}+\widetilde{\bm{J}}_h^{n+1})\cdot\nabla\bm{u}_2^{n+1},\widetilde{\bm{u}}_h^{n+1}\right)\right].
                \end{aligned}
                \right.
            \end{equation}

            In practice, all coefficient matrices and load vectors could be computed in a single element traversal at each time step. Meanwhile, we need to solve two mixed systems and three saddle point systems with the same coefficient matrix, and one linear algebraic system of two equations. Although this scheme requires additional computations for load vectors and finite element solutions, the advantage of the decoupled nature, ease of implementation, and unconditional energy stability makes it very attractive. Moreover, the systems in Steps 1 and 2 could be solved in parallel since all of them are independent from each other.

            Now we look at the well-posedness and energy stability of \eqref{MCHNSt_1} in the following theorems.
            \begin{theorem}
                The numerical scheme \eqref{MCHNSt_1} is globally mass conservative:
                \begin{equation}
                    \int_\Omega\phi_h^{n+1}\mathrm{d}\bm{x}=\int_\Omega\phi_h^n\mathrm{d}\bm{x}.
                \end{equation}
            \end{theorem}
            \begin{proof}
                This identity could be easily derived by taking $\omega_h=1$ in \eqref{CHNS_phit_1}.
            \end{proof}
            \begin{theorem}
                The numerical scheme \eqref{MCHNSt_1} admits a unique solution.
            \end{theorem}
            \begin{proof}
                The well-posedness of \eqref{phase_0}-\eqref{phase_1} in Step 1 cannot be directly derived from Lax-Milgram theorem \cite{2021_Ern_1}, since there is no guarantee for the $H^1$ stability of $\mu$  (the bilinear form only involves $\nabla\mu$). In fact, they are well-posed in $\bar{H}^1(\Omega)\times\bar{H}^1(\Omega)$, where $\bar{H}^1(\Omega):=\left\{\varphi\in{H}^1(\Omega):\int_\Omega\varphi=0\right\}$. By defining $\hat{\phi}=\phi-\int_\Omega\phi,\hat{\mu}=\mu-\int_\Omega\mu$, we are able to recover $(\phi,\mu)$ form $(\hat{\phi},\hat{\mu})$, which implies the well-posedness in $H^1(\Omega)\times{H}^1(\Omega)$; more technical details could be seen in~\cite{2020_yang}. Since $\bm{X}_h\times{M}_h$ satisfies the LBB condition, by Babu\v{s}ka-Brezzi theorem \cite{2021_Ern_1}, we get the unique solvability of \eqref{flow_u0}-\eqref{flow_u2} in Step 2.

                With regard to the linear system \eqref{CHNS_xi} in Step 3, we now prove that $A_2B_1\neq{A}_1B_2$. Taking $(\psi_h,\omega_h)=(\gamma_0\phi_1^{n+1},\mu_1^{n+1})$ in \eqref{phase_1}, the following equalities become available:
                \begin{equation}
                    \begin{aligned}
                        \gamma_0(\phi_1^{n+1},\mu_1^{n+1})+\frac{\tau}{\mathrm{Pe}}\|\nabla\mu_1^{n+1}\|^2=&\tau(\widetilde{\phi}_h^{n+1}\widetilde{\bm{u}}_h^{n+1},\nabla\mu_1^{n+1}),\\
                        \gamma_0(\mu_1^{n+1},\phi_1^{n+1})-\gamma_0\mathrm{Cn}^2\|\nabla\phi_1^{n+1}\|^2-\gamma_0s\|\phi_1^{n+1}\|^2=&(G'(\widetilde{\phi}_h^{n+1}),\gamma_0\phi_1^{n+1}) .
                    \end{aligned}
                \end{equation}
                Taking $(\bm{v}_h,q_h)=(\bm{u}_1^{n+1},P_1^{n+1}),(\bm{u}_2^{n+1},0)$ and $(\bm{0},P_2^{n+1})$ in \eqref{flow_u1}, respectively, we see that
                \begin{equation}
                    \begin{aligned}
                        \gamma_0\|(\widetilde{\sigma}_h\bm{u}_1)^{n+1}\|^2+\frac{2\tau}{\mathrm{Re}}\|(\widetilde{\eta}_h^{n+1})^\frac{1}{2}\mathbb{D}(\bm{u}_1^{n+1})\|^2=&-\frac{\tau}{\mathrm{We}\mathrm{Cn}}(\widetilde{\phi}_h^{n+1}\nabla\widetilde{\mu}_h^{n+1},\bm{u}_1^{n+1}),\\
                        \gamma_0((\widetilde{\sigma}_h\bm{u}_1)^{n+1},(\widetilde{\sigma}_h\bm{u}_2)^{n+1})+\frac{2\tau}{\mathrm{Re}}\left(\widetilde{\eta}^{n+1}\mathbb{D}(\bm{u}_1^{n+1}),\mathbb{D}(\bm{u}_2^{n+1})\right)-&\tau(P_1^{n+1},\nabla\cdot\bm{u}_2^{n+1})\\
                        =&-\frac{\tau}{\mathrm{We}\mathrm{Cn}}(\widetilde{\phi}_h^{n+1}\nabla\widetilde{\mu}_h^{n+1},\bm{u}_2^{n+1}),\\
                        \tau(\nabla\cdot\bm{u}_1^{n+1},P_2^{n+1})=&0.\\
                    \end{aligned}
                \end{equation}
                Taking $(\bm{v}_h,q_h)=(\bm{u}_1^{n+1},0),(\bm{0},P_1^{n+1})$ and $(\bm{u}_2^{n+1},P_2^{n+1})$ in \eqref{flow_u2}, respectively, we obtain
                \begin{equation}
                    \begin{aligned}
                        \gamma_0((\widetilde{\sigma}_h\bm{u}_2)^{n+1},(\widetilde{\sigma}_h\bm{u}_1)^{n+1})+\frac{2\tau}{\mathrm{Re}}&\left(\widetilde{\eta}_h^{n+1}\mathbb{D}(\bm{u}_2^{n+1}),\mathbb{D}(\bm{u}_1^{n+1})\right)-\tau(P_2^{n+1},\nabla\cdot\bm{u}_1^{n+1})\\
                        =&-\frac{\tau}{2}\left[\left((\widetilde{\rho}_h^{n+1}\widetilde{\bm{u}}_h^{n+1}+\widetilde{\bm{J}}_h^{n+1})\cdot\nabla\widetilde{\bm{u}}_h^{n+1},\bm{u}_1^{n+1}\right)\right.\\
                        &\left.-\left((\widetilde{\rho}_h^{n+1}\widetilde{\bm{u}}_h^{n+1}+\widetilde{\bm{J}}_h^{n+1})\cdot\nabla\bm{u}_1^{n+1},\widetilde{\bm{u}}_h^{n+1}\right)\right],\\
                        \tau(\nabla\cdot\bm{u}_2^{n+1},P_1^{n+1})=&0,\\
                        \gamma_0\|(\widetilde{\sigma}_h\bm{u}_2)^{n+1}\|^2+\frac{2\tau}{\mathrm{Re}}\|(\widetilde{\eta}_h^{n+1})^\frac{1}{2}\mathbb{D}(\bm{u}_2^{n+1})\|^2=&-\frac{\tau}{2}\left[\left((\widetilde{\rho}_h^{n+1}\widetilde{\bm{u}}_h^{n+1}+\widetilde{\bm{J}}_h^{n+1})\cdot\nabla\widetilde{\bm{u}}_h^{n+1},\bm{u}_2^{n+1}\right)\right.\\
                        &\left.-\left((\widetilde{\rho}_h^{n+1}\widetilde{\bm{u}}_h^{n+1}+\widetilde{\bm{J}}_h^{n+1})\cdot\nabla\bm{u}_2^{n+1},\widetilde{\bm{u}}_h^{n+1}\right)\right].
                    \end{aligned}
                \end{equation}
                As a result, $A_1,A_2,B_1,B_2$ could be reformulated as
                \begin{equation}
                    \begin{aligned}
                        A_1=&\gamma_0\widetilde{U}^{n+1}+\frac{1}{2\widetilde{U}^{n+1}}\left(\gamma_0\mathrm{Cn}^2\|\nabla\phi_1^{n+1}\|^2+\gamma_0s\|\phi_1^{n+1}\|^2+\frac{\tau}{\mathrm{Pe}}\|\nabla\mu_1^{n+1}\|^2\right.\\
                        &\left.+\gamma_0\mathrm{We}\mathrm{Cn}\|(\widetilde{\sigma}_h\bm{u}_1)^{n+1}\|^2+\frac{2\tau\mathrm{We}\mathrm{Cn}}{\mathrm{Re}}\|(\widetilde{\eta}_h^{n+1})^\frac{1}{2}\mathbb{D}(\bm{u}_1^{n+1})\|^2\right),\\
                        A_2=&\frac{\mathrm{We}\mathrm{Cn}}{2\widetilde{U}^{n+1}}\left[\gamma_0((\widetilde{\sigma}_h\bm{u}_1)^{n+1},(\widetilde{\sigma}_h\bm{u}_2)^{n+1})+\frac{2\tau}{\mathrm{Re}}\left(\widetilde{\eta}_h^{n+1}\mathbb{D}(\bm{u}_1^{n+1}),\mathbb{D}(\bm{u}_2^{n+1})\right)\right],\\
                        B_1=&e^{t^{n+1}/T}\left[\gamma_0((\widetilde{\sigma}_h\bm{u}_2)^{n+1},(\widetilde{\sigma}_h\bm{u}_1)^{n+1})+\frac{2\tau}{\mathrm{Re}}\left(\widetilde{\eta}_h^{n+1}\mathbb{D}(\bm{u}_2^{n+1}),\mathbb{D}(\bm{u}_1^{n+1})\right)\right],\\
                        B_2=&e^{-t^{n+1}/T}\left(\gamma_0+\frac{\tau}{T}\right)+e^{t^{n+1}/T}\left(\gamma_0\|(\widetilde{\sigma}_h\bm{u}_2)^{n+1}\|^2+\frac{2\tau}{\mathrm{Re}}\|(\widetilde{\eta}_h^{n+1})^\frac{1}{2}\mathbb{D}(\bm{u}_2^{n+1})\|^2\right).
                    \end{aligned}
                \end{equation}
                Since $\widetilde{U}^{n+1}>0$, an application of Cauchy inequality gives
                \begin{equation}
                    \begin{aligned}
                    A_2B_1=&\frac{\mathrm{We}\mathrm{Cn}\;e^{t^{n+1}/T}}{2\widetilde{U}^{n+1}}\left(\gamma_0((\widetilde{\sigma}_h\bm{u}_1)^{n+1},(\widetilde{\sigma}_h\bm{u}_2)^{n+1})+\frac{2\tau}{\mathrm{Re}}\left(\widetilde{\eta}_h^{n+1}\mathbb{D}(\bm{u}_1^{n+1}),\mathbb{D}(\bm{u}_2^{n+1})\right)\right)^2\\
                    &\le\frac{\mathrm{We}\mathrm{Cn}\;e^{t^{n+1}/T}}{2\widetilde{U}^{n+1}}\left(\gamma_0\|(\widetilde{\sigma}_h\bm{u}_1)^{n+1}\|^2+\frac{2\tau}{\mathrm{Re}}\|(\widetilde{\eta}_h^{n+1})^\frac{1}{2}\mathbb{D}(\bm{u}_1^{n+1})\|^2\right)\\
                    &\times\left(\gamma_0\|(\widetilde{\sigma}\bm{u}_2)^{n+1}\|^2+\frac{2\tau}{\mathrm{Re}}\|(\widetilde{\eta}_h^{n+1})^\frac{1}{2}\mathbb{D}(\bm{u}_2^{n+1})\|^2\right)\\
                    &<{A}_1B_2,
                \end{aligned}
                \end{equation}
                which implies the desired result.
            \end{proof}
            \begin{remark}
                In the case of a degenerate mobility, such as $m=\left(\phi^2-1\right)^2$, the uniqueness of $(\phi_h,\mu_h)$ for \eqref{phase_0}-\eqref{phase_1} (in Step 1) turns out to be a direct consequence of the discrete energy law. Subsequently, by employing the Fredholm alternative, we obtain the unique solvability of the discrete system in Step 1.
            \end{remark}
            \begin{theorem}
                \label{discrete_energy_estimate_1}
                In the absence of $\widetilde{f}_h^{n+1},$ the scheme \eqref{MCHNSt_1} satisfies the following modified discrete energy law:
                \begin{equation}
                    \mathscr{E}^{n+1}-\mathscr{E}^n=-\mathscr{Q}^{n+1}-\frac{4\tau}{\mathrm{Re}}\|(\widetilde{\eta}_h^{n+1})^\frac{1}{2}\mathbb{D}(\bm{u}_h^{n+1})\|^2-\frac{2\tau}{T}|{Q}^{n+1}|^2-\frac{2\tau}{\mathrm{Pe}\mathrm{We}\mathrm{Cn}}\|\nabla\mu_h^{n+1}\|^2,
                \end{equation}
                where
                \begin{equation}
                    \label{discrete_energy_1}
                    \mathscr{E}^n=
                    \begin{cases}
                        \|(\widetilde{\sigma}\bm{u})_h^n\|^2+\frac{\mathrm{Cn}}{\mathrm{We}}\|\nabla\phi_h^n\|^2+\frac{s}{\mathrm{We}\mathrm{Cn}}\|\phi_h^n\|^2+\frac{2}{\mathrm{We}\mathrm{Cn}}|R^n|^2+|Q^n|^2 , & J=1,\\
                        \frac{1}{2}\|(\widetilde{\sigma}\bm{u})_h^n\|^2+\frac{1}{2}\|\widetilde{(\widetilde{\sigma}\bm{u})}_h^{n+1}\|^2+\frac{\mathrm{Cn}}{2\mathrm{We}}\|\nabla\phi_h^n\|^2+\frac{\mathrm{Cn}}{2\mathrm{We}}\|\nabla\widetilde{\phi}_h^{n+1}\|^2+\frac{s}{2\mathrm{We}\mathrm{Cn}}\|\phi_h^n\|^2\\
                        +\frac{s}{2\mathrm{We}\mathrm{Cn}}\|\widetilde{\phi}_h^{n+1}\|^2+\frac{1}{\mathrm{We}\mathrm{Cn}}|R^n|^2+\frac{1}{\mathrm{We}\mathrm{Cn}}|\widetilde{R}^{n+1}|^2+\frac{1}{2}|Q^n|^2+\frac{1}{2}|\widetilde{Q}^{n+1}|^2 , & J=2,\\
                    \end{cases}
                \end{equation}
                and
                \begin{equation}
                    \mathscr{Q}^n=
                    \begin{cases}
                        \|\mathscr{D}(\widetilde{\sigma}\bm{u})_h^n\|^2+\frac{s}{\mathrm{We}\mathrm{Cn}}\|\mathscr{D}\phi_h^n\|^2+\frac{\mathrm{Cn}}{\mathrm{We}}\|\nabla\mathscr{D}\phi_h^n\|^2+\frac{2}{\mathrm{We}\mathrm{Cn}}|\mathscr{D}R^n|^2+|\mathscr{D}Q^n|^2  ,                                  & J=1, \\
                        \frac{1}{2}\|\mathscr{D}^2(\widetilde{\sigma}\bm{u})_h^n\|^2+\frac{\mathrm{Cn}}{2\mathrm{We}}\|\nabla\mathscr{D}^2\phi_h^n\|^2+\frac{s}{2\mathrm{We}\mathrm{Cn}}\|\mathscr{D}^2\phi_h^n\|^2+\frac{1}{\mathrm{We}\mathrm{Cn}}|\mathscr{D}^2R^n|^2+\frac{1}{2}|\mathscr{D}^2Q^n|^2 , & J=2.
                    \end{cases}
                \end{equation}
            \end{theorem}
            \begin{proof}
                Taking $(\psi_h,\omega_h)=(2\tau\delta_\tau\phi_h^{n+1},2\tau\mu_h^{n+1})$ in \eqref{CHNS_phit_1}-\eqref{CHNS_mut_1} and multiplying \eqref{CHNS_Rt_1} by $4\tau{R}^{n+1}$, we get
                \begin{equation}
                    \label{discrete_energy_phase_1}
                    \begin{aligned}
                        (\delta_\tau\phi_h^{n+1},2\tau\mu_h^{n+1})=&\xi_1^{n+1}(\widetilde{\phi}_h^{n+1}\widetilde{\bm{u}}_h^{n+1},2\tau\nabla\mu_h^{n+1})-\frac{2\tau}{\mathrm{Pe}}\|\nabla\mu_h^{n+1}\|^2,\\
                        (\mu_h^{n+1},2\tau\delta_\tau\phi_h^{n+1})=&\mathrm{Cn}^2(\nabla\phi_h^{n+1},2\tau\delta_\tau\nabla\phi_h^{n+1})+s(\phi_h^{n+1},2\tau\delta_\tau\phi_h^{n+1})+\xi_1^{n+1}(G'(\widetilde{\phi}_h^{n+1}),2\tau\delta_\tau\phi_h^{n+1}),\\
                        \delta_\tau{R}^{n+1}4\tau{R}^{n+1}=&\xi_1^{n+1}\left((G'(\widetilde{\phi}_h^{n+1}),2\tau\delta_\tau\phi_h^{n+1})-(\widetilde{\phi}_h^{n+1}\widetilde{\bm{u}}_h^{n+1},2\tau\nabla\mu_h^{n+1})+(\widetilde{\phi}_h^{n+1}\nabla\widetilde{\mu}_h^{n+1},2\tau\bm{u}_h^{n+1})\right).
                    \end{aligned}
                \end{equation}
                Taking $(\bm{v}_h,q_h)=(2\tau\bm{u}_h^{n+1},2\tau{P}_h^{n+1})$ in \eqref{CHNS_rhout_1} and multiplying \eqref{CHNS_Qt_1} by $2\tau{Q}^{n+1}$, we see that
                \begin{equation}
                    \label{discrete_energy_flow_1}
                    \begin{aligned}
                        \left(\delta_\tau(\widetilde{\sigma}\bm{u})_h^{n+1},2\tau(\widetilde{\sigma}\bm{u})_h^{n+1}\right)+\frac{4\tau}{\mathrm{Re}}\|(\widetilde{\eta}_h^{n+1})^\frac{1}{2}\mathbb{D}(\bm{u}_h^{n+1})\|^2
                        =&-\frac{\xi_2^{n+1}}{2}\left[\left((\widetilde{\rho}_h^{n+1}\widetilde{\bm{u}}_h^{n+1}+\widetilde{\bm{J}}_h^{n+1})\cdot\nabla\widetilde{\bm{u}}_h^{n+1},2\tau\bm{u}_h^{n+1}\right)\right.\\
                        &\left.-\left(2\tau(\widetilde{\rho}_h^{n+1}\widetilde{\bm{u}}_h^{n+1}+\widetilde{\bm{J}}_h^{n+1})\cdot\nabla\bm{u}_h^{n+1},\widetilde{\bm{u}}_h^{n+1}\right)\right]\\
                        &-\frac{\xi_1^{n+1}}{\mathrm{We}\mathrm{Cn}}(\widetilde{\phi}_h^{n+1}\nabla\widetilde{\mu}_h^{n+1},2\tau\bm{u}_h^{n+1}),\\
                        \delta_\tau{Q}^{n+1}2\tau{Q}^{n+1}+\frac{2\tau}{T}|{Q}^{n+1}|^2=&\frac{\xi_2^{n+1}}{2}\left[\left((\widetilde{\rho}_h^{n+1}\widetilde{\bm{u}}_h^{n+1}+\widetilde{\bm{J}}_h^{n+1})\cdot\nabla\widetilde{\bm{u}}_h^{n+1},2\tau\bm{u}_h^{n+1}\right)\right.\\
                        &\left.-\left(2\tau(\widetilde{\rho}_h^{n+1}\widetilde{\bm{u}}_h^{n+1}+\widetilde{\bm{J}}_h^{n+1})\cdot\nabla\bm{u}_h^{n+1},\widetilde{\bm{u}}_h^{n+1}\right)\right].
                    \end{aligned}
                \end{equation}
                A summation of \eqref{discrete_energy_phase_1} and \eqref{discrete_energy_flow_1} gives
                \begin{equation}
                    \label{discrete_energy_mid_1}
                    \begin{aligned}
                        \frac{s}{\mathrm{We}\mathrm{Cn}}&(\phi_h^{n+1},2\tau\delta_\tau\phi_h^{n+1})+\frac{\mathrm{Cn}}{\mathrm{We}}(\nabla\phi_h^{n+1},2\tau\delta_\tau\nabla\phi_h^{n+1})+\frac{1}{\mathrm{We}\mathrm{Cn}}\delta_\tau{R}^{n+1}4\tau{R}^{n+1}+\left(\delta_\tau(\widetilde{\sigma}\bm{u})_h^{n+1},2\tau(\widetilde{\sigma}\bm{u})_h^{n+1}\right)\\
                        &+\delta_\tau{Q}^{n+1}2\tau{Q}^{n+1}+\frac{4\tau}{\mathrm{Re}}\|(\widetilde{\eta}_h^{n+1})^\frac{1}{2}\mathbb{D}(\bm{u}_h^{n+1})\|^2+\frac{2\tau}{T}|{Q}^{n+1}|^2+\frac{2\tau}{\mathrm{Pe}\mathrm{We}\mathrm{Cn}}\|\nabla\mu_h^{n+1}\|^2\\
                        =&0.
                    \end{aligned}
                \end{equation}
                In turn, the desired result is obtained by applying the following identity for a generic variable $\chi$:
                \begin{equation}
                    (\delta_\tau\chi^{n+1})2\tau\chi^{n+1}=
                    \begin{cases}
                        |\chi^{n+1}|^2-|\chi^n|^2+|\mathscr{D}\chi^{n+1}|^2  ,                                                                             & J=1, \\
                        \frac{1}{2}\left(|\chi^{n+1}|^2-|\chi^n|^2+|\widetilde{\chi}^{n+2}|^2-|\widetilde{\chi}^{n+1}|^2+|\mathscr{D}^2\chi^{n+1}|^2\right) ,  & J=2.
                    \end{cases}
                \end{equation}
            \end{proof}

        \subsection{The artificial compressibility scheme}
            In the algorithm \eqref{MCHNSt_1}, the Navier-Stokes equations are solved by a saddle point approach, which turns out to be computationally expensive. A decoupled structure is much more desired than a saddle point system in the large-scale simulations. Among the available approaches for incompressible flows with variable density (projection method \cite{2000_guermond}, gauge Uzawa method \cite{2007_pyo, 2017_wu}, pressure stabilization method \cite{2009_guermond, 2011_guermond}), the energy stability analysis of second-order in time scheme is only available for the pressure stabilization method, and only a suboptimal estimate was derived in \cite{2011_guermond}. However, the pressure stabilization method is not applicable to decouple the Stokes solver that appears in \eqref{MCHNSt_1} (see Remark \ref{reason_PS} below). Apart from the above-mentioned efforts, which have been extended to incompressible flows with variable density, the less investigated artificial compressibility method might be a substitute choice. The following perturbation of the Navier-Stokes equations is considered in this method:
            \begin{equation}
                \begin{aligned}
                    \rho\left(\bm{u}_t+\bm{u}\cdot\nabla\bm{u}\right)-\nabla\cdot\left(2\eta\mathbb{D}(\bm{u})\right)+\nabla{p}=&\bm{f},\\
                    -\nabla\cdot\bm{u}-\frac{\varepsilon}{\rho}{P}_t=&0,
                \end{aligned}
            \end{equation}
            where $\varepsilon$ is a very small and positive perturbation parameter (in a scale of $\mathcal{O}(\tau^2)$ \cite{2017_decaria}). In fact, it is not a new idea, and this method has been employed in \cite{2022_gao} (by Gao et al.) to develop a fully decoupled, first-order in time finite element scheme for the Cahn-Hilliard-Navier-Stokes-Darcy system with different densities. To obtain the second-order accuracy in time, we follow the ideas of \cite{2022_gao, 2009_guermond} and propose the following fully discrete finite element scheme: find $(\phi_h^{n+1},\mu_h^{n+1})\in{V}_h\times{V}_h,\bm{u}_h^{n+1}\in\bm{X}_h,P_h^{n+1}\in\times{N}_h,R^{n+1}\in\mathbb{R}^+,Q^{n+1}\in\mathbb{R}^+,$ such that for all $(\psi_h,\omega_h)\in{V}_h\times{V}_h,\bm{v}_h\in\bm{X}_h,q_h\in{N}_h,$
            \begin{subequations}
                \label{MCHNSt_2}
                \begin{align}
                    (\delta_\tau\phi_h^{n+1},\omega_h)=&\xi_1^{n+1}(\widetilde{\phi}_h^{n+1}\widetilde{\bm{u}}_h^{n+1},\nabla\omega_h)-\frac{1}{\mathrm{Pe}}(\nabla\mu_h^{n+1},\nabla\omega_h),\label{CHNS_phit_2}\\
                    (\mu_h^{n+1},\psi_h)=&\mathrm{Cn}^2(\nabla\phi_h^{n+1},\nabla\psi_h)+s(\phi_h^{n+1},\psi_h)+\xi_1^{n+1}(G'(\widetilde{\phi}_h^{n+1}),\psi_h),\label{CHNS_mut_2}\\
                    \delta_\tau{R}^{n+1}=&\frac{1}{2U^{n+1}}\left((G'(\widetilde{\phi}_h^{n+1}),\delta_\tau\phi_h^{n+1})-(\widetilde{\phi}_h^{n+1}\widetilde{\bm{u}}_h^{n+1},\nabla\mu_h^{n+1})+(\widetilde{\phi}_h^{n+1}\nabla\widetilde{\mu}_h^{n+1},\bm{u}_h^{n+1})\right),\label{CHNS_Rt_2}\\
                    \left(\widetilde{\sigma}_h^{n+1}\delta_\tau(\widetilde{\sigma}\bm{u})_h^{n+1},\bm{v}_h\right)&+\zeta\left(\nabla\cdot\delta_\tau\bm{u}_h^{n+1},\nabla\cdot\bm{v}_h\right)+\frac{2}{\mathrm{Re}}\left(\widetilde{\eta}_h^{n+1}\mathbb{D}(\bm{u}_h^{n+1}),\mathbb{D}(\bm{v}_h)\right)\notag\\
                    =&-\frac{\xi_2^{n+1}}{2}\left[\left((\widetilde{\rho}_h^{n+1}\widetilde{\bm{u}}_h^{n+1}+\widetilde{\bm{J}}_h^{n+1})\cdot\nabla\widetilde{\bm{u}}_h^{n+1},\bm{v}_h\right)-\left((\widetilde{\rho}_h^{n+1}\widetilde{\bm{u}}_h^{n+1}+\widetilde{\bm{J}}_h^{n+1})\cdot\nabla\bm{v},\widetilde{\bm{u}}_h^{n+1}\right)\right]\notag\\
                    &+(P_h^\sharp,\nabla\cdot\bm{v}_h)-\frac{\xi_1^{n+1}}{\mathrm{We}\mathrm{Cn}}(\widetilde{\phi}_h^{n+1}\nabla\widetilde{\mu}_h^{n+1},\bm{v}_h)+\frac{1}{\mathrm{Fr}}(\widetilde{\bm{f}}_h^{n+1},\bm{v}_h),\label{CHNS_rhout_2}\\
                    (\mathscr{D}P_h^{n+1},q_h)=&-\frac{\gamma_0\varrho}{\tau}(\nabla\cdot\bm{u}_h^{n+1},q_h),\label{MCHNS_rhot_2}\\
                    \delta_\tau{Q}^{n+1}=&-\frac{Q^{n+1}}{T}+\frac{1}{2}e^{t^{n+1}/T}\left[\left((\widetilde{\rho}_h^{n+1}\widetilde{\bm{u}}_h^{n+1}+\widetilde{\bm{J}}_h^{n+1})\cdot\nabla\widetilde{\bm{u}}_h^{n+1},\bm{u}_h^{n+1}\right)\right.\notag\\
                    &\left.-\left((\widetilde{\rho}_h^{n+1}\widetilde{\bm{u}}_h^{n+1}+\widetilde{\bm{J}}_h^{n+1})\cdot\nabla\bm{u}_h^{n+1},\widetilde{\bm{u}}_h^{n+1}\right)\right],\label{CHNS_Qt_2}
                \end{align}
            \end{subequations}
            where $\varrho=\frac{\min(\rho_1,\rho_2)}{\rho_r},\zeta\ge{3}\varrho$ and
            \begin{equation}
                P_h^\sharp=
                \begin{cases}
                    2P_h^n-P_h^{n-1}  ,                        & J=1,\\
                    \frac{1}{3}(7P_h^n-5P_h^{n-1}+P_h^{n-2}) , & J=2.
                \end{cases}
            \end{equation}
            \begin{remark}
                \label{reason_PS}
                The key ingredient for achieving energy stability in a fully decoupled scheme lies in the ability of discrete momentum terms to absorb additional terms generated by the pressure update equation. If the pressure stabilization method is employed to decouple the Navier-Stokes equations, however, it is discovered that, if a direct BDF-type discretization is applied to the term $\sigma(\sigma\bm{u})_t$, a theoretical justification of the energy stability analysis is not available any more. Hence, we utilize the artificial compressibility method instead, which also leads to a time-independent pressure update equation. Although the pressure stabilization method has been widely used to develop fully decoupled schemes in the existing literature, all of them are only first-order accurate in time \cite{2021_chen, 2022_chen, 2016_chen, 2015_shen}.
                We believe that, with the same discretization for the term $\sigma(\sigma\bm{u})_t$ in \cite{2011_guermond}, a second-order in time scheme can also be constructed with provable suboptimal energy stability.
            \end{remark}
            \begin{remark} 
                A direct discretization for the pressure update equation would be
                \begin{equation}
                    (\delta_\tau{P}_h^{n+1},q_h)=-\frac{\gamma_0\varrho}{\tau^2}(\nabla\cdot\bm{u}_h^{n+1},q_h),
                \end{equation}
                where $\varepsilon=\frac{\tau^2}{\gamma_0}$ is set, and a penalty formulation is adopted for $\rho$ as in \cite{2009_guermond}. However, this formulation will make the energy stability analysis much more complicated if $J=2$. Hence, we still use the first order BDF to discretize it while $\gamma_0$ remains the same, and introduce the term $\zeta\nabla\nabla\cdot\bm{u}_t$ to ensure an energy stability. In turn, this approach gives
                \begin{equation}
                    \frac{1}{\tau}(P_h^{n+1}-P_h^n,q_h)=-\frac{\gamma_0\varrho}{\tau^2}(\nabla\cdot\bm{u}_h^{n+1},q_h),
                \end{equation}
                which is equivalent to \eqref{MCHNS_rhot_2}. 
            \end{remark}

            In comparison with many other decoupling methods that involve a pressure Poisson equation, the proposed scheme \eqref{MCHNSt_2} is computationally far less expensive, since only a straightforward step of time marching is required to update the pressure. In addition, this approach eliminates the necessity of an artificial Neumann boundary condition, thereby avoiding the numerical boundary layer and non-physical numerical oscillation near the boundary \cite{2006_guermond}. On the other hand, the artificial compression method has the drawback of introducing non-physical acoustic waves, which will be discussed in section \ref{acoustic_waves}. The proposed scheme \eqref{MCHNSt_2} is also globally mass conservative and well-posed, and its numerical implementation process is similar to that of \eqref{MCHNSt_1}. Henceforth, we only provide its energy stability proof, for the sake of brevity.

            \begin{theorem}
                \label{energy_law1}
                In the absence of $\widetilde{f}_h^{n+1},$ for $J=1,$ the scheme \eqref{MCHNSt_2} satisfies the following modified discrete energy law:
                \begin{equation}
                    \mathscr{E}^N+\tau\sum_{k=1}^N\mathscr{P}^k\le{C} , \quad {N}\ge{1},
                \end{equation}
                where $C$ is a generic constant dependent on the initial data and dimensionless numbers, and
                \begin{equation}
                    \label{discrete_energy_2}
                    \begin{aligned}
                        \mathscr{E}^n=&\|(\widetilde{\sigma}\bm{u})_h^n\|^2+\frac{\mathrm{Cn}}{\mathrm{We}}\|\nabla\phi_h^n\|^2+\frac{s}{\mathrm{We}\mathrm{Cn}}\|\phi_h^n\|^2+\frac{2}{\mathrm{We}\mathrm{Cn}}|R^n|^2+|Q^n|^2+\zeta\|\nabla\cdot\bm{u}_h^n\|^2+\frac{\tau^2}{\varrho}\|P_h^n\|^2,\\
                        \mathscr{P}^n=&\frac{4}{\mathrm{Re}}\|(\eta_h^{n-1})^\frac{1}{2}\mathbb{D}(\bm{u}_h^n)\|^2+\frac{2}{T}|{Q}^n|^2+\frac{2}{\mathrm{Pe}\mathrm{We}\mathrm{Cn}}\|\nabla\mu_h^n\|^2.
                    \end{aligned}
                \end{equation}
            \end{theorem}
            \begin{proof}
                At any time step, taking $(\psi_h,\omega_h)=(2\tau\delta_\tau\phi_h^{n+1},2\tau\mu_h^{n+1})$ in \eqref{CHNS_phit_2}-\eqref{CHNS_mut_2} and multiplying \eqref{CHNS_Rt_2} by $4\tau{R}^{n+1}$ leads to
                \begin{equation}
                    \label{discrete_energy_phase_2_1}
                    \begin{aligned}
                        (\delta_\tau\phi_h^{n+1},2\tau\mu_h^{n+1})=&\xi_1^{n+1}(\phi_h^n\bm{u}_h^n,2\tau\nabla\mu_h^{n+1})-\frac{2\tau}{\mathrm{Pe}}\|\nabla\mu_h^{n+1}\|^2,\\
                        (\mu_h^{n+1},2\tau\delta_\tau\phi_h^{n+1})=&\mathrm{Cn}^2(\nabla\phi_h^{n+1},2\tau\delta_\tau\nabla\phi_h^{n+1})+s(\phi_h^{n+1},2\tau\delta_\tau\phi_h^{n+1})+\xi_1^{n+1}(G'(\phi_h^n),2\tau\delta_\tau\phi_h^{n+1}),\\
                        \delta_\tau{R}^{n+1}4\tau{R}^{n+1}=&\xi_1^{n+1}\left((G'(\phi_h^n),2\tau\delta_\tau\phi_h^{n+1})-(\phi_h^n\bm{u}_h^n,2\tau\nabla\mu_h^{n+1})+(\phi_h^n\nabla\mu_h^n,2\tau\bm{u}_h^{n+1})\right).
                    \end{aligned}
                \end{equation}
                Taking $\bm{v}_h=2\tau\bm{u}_h^{n+1}$ in \eqref{CHNS_rhout_2} and multiplying  \eqref{CHNS_Qt_2} by $2\tau{Q}^{n+1}$, we get
                \begin{equation}
                    \label{discrete_energy_flow_2_1}
                    \begin{aligned}
                        \left(\delta_\tau(\widetilde{\sigma}\bm{u})_h^{n+1},2\tau(\widetilde{\sigma}\bm{u})_h^{n+1}\right)&+\zeta\left(\nabla\cdot\delta_\tau\bm{u}_h^{n+1},2\tau\nabla\cdot\bm{u}_h^{n+1}\right)+\frac{4\tau}{\mathrm{Re}}\|(\eta_h^n)^\frac{1}{2}\mathbb{D}(\bm{u}_h^{n+1})\|^2\\
                        =&-\frac{\xi_2^{n+1}}{2}\left[\left((\rho_h^n\bm{u}_h^n+\bm{J}_h^n)\cdot\nabla\bm{u}_h^n,2\tau\bm{u}_h^{n+1}\right)-\left(2\tau(\rho_h^n\bm{u}_h^n+\bm{J}_h^n)\cdot\nabla\bm{u}_h^{n+1},\bm{u}_h^n\right)\right]\\
                        &+(P_h^\sharp,2\tau\nabla\cdot\bm{u}_h^{n+1})-\frac{\xi_1^{n+1}}{\mathrm{We}\mathrm{Cn}}(\phi_h^n\nabla\mu_h^n,2\tau\bm{u}_h^{n+1}) , \\
                        \delta_\tau{Q}^{n+1}2\tau{Q}^{n+1}+\frac{2\tau}{T}|{Q}^{n+1}|^2=&\frac{\xi_2^{n+1}}{2}\left[\left((\rho_h^n\bm{u}_h^n+\bm{J}_h^n)\cdot\nabla\bm{u}_h^n,2\tau\bm{u}_h^{n+1}\right)-\left(2\tau(\rho_h^n\bm{u}_h^n+\bm{J}_h^n)\cdot\nabla\bm{u}_h^{n+1},\bm{u}_h^n\right)\right].
                    \end{aligned}
                \end{equation}
                A summation of equations \eqref{discrete_energy_phase_2_1} and \eqref{discrete_energy_flow_2_1} yields
                \begin{equation}
                    \label{discrete_energy_mid_2_1}
                    \begin{aligned}
                        \frac{s}{\mathrm{We}\mathrm{Cn}}(\phi_h^{n+1},2\tau\delta_\tau\phi_h^{n+1})&+\frac{\mathrm{Cn}}{\mathrm{We}}(\nabla\phi_h^{n+1},2\tau\delta_\tau\nabla\phi_h^{n+1})+\frac{1}{\mathrm{We}\mathrm{Cn}}\delta_\tau{R}^{n+1}4\tau{R}^{n+1}+\left(\delta_\tau(\widetilde{\sigma}\bm{u})_h^{n+1},2\tau(\widetilde{\sigma}\bm{u})_h^{n+1}\right)\\
                        &+\zeta\left(\nabla\cdot\delta_\tau\bm{u}_h^{n+1},2\tau\nabla\cdot\bm{u}_h^{n+1}\right)+\delta_\tau{Q}^{n+1}2\tau{Q}^{n+1}-(P_h^\sharp,2\tau\nabla\cdot\bm{u}_h^{n+1})\\
                        =&-\frac{4\tau}{\mathrm{Re}}\|(\eta_h^n)^\frac{1}{2}\mathbb{D}(\bm{u}_h^{n+1})\|^2-\frac{2\tau}{T}|{Q}^{n+1}|^2-\frac{2\tau}{\mathrm{Pe}\mathrm{We}\mathrm{Cn}}\|\nabla\mu_h^{n+1}\|^2.
                    \end{aligned}
                \end{equation}
                Notice that $P_h^\sharp=2P_h^n-P_h^{n-1}=P_h^{n+1}-\mathscr{D}^2P_h^{n+1}$. By \eqref{MCHNS_rhot_2}, we see that
                \begin{equation}
                    -(P_h^\sharp,2\tau\nabla\cdot\bm{u}_h^{n+1})=\frac{2\tau^2}{\varrho}(\mathscr{D}P_h^{n+1},P_h^{n+1}-\mathscr{D}^2P_h^{n+1}).
                \end{equation}
                Meanwhile, the following identity is a direct consequence of \eqref{MCHNS_rhot_2}:
                \begin{equation}
                    (\mathscr{D}^2P_h^{n+1},q_h)=-\frac{\varrho}{\tau}(\nabla\cdot\mathscr{D}\bm{u}_h^{n+1},q_h) , \quad \forall{q}_h\in{N}_h,
                \end{equation}
                which in turn implies that
                \begin{equation}
                    \label{control_nableP_2}
                    \frac{\tau^2}{\varrho}\|\mathscr{D}^2P_h^{n+1}\|^2\le\zeta\|\nabla\cdot\mathscr{D}\bm{u}_h^{n+1}\|^2.
                \end{equation}

                We recall the following identity
                \begin{equation}
                    (\delta_\tau\chi^{n+1})2\tau\chi^{n+1}=|\chi^{n+1}|^2-|\chi^n|^2+|\mathscr{D}\chi^{n+1}|^2 .
                \end{equation}
                Its substitution into \eqref{discrete_energy_mid_2_1} gives (for simplicity we drop some unnecessary $|\mathscr{D}\chi^{n+1}|^2$ terms on the left-hand side)
                \begin{equation}
                    \begin{aligned}
                        \|(\widetilde{\sigma}&\bm{u})_h^{n+1}\|^2-\|(\widetilde{\sigma}\bm{u})_h^n\|^2+\frac{\mathrm{Cn}}{\mathrm{We}}\left(\|\nabla\phi_h^{n+1}\|^2-\|\nabla\phi_h^n\|^2\right)+\frac{s}{\mathrm{We}\mathrm{Cn}}\left(\|\phi_h^{n+1}\|^2-\|\phi_h^n\|^2\right)\\
                        &+\frac{2}{\mathrm{We}\mathrm{Cn}}\left(|R^{n+1}|^2-|R^n|^2\right)+|Q^{n+1}|^2-|Q^n|^2+\zeta\left(\|\nabla\cdot\bm{u}_h^{n+1}\|^2-\|\nabla\cdot\bm{u}_h^n\|^2+\|\nabla\cdot\mathscr{D}\bm{u}_h^{n+1}\|^2\right)\\
                        &+\frac{\tau^2}{\varrho}\left(\|P_h^{n+1}\|^2-\|P_h^n\|^2+\|\mathscr{D}{P}_h^{n+1}\|^2-\|\mathscr{D}{P}_h^{n+1}\|^2+\|\mathscr{D}{P}_h^n\|^2-\|\mathscr{D}^2{P}_h^{n+1}\|^2\right)\\
                        \le&-\frac{4\tau}{\mathrm{Re}}\|(\eta_h^n)^\frac{1}{2}\mathbb{D}(\bm{u}_h^{n+1})\|^2-\frac{2\tau}{T}|{Q}^{n+1}|^2-\frac{2\tau}{\mathrm{Pe}\mathrm{We}\mathrm{Cn}}\|\nabla\mu_h^{n+1}\|^2.
                    \end{aligned}
                \end{equation}
                Subsequently, its combination with \eqref{control_nableP_2} results in
                \begin{equation}
                    \begin{aligned}
                        \|(\widetilde{\sigma}&\bm{u})_h^{n+1}\|^2-\|(\widetilde{\sigma}\bm{u})_h^n\|^2+\frac{\mathrm{Cn}}{\mathrm{We}}\left(\|\nabla\phi_h^{n+1}\|^2-\|\nabla\phi_h^n\|^2\right)+\frac{s}{\mathrm{We}\mathrm{Cn}}\left(\|\phi_h^{n+1}\|^2-\|\phi_h^n\|^2\right)+|Q^{n+1}|^2-|Q^n|^2\\
                        &+\frac{2}{\mathrm{We}\mathrm{Cn}}\left(|R^{n+1}|^2-|R^n|^2\right)+\zeta\left(\|\nabla\cdot\bm{u}_h^{n+1}\|^2-\|\nabla\cdot\bm{u}_h^n\|^2\right)+\frac{\tau^2}{\varrho}\left(\|P_h^{n+1}\|^2-\|P_h^n\|^2\right)\\
                        \le&-\frac{4\tau}{\mathrm{Re}}\|(\eta_h^n)^\frac{1}{2}\mathbb{D}(\bm{u}_h^{n+1})\|^2-\frac{2\tau}{T}|{Q}^{n+1}|^2-\frac{2\tau}{\mathrm{Pe}\mathrm{We}\mathrm{Cn}}\|\nabla\mu_h^{n+1}\|^2.
                    \end{aligned}
                \end{equation}

                At the initial time step, we see that $P_h^\sharp=P_h^0=P_h^1-\mathscr{D}P_h^1$, so that \eqref{discrete_energy_mid_2_1} becomes
                \begin{equation}
                    \begin{aligned}
                        \|(\widetilde{\sigma}&\bm{u})_h^1\|^2+\frac{\mathrm{Cn}}{\mathrm{We}}\|\nabla\phi_h^1\|^2+\frac{s}{\mathrm{We}\mathrm{Cn}}\|\phi_h^1\|^2+\frac{2}{\mathrm{We}\mathrm{Cn}}|R^1|^2+|Q^1|^2+2\zeta\|\nabla\cdot\bm{u}_h^1\|^2\\
                        &+\frac{4\tau}{\mathrm{Re}}\|(\eta_h^0)^\frac{1}{2}\mathbb{D}(\bm{u}_h^1)\|^2+\frac{2\tau}{T}|{Q}^1|^2+\frac{2\tau}{\mathrm{Pe}\mathrm{We}\mathrm{Cn}}\|\nabla\mu_h^1\|^2+\frac{\tau^2}{\varrho}\left(\|P_h^1\|^2-\|\mathscr{D}{P}_h^1\|^2\right)\\
                        \le&\|(\sigma\bm{u})_h^0\|^2+\frac{\mathrm{Cn}}{\mathrm{We}}\|\nabla\phi_h^0\|^2+\frac{s}{\mathrm{We}\mathrm{Cn}}\|\phi_h^0\|^2+\frac{2}{\mathrm{We}\mathrm{Cn}}|R^0|^2+|Q^0|^2+\frac{\tau^2}{\varrho}\|P_h^0\|^2.
                    \end{aligned}
                \end{equation}
                Meanwhile, the following inequality comes from \eqref{MCHNS_rhot_2}:
                \begin{equation}
                    \label{control_nableP_1}
                    \frac{\tau^2}{\varrho}\|\mathscr{D}P_h^1\|^2\le\zeta\|\nabla\cdot\bm{u}_h^1\|^2 .
                \end{equation}
                Consequently, the desired result could be obtained by a summation from $n=0,1,\cdots,N-1$.
            \end{proof}

            \begin{lemma}[Three term recursion inequality \cite{2011_guermond}]
                \label{ttri}
                Let $\{x^n\}_{n\ge{0}}$ satisfy the three term recursion inequality
                \begin{equation}
                    3x^{n+1}-4x^n+x^{n-1}\le{g}^{n+1} , \quad {n}\ge{1},
                \end{equation}
                with initial data $x^0$ and $x^1.$ Then there are constants $c_0$ and $c_1$ that depend only on $x^0$ and $x^1$ such that for any $N\ge{2}$,
                \begin{equation}
                    x^N\le{c}_0+\frac{c_1}{3^N}+\sum_{j=2}^N\frac{1}{3^{N+1-j}}\sum_{k=2}^jg^k.
                \end{equation}
            \end{lemma}
            \begin{theorem}
                \label{energy_law2}
                In the absence of $\widetilde{f}_h^{n+1},$  for $J=2,$ the numerical scheme \eqref{MCHNSt_2} satisfies the following modified discrete energy law:
                \begin{equation}
                    \mathscr{E}^N+\tau\sum_{k=2}^N\mathscr{P}^k\le{C} , \quad {N}\ge{2},
                \end{equation}
                where $C$ is a generic constant dependent on the initial data and dimensionless numbers, and
                \begin{equation}
                    \begin{aligned}
                        \mathscr{E}^n=&\frac{1}{2}\|(\widetilde{\sigma}\bm{u})_h^n\|^2+\frac{\mathrm{Cn}}{2\mathrm{We}}\|\nabla\phi_h^n\|^2+\frac{s}{2\mathrm{We}\mathrm{Cn}}\|\phi_h^n\|^2+\frac{1}{\mathrm{We}\mathrm{Cn}}|R^n|^2+\frac{1}{2}|Q^n|^2+\frac{\zeta}{2}\|\nabla\cdot\bm{u}_h^n\|^2+\frac{2\tau^2}{9\varrho}\|P_h^n\|^2,\\
                        \mathscr{P}^n=&\frac{4}{3\mathrm{Re}}\|(\widetilde{\eta}_h^n)^\frac{1}{2}\mathbb{D}(\bm{u}_h^n)\|^2+\frac{2}{3T}|{Q}^n|^2+\frac{2}{3\mathrm{Pe}\mathrm{We}\mathrm{Cn}}\|\nabla\mu_h^n\|^2+\frac{2\tau}{9\varrho}\|\mathscr{D}P_h^{n-1}\|^2.
                    \end{aligned}
                \end{equation}
            \end{theorem}
            \begin{proof}
                At each time step, with the same treatment as in Theorem~\ref{energy_law1}, we obtain
                \begin{equation}
                    \label{discrete_energy_mid_2_2}
                    \begin{aligned}
                        \frac{s}{\mathrm{We}\mathrm{Cn}}(\phi_h^{n+1},2\tau\delta_\tau\phi_h^{n+1})&+\frac{\mathrm{Cn}}{\mathrm{We}}(\nabla\phi_h^{n+1},2\tau\delta_\tau\nabla\phi_h^{n+1})+\frac{1}{\mathrm{We}\mathrm{Cn}}\delta_\tau{R}^{n+1}4\tau{R}^{n+1}+\left(\delta_\tau(\widetilde{\sigma}\bm{u})_h^{n+1},2\tau(\widetilde{\sigma}\bm{u})_h^{n+1}\right)\\
                        &+\zeta\left(\nabla\cdot\delta_\tau\bm{u}_h^{n+1},2\tau\nabla\cdot\bm{u}_h^{n+1}\right)+\delta_\tau{Q}^{n+1}2\tau{Q}^{n+1}-(P_h^\sharp,2\tau\nabla\cdot\bm{u}_h^{n+1})\\
                        =&-\frac{4\tau}{\mathrm{Re}}\|(\widetilde{\eta}_h^{n+1})^\frac{1}{2}\mathbb{D}(\bm{u}_h^{n+1})\|^2-\frac{2\tau}{T}|{Q}^{n+1}|^2-\frac{2\tau}{\mathrm{Pe}\mathrm{We}\mathrm{Cn}}\|\nabla\mu_h^{n+1}\|^2.
                    \end{aligned}
                \end{equation}
                Notice that $P_h^\sharp=\frac{1}{3}(7P_h^n-5P_h^{n-1}+P_h^{n-2})=P_h^{n+1}-\mathscr{D}^2P_h^{n+1}+\frac{1}{3}\mathscr{D}^2P_h^n$. By \eqref{MCHNS_rhot_2}, we see that
                \begin{equation}
                    -(P_h^\sharp,2\tau\nabla\cdot\bm{u}_h^{n+1})=\frac{4\tau^2}{3\varrho}(\mathscr{D}P_h^{n+1},P_h^{n+1}-\mathscr{D}^2P_h^{n+1})
                    +\frac{4\tau^2}{9\varrho}(\mathscr{D}P_h^{n+1},\mathscr{D}^2P_h^n) .
                \end{equation}
                Meanwhile, the following inequality could be derived in a similar manner as in \eqref{control_nableP_2}:
                \begin{equation}
                    \label{control_nableP_3}
                    \frac{4\tau^2}{9\varrho}\|\mathscr{D}^3P_h^{n+1}\|^2\le\zeta\|\nabla\cdot\mathscr{D}^2\bm{u}_h^{n+1}\|^2.
                \end{equation}

                Because of the following identities
                \begin{equation}
                    (\delta_t\chi^{n+1})2\tau\chi^{n+1}=\frac{1}{2}\biggl(3|\chi^{n+1}|^2-4|\chi^n|^2+|\chi^{n-1}|^2+2|\mathscr{D}\chi^{n+1}|^2-2|\mathscr{D}\chi^n|^2+|\mathscr{D}^2\chi^{n+1}|^2\biggr),
                \end{equation}
                \begin{equation}
                    \varrho\|\nabla\cdot\mathscr{D}\bm{u}_h^{n+1}\|^2=\|\varrho^\frac{1}{2}\nabla\cdot\mathscr{D}\bm{u}_h^{n+1}+\frac{2\tau}{3\varrho^\frac{1}{2}}\mathscr{D}^2P_h^{n+1}\|^2+\frac{4\tau^2}{9\varrho}\|\mathscr{D}^2P_h^{n+1}\|^2,
                \end{equation}
                it is clear that their substitution into \eqref{discrete_energy_mid_2_2} leads to
                \begin{equation}
                    \label{intermediate}
                    \begin{aligned}
                        \frac{1}{2}&\left(3\|(\widetilde{\sigma}\bm{u})_h^{n+1}\|^2-4\|(\widetilde{\sigma}\bm{u})_h^n\|^2+\|(\widetilde{\sigma}\bm{u})_h^{n-1}\|^2+2\|\mathscr{D}(\widetilde{\sigma}\bm{u})_h^{n+1}\|^2-2\|\mathscr{D}(\widetilde{\sigma}\bm{u})_h^n\|^2\right)\\
                        &+\frac{\mathrm{Cn}}{2\mathrm{We}}\biggl(3\|\nabla\phi_h^{n+1}\|^2-4\|\nabla\phi_h^n\|^2+\|\nabla\phi_h^{n-1}\|^2+2\|\nabla\mathscr{D}\phi_h^{n+1}\|^2-2\|\nabla\mathscr{D}\phi_h^n\|^2\biggr)\\
                        &+\frac{s}{2\mathrm{We}\mathrm{Cn}}\biggl(3\|\phi_h^{n+1}\|^2-4\|\phi_h^n\|^2+\|\phi_h^{n-1}\|^2+2\|\mathscr{D}\phi_h^{n+1}\|^2-2\|\mathscr{D}\phi_h^n\|^2\biggr)\\
                        &+\frac{1}{\mathrm{We}\mathrm{Cn}}\biggl(3|R^{n+1}|^2-4|R^n|^2+|R^{n-1}|^2+2|\mathscr{D}R^{n+1}|^2-2|\mathscr{D}R^n|^2\biggr)\\
                        &+\frac{1}{2}\left(3|Q^{n+1}|^2-4|Q^n|^2+|Q^{n-1}|^2+2|\mathscr{D}Q^{n+1}|^2-2|\mathscr{D}Q^n|^2\right)\\
                        &+\frac{\zeta}{2}\left(3\|\nabla\cdot\bm{u}_h^{n+1}\|^2-4\|\nabla\cdot\bm{u}_h^n\|^2+\|\nabla\cdot\bm{u}_h^{n-1}\|^2+\|\nabla\cdot\mathscr{D}^2\bm{u}_h^{n+1}\|^2\right)\\
                        &+\|(\zeta-\varrho)^\frac{1}{2}\nabla\cdot\mathscr{D}\bm{u}_h^{n+1}\|^2-\|(\zeta-\varrho)^\frac{1}{2}\nabla\cdot\mathscr{D}\bm{u}_h^n\|^2\\
                        &+\|\varrho^\frac{1}{2}\nabla\cdot\mathscr{D}\bm{u}_h^{n+1}+\frac{2\tau}{3\varrho^\frac{1}{2}}\mathscr{D}^2P_h^{n+1}\|^2-\|\varrho^\frac{1}{2}\nabla\cdot\mathscr{D}\bm{u}_h^n+\frac{2\tau}{3\varrho^\frac{1}{2}}\mathscr{D}^2P_h^n\|^2\\
                        &+\frac{2\tau^2}{3\varrho}\biggl(\|P_h^{n+1}\|^2-\|P_h^n\|^2+\|\mathscr{D}P_h^n\|^2\biggr)\\
                        &-\frac{2\tau^2}{9\varrho}\|\mathscr{D}^2P_h^{n+1}\|^2-\frac{4\tau^2}{9\varrho}\|\mathscr{D}^2P_h^n\|^2+\frac{4\tau^2}{9\varrho}(\mathscr{D}P_h^{n+1},\mathscr{D}^2P_h^n)\\
                        \le&-\frac{4\tau}{\mathrm{Re}}\|(\widetilde{\eta}_h^{n+1})^\frac{1}{2}\mathbb{D}(\bm{u}_h^{n+1})\|^2-\frac{2\tau}{T}|{Q}^{n+1}|^2-\frac{2\tau}{\mathrm{Pe}\mathrm{We}\mathrm{Cn}}\|\nabla\mu_h^{n+1}\|^2.
                    \end{aligned}
                \end{equation}
                Again, we have dropped some unnecessary $|\mathscr{D}^2\chi^{n+1}|^2$ terms on the left-hand side, for simplicity of presentation. Moreover, based on the following equalities
                \begin{equation}
                    \begin{aligned}
                        \|\mathscr{D}^3P_h^{n+1}\|^2&=\|\mathscr{D}^2P_h^{n+1}\|^2-2(\mathscr{D}^2P_h^{n+1},\mathscr{D}^2P_h^n)+\|\mathscr{D}^2P_h^n\|^2,\\
                        \mathscr{D}^2P_h^n+2\mathscr{D}^2P_h^{n+1}-2\mathscr{D}P_h^{n+1}&=-\mathscr{D}P_h^n-\mathscr{D}P_h^{n-1},
                    \end{aligned}
                \end{equation}
                the last three terms on the left-hand side of \eqref{intermediate} become
                \begin{equation}
                    \begin{aligned}
                        -\frac{2\tau^2}{9\varrho}&\|\mathscr{D}^2P_h^{n+1}\|^2-\frac{4\tau^2}{9\varrho}\|\mathscr{D}^2P_h^n\|^2+\frac{4\tau^2}{9\varrho}(\mathscr{D}P_h^{n+1},\mathscr{D}^2P_h^n)\\
                        &=-\frac{2\tau^2}{9\varrho}\|\mathscr{D}^3P_h^{n+1}\|^2+\frac{2\tau^2}{9\varrho}(\mathscr{D}P_h^n+\mathscr{D}P_h^{n-1},\mathscr{D}^2P_h^n)\\
                        &\ge-\frac{\zeta}{2}\|\nabla\cdot\mathscr{D}^2\bm{u}_h^{n+1}\|^2+\frac{2\tau^2}{9\varrho} \Big( \|\mathscr{D}P_h^n\|^2-\|\mathscr{D}P_h^{n-1}\|^2 \Big ).
                    \end{aligned}
                \end{equation}

                The following quantities are introduced for the convenience of presentation:
                \begin{equation}
                    \begin{aligned}
                        x^n=&\frac{1}{2}\|(\widetilde{\sigma}\bm{u})_h^n\|^2+\frac{\mathrm{Cn}}{2\mathrm{We}}\|\nabla\phi_h^n\|^2+\frac{s}{2\mathrm{We}\mathrm{Cn}}\|\phi_h^n\|^2+\frac{1}{\mathrm{We}\mathrm{Cn}}|R^n|^2+\frac{1}{2}|Q^n|^2+\frac{\zeta}{2}\|\nabla\cdot\bm{u}_h^n\|^2,\\
                        b^n=&\frac{4\tau}{\mathrm{Re}}\|(\widetilde{\eta}_h^n)^\frac{1}{2}\mathbb{D}(\bm{u}_h^n)\|^2+\frac{2\tau}{T}|{Q}^n|^2+\frac{2\tau}{\mathrm{Pe}\mathrm{We}\mathrm{Cn}}\|\nabla\mu_h^n\|^2+\frac{2\tau^2}{3\varrho}\|\mathscr{D}P_h^{n-1}\|^2,\\
                        d^n=&\|\mathscr{D}(\widetilde{\sigma}\bm{u})_h^n\|^2+\frac{\mathrm{Cn}}{\mathrm{We}}\|\nabla\mathscr{D}\phi_h^n\|^2+\frac{s}{\mathrm{We}\mathrm{Cn}}\|\mathscr{D}\phi_h^n\|^2+\frac{2}{\mathrm{We}\mathrm{Cn}}|\mathscr{D}R^n|^2+|\mathscr{D}Q^n|^2\\
                        &+\|(\zeta-\varrho)^\frac{1}{2}\nabla\cdot\mathscr{D}\bm{u}_h^n\|^2+\|\varrho^\frac{1}{2}\nabla\cdot\mathscr{D}\bm{u}_h^n+\frac{2\tau}{3\varrho^\frac{1}{2}}\mathscr{D}^2P_h^n\|^2+\frac{2\tau^2}{3\varrho}\|P_h^n\|^2+\frac{2\tau^2}{9\varrho}\|\mathscr{D}P_h^{n-1}\|^2 .
                    \end{aligned}
                \end{equation}
                In turn, estimate \eqref{intermediate} could be rewritten as
                \begin{equation}
                    3x^{n+1}-4x^n+x^{n-1}\le{g}^{n+1} , \quad {n}\ge{2},
                \end{equation}
                where $g^{n+1}=-(b^{n+1}+d^{n+1}-d^n)$. By Lemma~\ref{ttri}, it is clear that
                \begin{equation}
                    x^N\le{c}\left(1+\frac{1}{3^N}\right)({x}^1+x^2)-\sum_{j=3}^N\frac{1}{3^{N+1-j}}\sum_{k=3}^j(b^k+d^k-d^{k-1}), \quad \forall N \ge 3 ,
                \end{equation}
                which leads to
                \begin{equation}
                    x^N+\frac{1}{3}\sum_{k=3}^Nb^k+\frac{1}{3}d^N\le{c}({x}^1+x^2+d^2) ,
                \end{equation}
                for some constants $c$. Therefore, the result is valid for $N\ge{3}$, by dropping some positive terms on the left-hand side, provided that $x^2,b^2$ and $d^2$ are bounded by the initial data.

                In terms of the initial time step, the identity $P_h^\sharp=2P_h^1-P_h^0=P_h^2-\mathscr{D}^2P_h^2$ indicates that
                \begin{equation}
                    -(P_h^\sharp,2\tau\nabla\cdot\bm{u}_h^2)=\frac{4\tau^2}{3\varrho}(\mathscr{D}P_h^2,P_h^2-\mathscr{D}^2P_h^2)=\frac{2\tau^2}{3\varrho}\left(\|P_h^2\|^2-\|P_h^1\|^2+\|\mathscr{D}P_h^1\|^2-\|\mathscr{D}^2P_h^2\|^2\right) .
                \end{equation}
                Consequently, estimate \eqref{discrete_energy_mid_2_2} could be rewritten as
                \begin{equation}
                    \begin{aligned}
                        \frac{1}{2}&\left(3\|(\widetilde{\sigma}\bm{u})_h^2\|^2+2\|\mathscr{D}(\widetilde{\sigma}\bm{u})_h^2\|^2\right)+\frac{\mathrm{Cn}}{2\mathrm{We}}\left(3\|\nabla\phi_h^2\|^2+2\|\nabla\mathscr{D}\phi_h^2\|^2\right)+\frac{s}{2\mathrm{We}\mathrm{Cn}}\left(3\|\phi_h^2\|^2+2\|\mathscr{D}\phi_h^2\|^2\right)\\
                        &+\frac{1}{\mathrm{We}\mathrm{Cn}}\left(3|R^2|^2+2|\mathscr{D}R^2|^2\right)+\frac{1}{2}\left(3|Q^2|^2+2|\mathscr{D}Q^2|^2\right)+\frac{\zeta}{2}\left(3\|\nabla\cdot\bm{u}_h^2\|^2+2\|\nabla\cdot\mathscr{D}\bm{u}_h^2\|^2\right)\\
                        &+\frac{2\tau^2}{3\varrho}\left(\|P_h^2\|^2-\|\mathscr{D}^2P_h^2\|^2\right)+\frac{4\tau}{\mathrm{Re}}\|(\widetilde{\eta}_h^2)^\frac{1}{2}\mathbb{D}(\bm{u}_h^2)\|^2+\frac{2\tau}{T}|{Q}^2|^2+\frac{2\tau}{\mathrm{Pe}\mathrm{We}\mathrm{Cn}}\|\nabla\mu_h^2\|^2+\frac{2\tau^2}{3\varrho}\|\mathscr{D}P_h^1\|^2\\
                        \le&2\|(\widetilde{\sigma}\bm{u})_h^1\|^2+\|\mathscr{D}(\widetilde{\sigma}\bm{u})_h^1\|^2+\frac{2\mathrm{Cn}}{\mathrm{We}}\|\nabla\phi_h^1\|^2+\frac{\mathrm{Cn}}{\mathrm{We}}\|\mathscr{D}\phi_h^1\|^2+\frac{2s}{\mathrm{We}\mathrm{Cn}}\|\phi_h^1\|^2+\frac{s}{\mathrm{We}\mathrm{Cn}}\|\mathscr{D}\phi_h^1\|^2\\
                        &+\frac{4}{\mathrm{We}\mathrm{Cn}}|R^1|^2+\frac{2}{\mathrm{We}\mathrm{Cn}}|\mathscr{D}R^1|^2+2|Q^1|^2+|\mathscr{D}Q^1|^2+3\zeta\|\nabla\cdot\bm{u}_h^1\|^2+\frac{2\tau^2}{3\varrho}\|P_h^1\|^2 .
                    \end{aligned}
                \end{equation}
                Meanwhile, because of the following inequality, which comes from \eqref{MCHNS_rhot_2}:
                \begin{equation}
                    \label{control_nableP_4}
                    \frac{2\tau^2}{3\varrho}\|\mathscr{D}^2P_h^2\|^2\le\frac{\zeta}{2}\|\nabla\cdot\mathscr{D}\bm{u}_h^2\|^2,
                \end{equation}
                we obtain bounds for $x^2,$ $b^2$ and some terms in $d^2$. Finally, with the help of the identity
                \begin{equation}
                    \|(\zeta-\varrho)^\frac{1}{2}\nabla\cdot\mathscr{D}\bm{u}_h^2\|^2+\|\varrho^\frac{1}{2}\nabla\cdot\mathscr{D}\bm{u}_h^2+\frac{2\tau}{3\varrho^\frac{1}{2}}\mathscr{D}^2P_h^2\|^2=\zeta\|\nabla\cdot\mathscr{D}\bm{u}_h^2\|^2-\frac{4\tau^2}{9\varrho}\|\mathscr{D}^2P_h^2\|^2 ,
                \end{equation}
                the bound for the rest terms in $d^2$ could be similarly derived. This completes the proof.
            \end{proof}
            \begin{remark}
                In practice, we use the formula $(\mathscr{D}P_h^1,q_h)=-\frac{\gamma_0\varrho}{\tau}(\nabla\cdot\bm{u}_h^1,q_h)$ to obtain $P_h^1,$ which accounts for the validity of \eqref{control_nableP_3}, \eqref{control_nableP_4} to simplify the analysis. One can also use the first order BDF method (presented above) to get $P_h^1$, and the discrete energy law is still valid, with a minor modification in the theoretical proof; also see \cite{2011_guermond} for more details.
            \end{remark}
            \begin{remark}
                In our knowledge, the three term recursion inequality is the only technique to prove the energy stability of the BDF2 scheme for incompressible flows with variable density. Of course, the theoretical result in Theorem \ref{energy_law2} is still suboptimal, and more advanced techniques need to be developed to derive a more accurate estimate. Irrespective of the theoretical challenges for the BDF2 scheme, the time filter technique~\cite{2022_decaria, 2021_decaria}, which has emerged as an efficient and easily implemented method in developing high order time-stepping schemes in CFD, might be applicable for developing second-order accurate schemes with provable energy stability. In fact, this technique has been successfully employed to incompressible flows with variable density \cite{2023_li}. Its implementation will be investigated in our future work.
            \end{remark}
            \begin{remark}
                In addition to the energy stability analysis, a theoretical justification of convergence analysis and error estimate is highly desirable for the AGG model. In fact, for various physical systems of a phase field model coupled with incompressible fluid motion, there have been some existing works of optimal rate error analysis in the case of a constant density, such as the Cahn-Hilliard-Hele-Shaw~\cite{chen19a, chen16, LiuY17}, Cahn-Hilliard-Navier-Stokes~\cite{2017_diegel}, Cahn-Hilliard-Stokes-Darcy~\cite{chen22c}, etc. Meanwhile, the convergence estimates have also been reported for the SAV numerical schemes~\cite{ChengQ2021a, WangM2021}, if there is no fluid motion. Of course, an extension to the AGG model, in which the density variation has played an important role, would be highly challenging. The convergence property has been reported in~\cite{2014_grun}, while a convergence rate analysis was not available. An optimal rate convergence analysis for the proposed numerical schemes, \eqref{MCHNSt_1} and \eqref{MCHNSt_2}, will be considered in the future works.
            \end{remark}

    \section{Numerical results}
    \label{section_4}
        In this section, several representative numerical tests are presented to investigate the validity of the proposed numerical schemes \eqref{MCHNSt_1} and \eqref{MCHNSt_2}. If without any specific explanations, in all the simulations, we take $s=2$, $S=10$ in the SAV formulation \eqref{SAV}, $\zeta=3\varrho$ in the scheme \eqref{MCHNSt_2}, and use the MINI $\mathbb{P}_1^{\text{b}}\times\mathbb{P}_1$ element for $(\bm{u}_h,P_h),$ and Lagrange $\mathbb{P}_1$ element for $\phi_h$ and $\mu_h$. For simplicity, we use PG (respective to AC) to denote the scheme \eqref{MCHNSt_1} (respective to  \eqref{MCHNSt_2}), and use CM (respective to DM) to denote constant mobility (respective to degenerate mobility).

        \subsection{Accuracy test}
            \begin{figure}[ht]
                \centering
                \subfloat[PG+CM.]{
                    \includegraphics[width=0.4\linewidth]{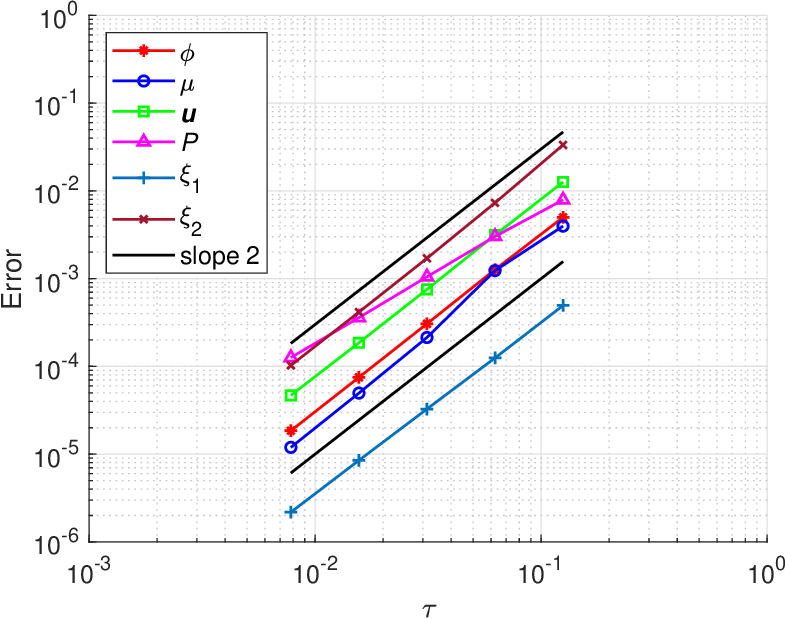}
                }
                \subfloat[AC+CM.]{
                    \includegraphics[width=0.4\linewidth]{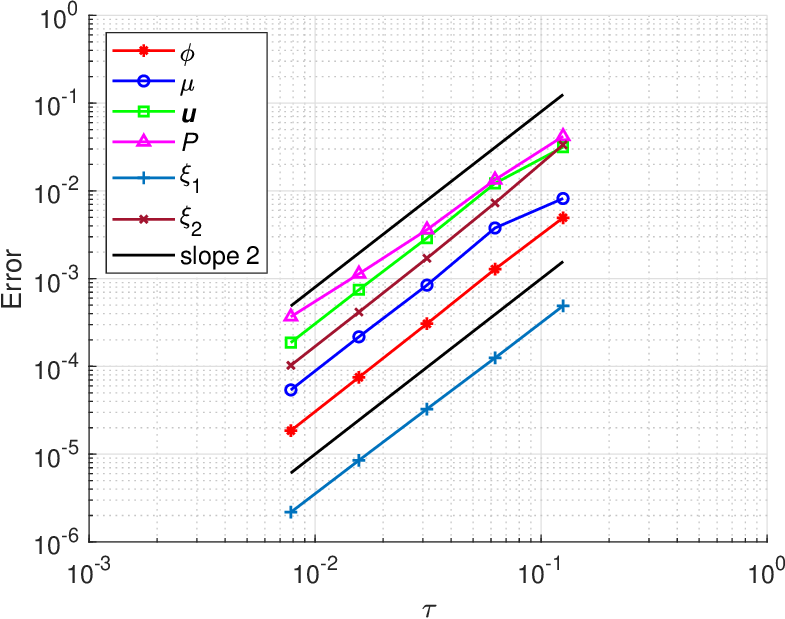}
                }
                \caption{Numerical errors, with $\tau=h$, at $T=0.5.$}
                \label{error}
            \end{figure}
            First we use manufactured analytic solutions to assess the convergence orders of the proposed schemes. The computational domain is taken as $\Omega=(0,1)^2$. With suitable extra source terms, the exact solutions are given by
            \begin{equation}
                \begin{aligned}
                    \phi=\mu=&\cos(\pi{x})\cos(\pi{y})\sin{t},\\
                    \bm{u}=&\left(\sin^2(\pi{x})\sin(2\pi{y}),-\sin(2\pi{x})\sin^2(\pi{y})\right)^T\sin{t},\\
                    P=&\cos(\pi{x})\sin(\pi{y})\cos{t}.
                \end{aligned}
            \end{equation}
            The model parameters are set as: $\mathrm{Cn}=1,\mathrm{Pe}=2\pi^2,\mathrm{Re}=2\pi^2,\mathrm{We}=1,\mathrm{Fr}=1,$ and $\rho_1/\rho_2=50$, $\eta_1/\eta_2=50$, $\rho_r=\rho_1$, $\eta_r=\eta_1$. To observe the convergence orders, we set time step size $\tau=h$, and refine the spatial mesh size as $h=2^{-i},i=3,4,5,6,7$. For simplicity, we only consider the $L^2$ error for $\phi,\mu,\bm{u},P$ and absolute error for $\xi_1,\xi_2$ at $T=0.5.$

            The results of the numerical errors for both schemes are displayed in Figure~\ref{error}. The expected second-order convergence has been observed for $\phi,\mu,\bm{u},\xi_1$ and $\xi_2$ in both schemes, while a loss of accuracy for $P$ is also presented. Actually, the convergence order of $P$ is approximately $\mathcal{O}(h^{1.5})$ for PG+CM and $\mathcal{O}(h^{1.7})$ for AC+CM, in this test. The super-convergence of the MINI element has already been investigated in many existing works, so that the $\mathcal{O}(h^{1.5})$ super-convergence of pressure $P$ for PG+CM seems reasonable. However, the $\mathcal{O}(h^{1.7})$ super-convergence of pressure $P$ for AC+CM has been beyond our expectation (a linear convergence order is expected), and a good explanation has not been available. Meanwhile, for the phase-field-fluid model with a free-slip boundary condition for the velocity variable, a full $\mathcal{O}(h^2)$ convergence analysis for the pressure variable has been reported in~\cite{chen24c}. A theoretical analysis of the convergence rate for the proposed numerical scheme, in terms of the associated physical variables, will be explored in the future works.

        \subsection{Stability test}
            \begin{figure}[ht]
                \centering
                \includegraphics[width=0.8\linewidth]{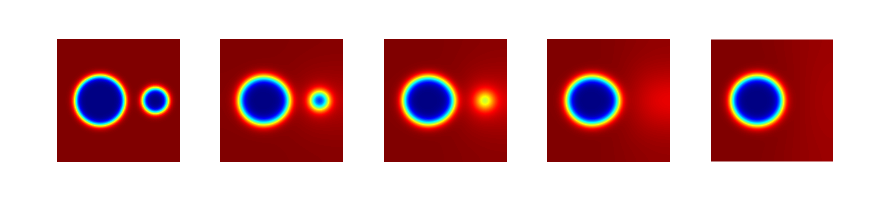}
                \caption{The profiles of $\phi$ computed by the AC+CM scheme, at a sequence of time instants, $t=0,0.4,0.5,0.6$ and $1$.}
                \label{stability_evolution}
            \end{figure}
            To verify the mass conservation and unconditional energy stability for the proposed schemes, we consider a benchmark problem of coarsening process for the Cahn-Hilliard equation \cite{2020_yang}. The computational domain is taken as $\Omega=(0,1)^2$, and the spatial mesh size is chosen to be $h=2^{-7}$. The initial data are set  as
            \begin{equation}
                \phi^0=-1+\tanh\left(\frac{\sqrt{(x-0.35)^2+(y-0.5)^2}-0.2}{\mathrm{Cn}}\right)+\tanh\left(\frac{\sqrt{(x-0.8)^2+(y-0.5)^2}-0.1}{\mathrm{Cn}}\right),
            \end{equation}
            while the other physical variables are taken to zero at $t=0$. The model parameters are given by: $\mathrm{Cn}=3\times{10}^{-2},\mathrm{Re}=100$, $\mathrm{We}=50$,  and $\rho_1/\rho_2=\eta_1/\eta_2=50$, $\rho_r=\rho_1$, $\eta_r=\eta_1$.
            \begin{figure}[ht]
                \centering
                \subfloat[Modified energy computed by PG+CM.]{
                    \includegraphics[width=0.4\linewidth]{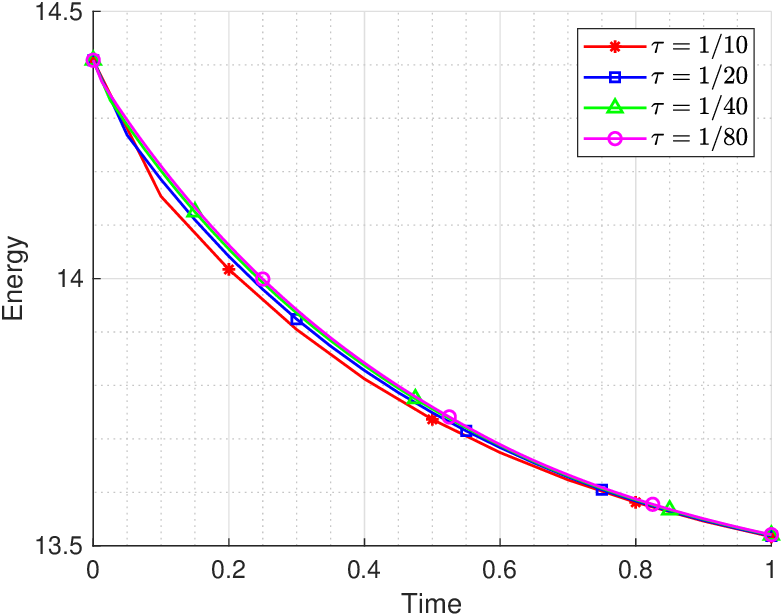}
                    \label{modified_energy_evolution_1}
                }
                \subfloat[Modified energy computed by AC+CM.]{
                    \includegraphics[width=0.4\linewidth]{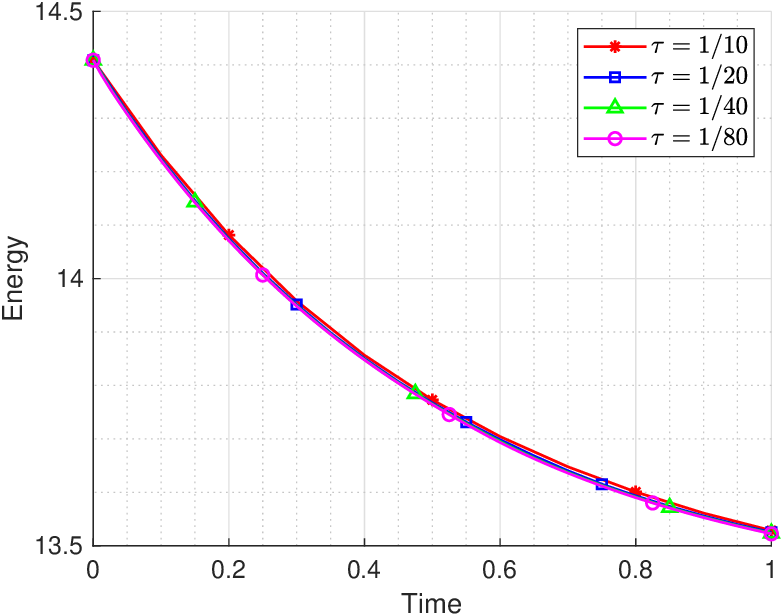}
                    \label{modified_energy_evolution_2}
                }
                \\
                \subfloat[Comparison of energies computed by PG+CM.]{
                    \includegraphics[width=0.4\linewidth]{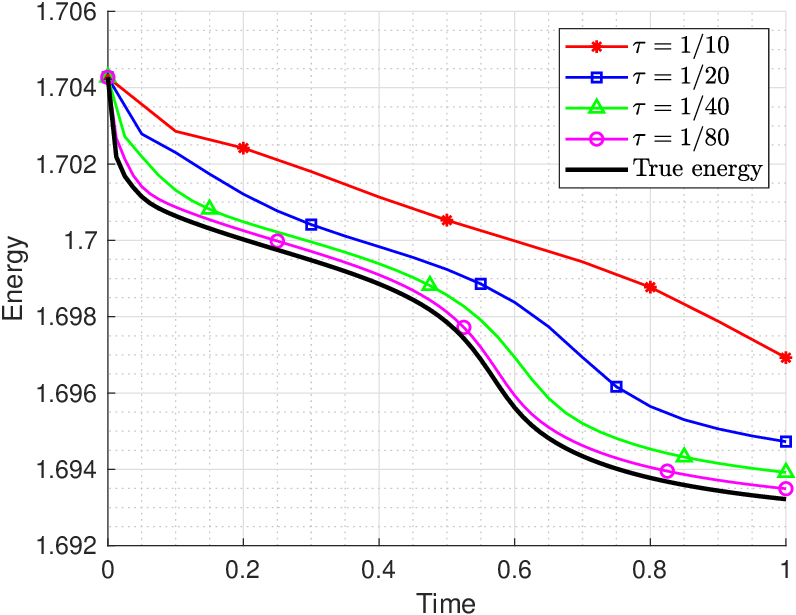}
                    \label{comparison_energy_evolution_1}
                }
                \subfloat[Comparison of energies computed by AC+CM.]{
                    \includegraphics[width=0.4\linewidth]{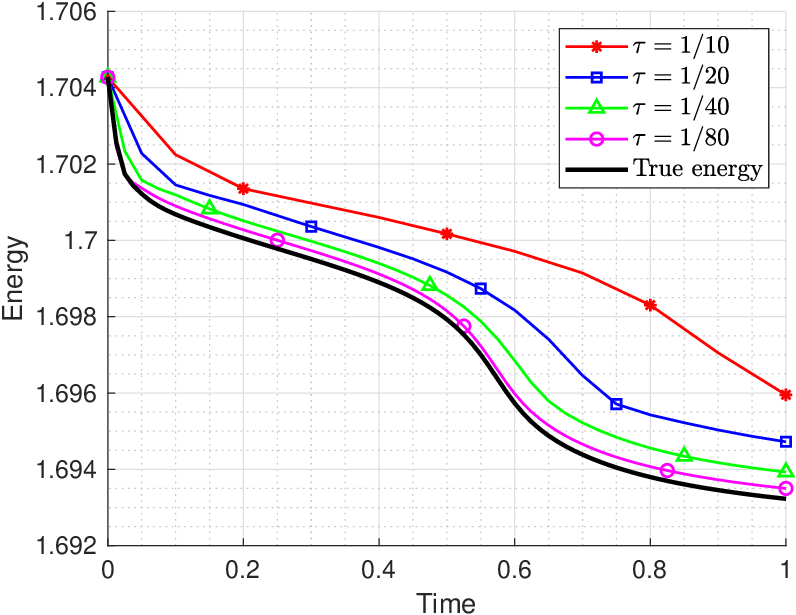}
                    \label{comparison_energy_evolution_2}
                }
                \caption{Time evolution of discrete energy.}
            \end{figure}

            In Figure~\ref{stability_evolution}, the evolution of the order parameter $\phi_h$ at a sequence of time instants is displayed, and the coarsening effect could be clearly observed in the process that the smaller bubble is absorbed into the bigger one. As depicted in Figures~\ref{modified_energy_evolution_1} and \ref{modified_energy_evolution_2}, the discrete modified energy computed by both schemes decays monotonically for different time step sizes, and this fact has numerically verified an unconditional energy stability. Although the modified energy dissipation law \eqref{MCHNS} is equivalent to the original one \eqref{CHNS}, the modified discrete energy in \eqref{discrete_energy_1} and \eqref{discrete_energy_2} does not approximate the true energy in \eqref{true_energy}. If we want to compare them, the extra terms involving $Q^n$ must be removed, and the constant $S$ should be subtracted. Therefore, as depicted in Figures~\ref{comparison_energy_evolution_1} and \ref{comparison_energy_evolution_2}, it can be observed that the modified energy approaches the true energy (computed by $\tau=\frac{1}{80}$) as the time step size is refined. 

            \begin{figure}[ht]
                \centering
                \subfloat[$\xi_1^n.$]{
                    \includegraphics[width=0.3\linewidth]{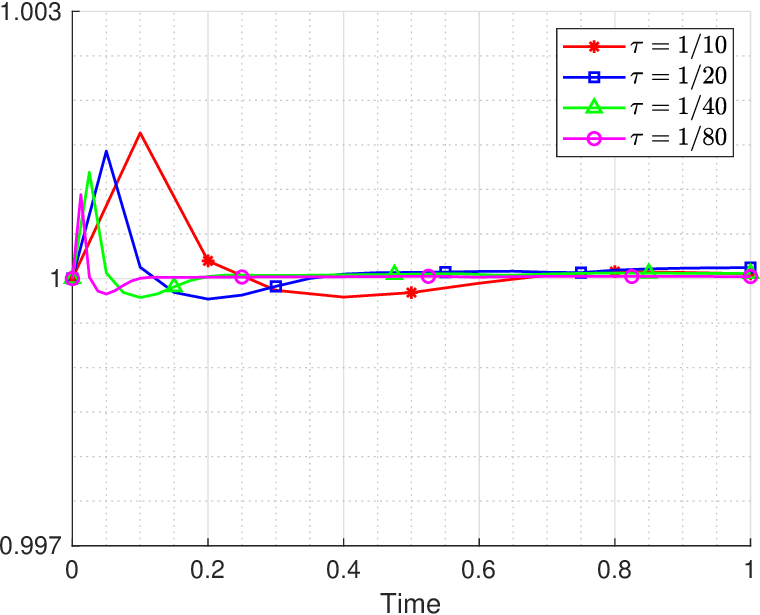}
                }
                \subfloat[$\xi_2^n.$]{
                    \includegraphics[width=0.3\linewidth]{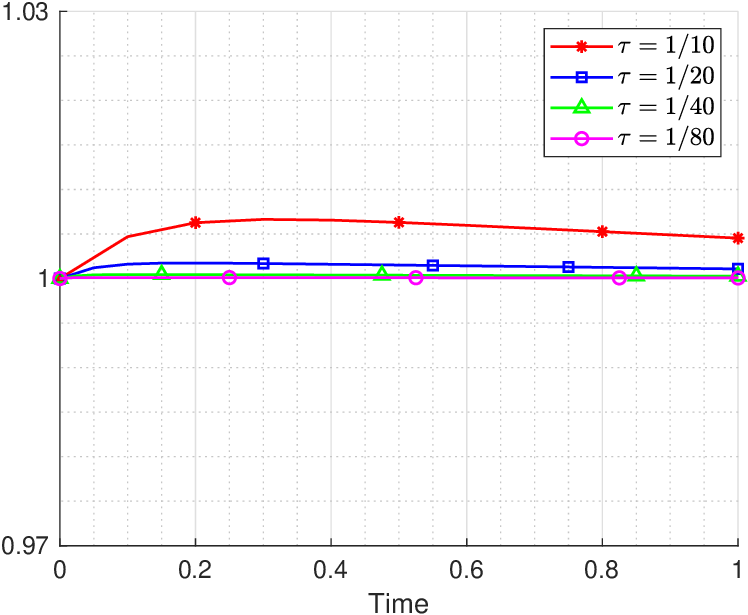}
                }
                \subfloat[$\int_\Omega\phi_h^n\mathrm{d}\bm{x}-\int_\Omega\phi_h^0\mathrm{d}\bm{x}.$]{
                    \includegraphics[width=0.3\linewidth]{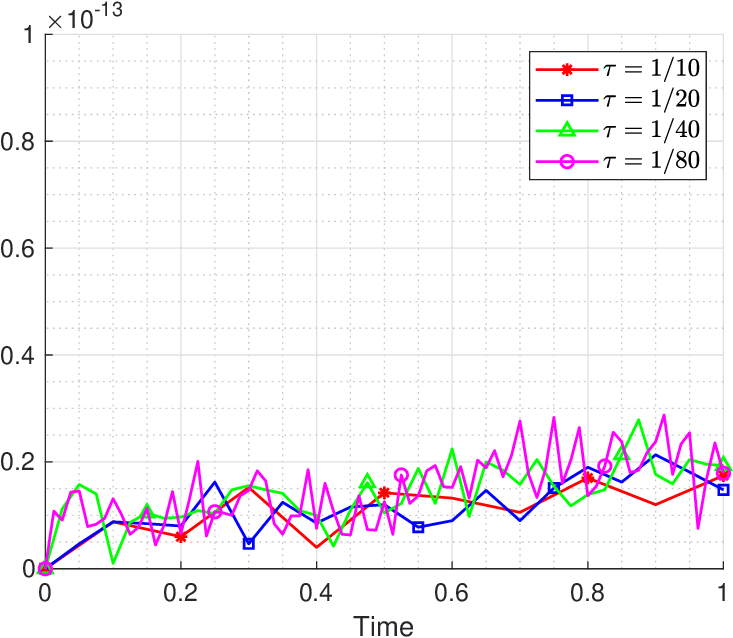}
                }
                \caption{Time evolution of the auxiliary variables and mass drift computed by AC+CM.}
                \label{properity_evolution}
            \end{figure}

            Since the numerical schemes are based on the SAV reformulation, it is imperative to verify the numerical robustness by taking a look at $\xi_1^n$ and $\xi_2^n$, and see whether they remain sufficiently close to 1. In \cite{2019_yang}, where the only existing work of fully decoupled, second-order in time scheme was presented, the authors observed that the auxiliary variables would decrease sharply even with a small time step size in rising air bubble simulations. This fact implies that the simulation is no longer reliable. Meanwhile, a second-order in time finite element scheme is studied in \cite{2021_fu}, while the authors failed to derive a decoupled structure by employing the ZEC feature as the associated auxiliary variables will quickly diminish to a value close to zero in numerical simulations, no matter how small the time step size is chosen.

            In fact, the proposed schemes in this article could circumvent these subtle drawbacks. The time evolution of $\xi_1$ and $\xi_2$ is displayed in Figure~\ref{properity_evolution}; it is clearly observed that both $\xi_1$ and $\xi_2$ are sufficiently close to $1$ even with a large time step size. This numerical result demonstrates the effectiveness of the proposed schemes. In addition, the robustness of $\xi_1$ seems to surpass that of $\xi_2$, possibly due to its physical relevance. Furthermore, the plot of the mass drift, i.e., $\int_\Omega\phi_h^n\mathrm{d}\bm{x}-\int_\Omega\phi_h^0\mathrm{d}\bm{x}$, is presented in Figure~\ref{properity_evolution}, and an excellent mass conservation could be observed.

        \subsection{Capillary wave}
            \begin{table}[ht]
                \centering
                \caption{Physical parameters and corresponding dimensionless numbers for the capillary wave benchmark problem.}
                \begin{tabular}{llllllll}
                \hline
                $\rho_1/\rho_2$ & $\nu$ & $\rho_1(\rho_r)$ & $\eta_1(\eta_r)$ & $g$ & $\lambda$ & $\mathrm{Re}$ & $\mathrm{We}$            \\
                \hline
                10              & 0.01  &10                & 0.1              & 1   & 1         & $100$         & $\frac{20\sqrt{2}}{3}$   \\
                100             & 0.01  &100               & 1                & 1   & 1         & $100$         & $\frac{200\sqrt{2}}{3}$  \\
                1000            & 0.01  &1000              & 10               & 1   & 1         & $100$         & $\frac{2000\sqrt{2}}{3}$ \\
                \hline
                \end{tabular}
                \label{parameters_capillary_wave}
            \end{table}
        
            The capillary wave problem \cite{1981_prosperetti}, for which an analytic solution is available, is investigated in this subsection. In this benchmark problem, the lighter fluid is placed on top of the heavier one, and the interface, initially perturbed by a sinusoidal function with a small amplitude $H_0$, wave number $k$ and zero initial velocity, will begin oscillating under the influence of gravity and surface tension. The analytic solution is derived in \cite{1981_prosperetti} for two fluids having the same kinematic vsicosity $\nu=\frac{\eta_1}{\rho_1}=\frac{\eta_2}{\rho_2}$ in an infinite domain. In turn, the evolution of the capillary wave amplitude $H(t)$ has the following form
            \begin{equation}
                \label{capillary_wave_analytic}
                \frac{H(t)}{H_0}=\frac{4(1-4\beta)\nu^2k^4}{8(1-4\beta)\nu^2k^4+\omega_0^2}\mathrm{erfc}\left(\sqrt{\nu{k}^2t}\right)+\sum_{i=1}^4\frac{z_i}{Z_i}\frac{\omega_0^2}{z_i^2-\nu{k}^2}e^{\left(z_i^2-\nu{k}^2\right)t}\mathrm{erfc}\left(z_i\sqrt{t}\right),
            \end{equation}
            where $\omega_0^2=\frac{(\rho_1-\rho_2)gk+\lambda{k}^3}{\rho_1+\rho_2},\beta=\frac{\rho_1\rho_2}{(\rho_1+\rho_2)^2}$, and $\mathrm{erfc}(\cdot)$ is the complementary error function. The variables $z_i,i=1,\cdots,4,$ are the four roots of the algebraic equation
            \begin{equation}
                z^4-4\beta\sqrt{\nu{k}^2}z^3+2(1-6\beta)\nu{k}^2z^2+4(1-3\beta)(\nu{k}^2)^\frac{3}{2}z+(1-4\beta)\nu^2k^4+\omega_0^2=0,
            \end{equation}
            and $Z_i=\prod_{1\le{j}\le{4},j\neq{i}}(z_j-z_i),i=1,\cdots,4.$
            The motion of the interface is simulated in a rectangular domain $\Omega=(0,1)\times(-1,1)$, with no-slip boundaries on the horizontal walls and periodic boundaries on the vertical walls. The initial perturbed interface is described by ($H_0=0.01,k=2\pi$):
            \begin{equation}
                \phi^0 (x,y)=-\tanh\left(\frac{y+0.01\cos{2\pi{x}}}{\mathrm{Cn}}\right).
            \end{equation}
            The other physical variables are set to be zero at the initial time step. The relevant physical parameters and corresponding dimensionless numbers are listed in Table \ref{parameters_capillary_wave}. Since the interface dynamics is only concentrated in a narrow region, we consider a mesh that is locally refined on $\Omega_1=(0,1)\times(-0.05,0.05)$, where $h=2^{-8}$ on $\Omega_1$ and $h=2^{-4}$ on $\Omega\backslash\Omega_1$. Meanwhile, $\mathrm{Cn}=1\times{10}^{-2}$ with $\tau=5\times{10}^{-4}$ is taken in all cases.

            Figure~\ref{capillary_wave_amplitude} displays time evolutions of capillary wave amplitude of different density ratios, in comparison with the analytic solution \eqref{capillary_wave_analytic} and numerical results from \cite{2019_yang} (the only fully decoupled scheme for the AGG model, and the results were obtained by high order spectral elements and very small time steps). Both schemes exhibit the same amplitude histories and agree well with the reference data.

            \begin{figure}[ht]
                \centering
                \subfloat[$\rho_1/\rho_2=10$.]{
                    \includegraphics[width=0.3\linewidth]{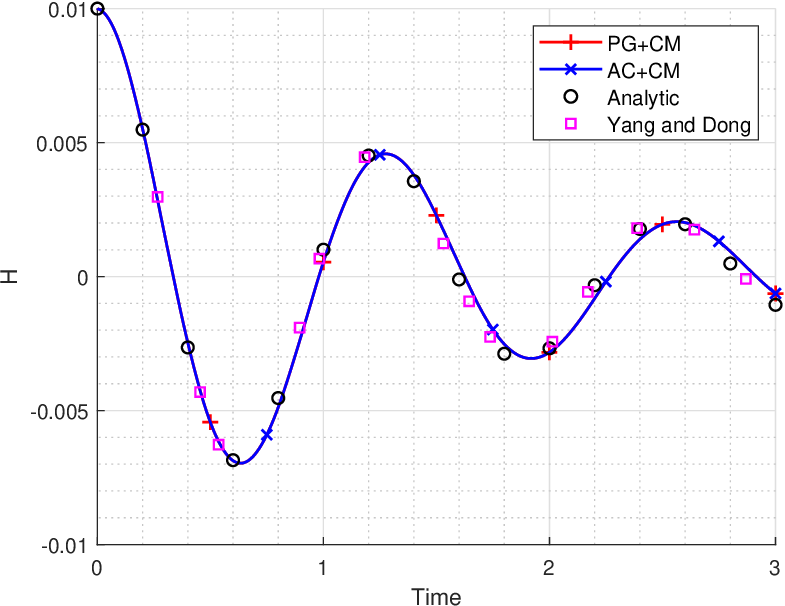}
                }
                \subfloat[$\rho_1/\rho_2=100$.]{
                    \includegraphics[width=0.3\linewidth]{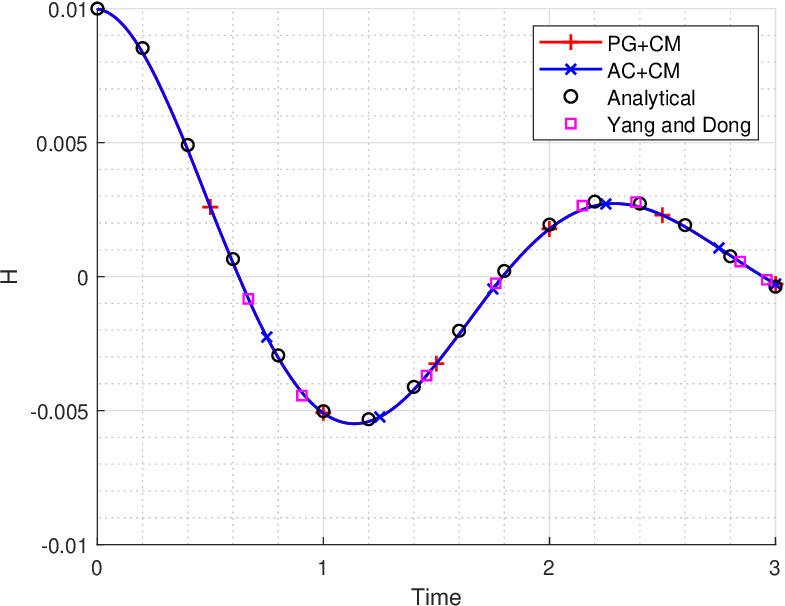}
                }
                \subfloat[$\rho_1/\rho_2=1000$.]{
                    \includegraphics[width=0.3\linewidth]{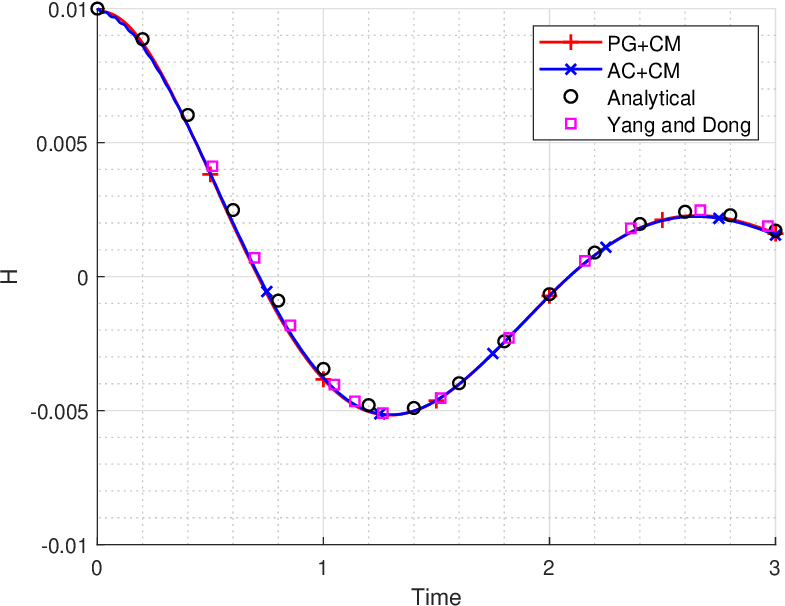}
                }
                \caption{Comparison of capillary wave amplitude versus time.}
                \label{capillary_wave_amplitude}
            \end{figure}

        \subsection{Single rising bubble}
        \label{benckmark_bubble}
            \begin{table}[ht]
                \centering
                \caption{Physical parameters and corresponding dimensionless numbers for the single rising bubble benchmark problem.}
                \begin{tabular}{lllllllll}
                \hline
                Test case & $\rho_1/\rho_2$ & $\rho_1(\rho_r)$ & $\eta_1/\eta_2$ & $\eta_1(\eta_r)$ & $g$  & $\lambda$ & $\mathrm{Re}$ & $\mathrm{We}$            \\
                \hline
                1         & 10              &1000              & 10              &10                & 0.98 & 24.5      & $70\sqrt{2}$  & $\frac{80\sqrt{2}}{3}$   \\
                2         & 1000            &1000              & 100             &10                & 0.98 & 1.96      & $70\sqrt{2}$  & $\frac{1000\sqrt{2}}{3}$ \\
                \hline
                \end{tabular}
                \label{parameters_bubble_rising}
            \end{table}
            In this subsection, we consider the benchmark problem of a single rising bubble in a liquid column proposed in \cite{2009_hysing}, an interplay between surface tension and buoyancy. The domain is $(0,1) \times (0,2)$, and a circular bubble is centered at $(0.5 ,0.5)$ with a radius of $0.25$, which corresponds to the following initial data for $\phi$:
            \begin{equation}
                \phi^0=\tanh \Big( \frac{\sqrt{(x-0.5)^2+(y-0.5)^2}-0.25}{\mathrm{Cn}} \Big) .
            \end{equation}
            The other physical variables are set to zero at the initial time step. Following the original set-up for Navier-Stokes equations, we choose no-slip boundaries on the horizontal walls, and free-slip boundaries on the vertical walls. The relevant physical parameters used in \cite{2009_hysing} and corresponding dimensionless numbers are listed in Table \ref{parameters_bubble_rising}. And to have a comprehensive understanding of bubble dynamics, the following benchmark quantities proposed in \cite{2009_hysing} are considered:
            \begin{itemize}
                \item Centroid $y_c=\displaystyle\frac{\int_{\phi<0}y\mathrm{d}\bm{x}}{\int_{\phi<0}\mathrm{d}\bm{x}}$
                \item Circularity \rlap{/}$c=\displaystyle\frac{2\sqrt{\int_{\phi<0}\pi\mathrm{d}\bm{x}}}{\int_{\phi=0}\mathrm{d}s}$
                \item Rise velocity $V_c=\displaystyle\frac{\int_{\phi<0}u_2\mathrm{d}\bm{x}}{\int_{\phi<0}\mathrm{d}\bm{x}}$
            \end{itemize}

            In both cases, we take $h=2^{-6},\tau=2\times{10}^{-3}$ with $\mathrm{Cn}=1\times{10}^{-2},$ or $h=2^{-7},\tau=1\times{10}^{-3}$ with $\mathrm{Cn}=5\times{10}^{-3}$. To investigate the impact of mobility function on interface dynamics, we choose $1/\mathrm{Pe}=0.1$ for $m=(\phi^2-1)^2$. In addition, the data provided by group 3 in \cite{2009_hysing}, which solves the sharp interface model using an arbitrary Lagrangian–Eulerian finite element method, is taken as reference.

        \subsubsection{Test case 1}
            \begin{figure}[ht]
                \centering
                \includegraphics[width=0.4\linewidth]{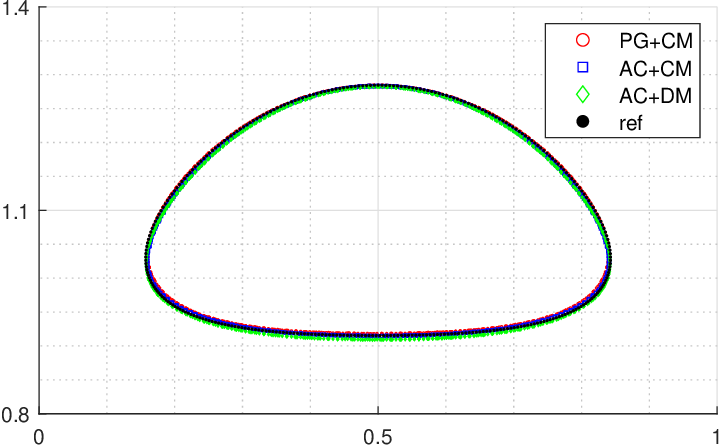}
                \caption{Comparison of the bubble shapes, computed by $h=2^{-7}$, for test case 1 at dimensional time $t^*=3$.}
                \label{bubble_shape_1}
            \end{figure}
            This test considers a low $\mathrm{We}$ number, which corresponds to a high surface tension and less deformation of the bubble shape. This shape at dimensional time $t^*=3$, the contour line of $\phi_h=0$, is depicted in Figure~\ref{bubble_shape_1}, in comparison with the reference data. All the numerical schemes create a highly similar ellipsoidal bubble shape, and have an excellent agreement with the reference data. To further confirm the numerical accuracy, we also present the time evolution plot (with respect to dimensional time) of benchmark quantities in Figure~\ref{benchmark_quantities_1}. Quantitative comparison with the benchmark values is displayed in Table \ref{benchmark_quantities_data_1}. The centroid and rise velocity exhibit an excellent agreement with the reference data, while only a minor deviation is observed in terms of the circularity. Based on this numerical investigation, it can be concluded that, the form of mobility has a negligible impact on these benchmark quantities with low $\mathrm{We}$ numbers (high surface tensions). Therefore, a constant mobility is significantly more efficient than a degenerate one in such a low $\mathrm{We}$ number region.
            \begin{figure}[ht]
                \centering
                \subfloat[Circularity.]{
                    \includegraphics[width=0.3\linewidth]{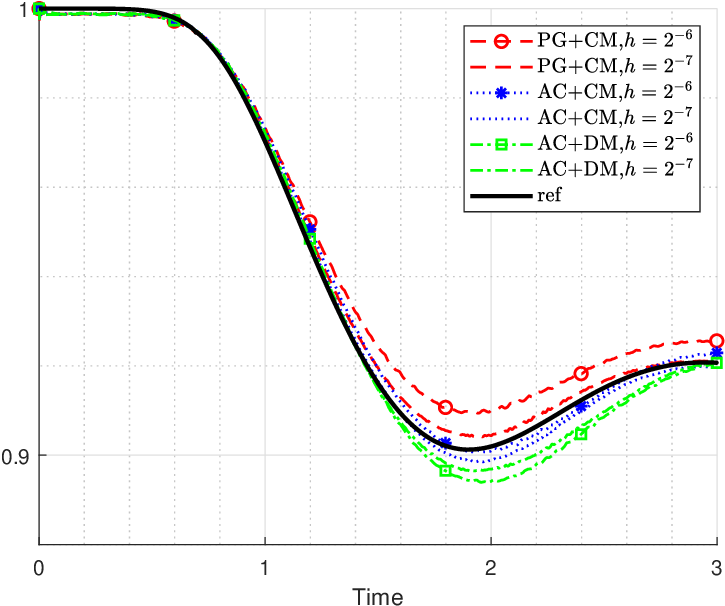}
                }
                \subfloat[Centroid.]{
                    \includegraphics[width=0.3\linewidth]{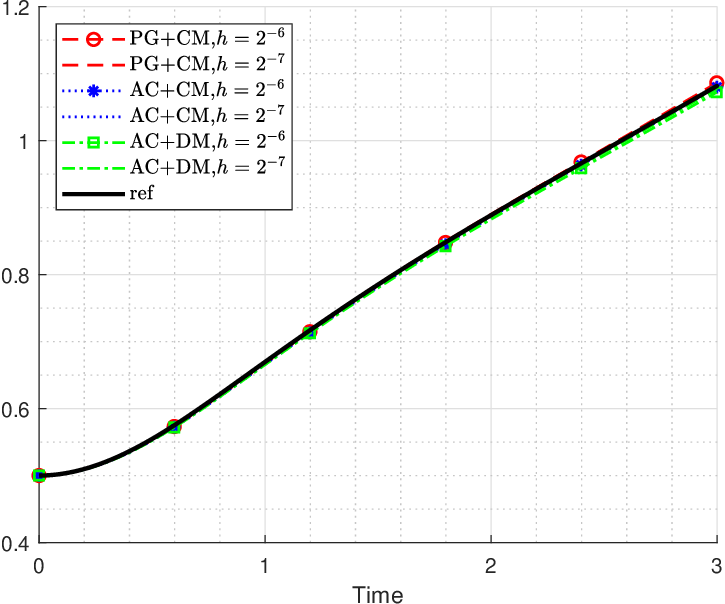}
                }
                \subfloat[Rise velocity.]{
                    \includegraphics[width=0.3\linewidth]{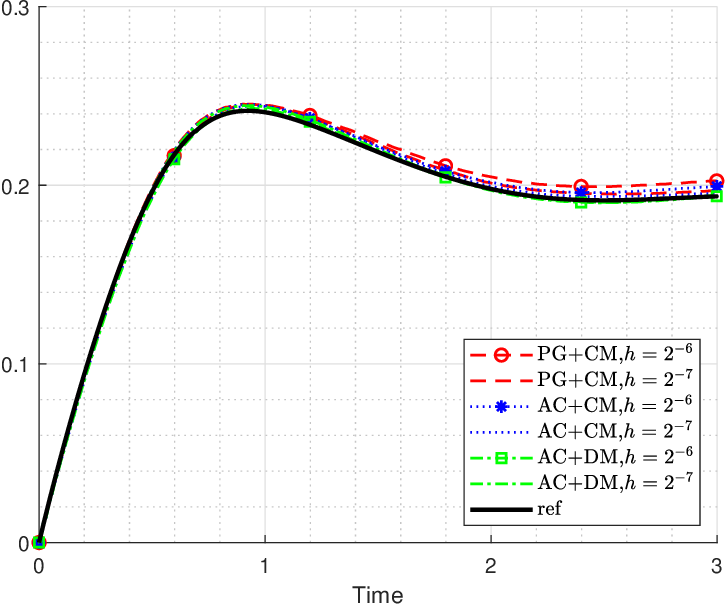}
                }
                \caption{Time evolution of benchmark quantities for test case 1.}
                \label{benchmark_quantities_1}
            \end{figure}
            \begin{table}[ht]
                \centering
                \caption{Minimum circularity and maximum rise velocity, with corresponding incidence time instants and final position of the centroid for test case 1.}
                \begin{tabular}{lllllll}
                \hline
                type                    & $h$         & \rlap{/}$c_{\mathrm{min}}$ & $t^*|_{\text{\rlap{/}}c=\text{\rlap{/}}c_{\mathrm{min}}}$ & $V_{c,\mathrm{max}}$ & $t^*|_{V_{c}=V_{c,\mathrm{max}}}$ & $y_c(t^*=3)$ \\
                \hline
                \multirow{2}{*}{PG+CM}  & $2^{-6}$    & 0.9094                     & 1.9415                                                  & 0.2451               & 0.9576                          & 1.0862       \\
                                        & $2^{-7}$    & 0.9041                     & 1.9637                                                  & 0.2455               & 0.9384                          & 1.0841       \\
                \multirow{2}{*}{AC+CM}  & $2^{-6}$    & 0.9007                     & 1.9254                                                  & 0.2443               & 0.9534                          & 1.0798       \\
                                        & $2^{-7}$    & 0.8985                     & 1.9637                                                  & 0.2453               & 0.9405                          & 1.0808       \\
                \multirow{2}{*}{AC+DM}  & $2^{-6}$    & 0.8939                     & 1.9557                                                  & 0.2430               & 0.9455                          & 1.0722       \\
                                        & $2^{-7}$    & 0.8964                     & 1.9041                                                  & 0.2446               & 0.9405                          & 1.0773       \\
                \multicolumn{1}{c}{ref} & $\diagdown$ & 0.9013                     & 1.9000                                                  & 0.2417               & 0.9239                          & 1.0817       \\
                \hline
                \end{tabular}
                \label{benchmark_quantities_data_1}
            \end{table}

        \subsubsection{Test case 2}
            \begin{figure}[ht]
                \centering
                \includegraphics[width=0.4\linewidth]{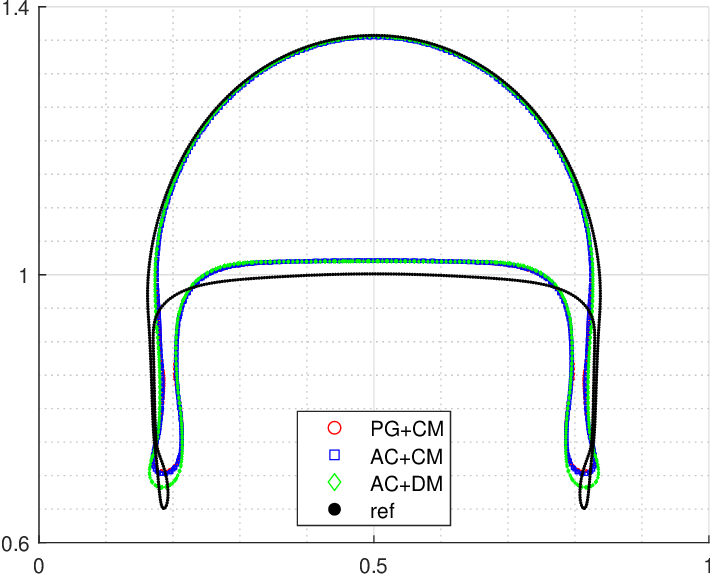}
                \caption{Comparison of the bubble shapes, computed by $h=2^{-7}$, for test case 2 at dimensional time $t^*=3$.}
                \label{bubble_shape_2}
            \end{figure}

            A low surface tension is considered in this test, resulting in high deformation of bubble shape. This shape at dimensional time $t^*=3$ is depicted in Figure \ref{bubble_shape_2}, in comparison with the reference data. All schemes exhibit a similar skirted bubble shape, and the phenomenon of break off is absent. Also, the tails of our bubbles are thicker than the reference one. In addition, Figure \ref{benchmark_quantities_2} displays the time evolution (with respect to dimensional time) of benchmark quantities, and Table \ref{benchmark_quantities_data_2} provides a quantitative comparison with the benchmark values. Again, the centroid exhibits a nice agreement with the reference data for all schemes. As the mesh size is further refined, the circularity and rise velocity of all schemes converge to the same solution, which is distinct from the reference one. It is also observed that, the rise velocity computed by AC exhibits an oscillating behavior. Although the differences in bubble shape and benchmark quantities computed by $h=2^{-7}$ with different types of mobility are negligible, degenerate mobility exhibits superior performance in interface dynamics, with a coarser mesh. Therefore, in the case of high interface deformation, it is advisable to employ degenerate mobility to more accurately capture the large deformation, especially with a low spatial resolution.
            \begin{figure}[ht]
                \centering
                \subfloat[Circularity.]{
                    \includegraphics[width=0.3\linewidth]{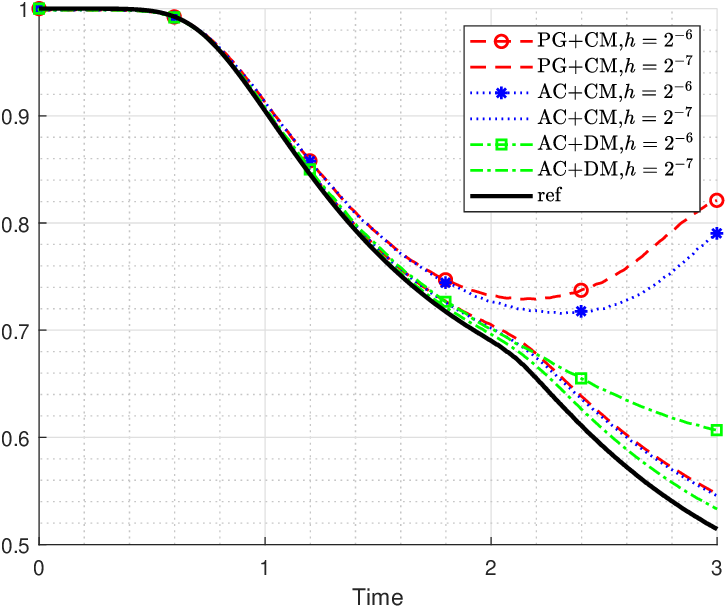}
                }
                \subfloat[Centroid.]{
                    \includegraphics[width=0.3\linewidth]{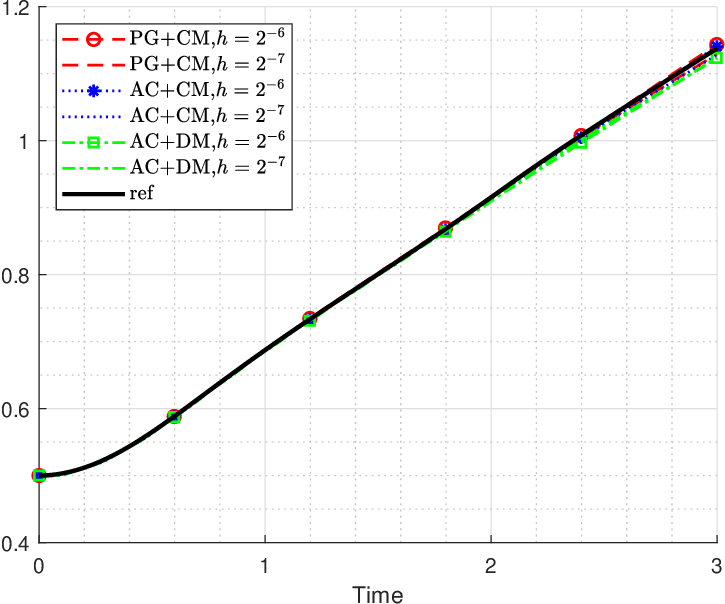}
                }
                \subfloat[Rise velocity.]{
                    \includegraphics[width=0.3\linewidth]{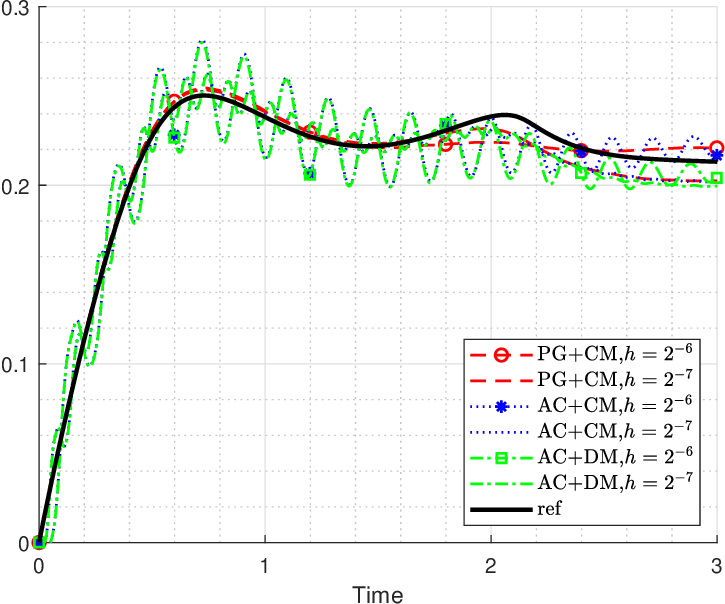}
                }
                \caption{Time evolution of benchmark quantities for test case 2.}
                \label{benchmark_quantities_2}
            \end{figure}
            \begin{table}[ht]
                \centering
                \caption{Minimum circularity and maximum rise velocity, with corresponding incidence time instants and final position of the centroid for test case 2.}
                \begin{tabular}{lllllll}
                \hline
                type                    & $h$         & \rlap{/}$c_{\mathrm{min}}$ & $t^*|_{\text{\rlap{/}}c=\text{\rlap{/}}c_{\mathrm{min}}}$ & $V_{c,\mathrm{max}}$ & $t^*|_{V_{c}=V_{c,\mathrm{max}}}$ & $y_c(t^*=3)$ \\
                \hline
                \multirow{2}{*}{PG+CM}  & $2^{-6}$    & 0.7290                     & 2.1395                                                  & 0.2545               & 0.7354                          & 1.1433       \\
                                        & $2^{-7}$    & 0.5474                     & 3.0000                                                  & 0.2533               & 0.7223                          & 1.1294       \\
                \multirow{2}{*}{AC+CM}  & $2^{-6}$    & 0.7157                     & 2.3072                                                  & 0.2815               & 0.7214                          & 1.1404       \\
                                        & $2^{-7}$    & 0.5455                     & 3.0000                                                  & 0.2623               & 0.7435                          & 1.1280       \\
                \multirow{2}{*}{AC+DM}  & $2^{-6}$    & 0.6067                     & 3.0000                                                  & 0.2798               & 0.7213                          & 1.1235       \\
                                        & $2^{-7}$    & 0.5332                     & 3.0000                                                  & 0.2617               & 0.7435                          & 1.1222       \\
                \multicolumn{1}{c}{ref} & $\diagdown$ & 0.5144                     & 3.0000                                                  & 0.2502               & 0.7317                          & 1.1376       \\
                \hline
                \end{tabular}
                \label{benchmark_quantities_data_2}
            \end{table}

        \subsubsection{Non-physical acoustic waves}
        \label{acoustic_waves}
            \begin{figure}[ht]
                \centering
                \includegraphics[width=0.4\linewidth]{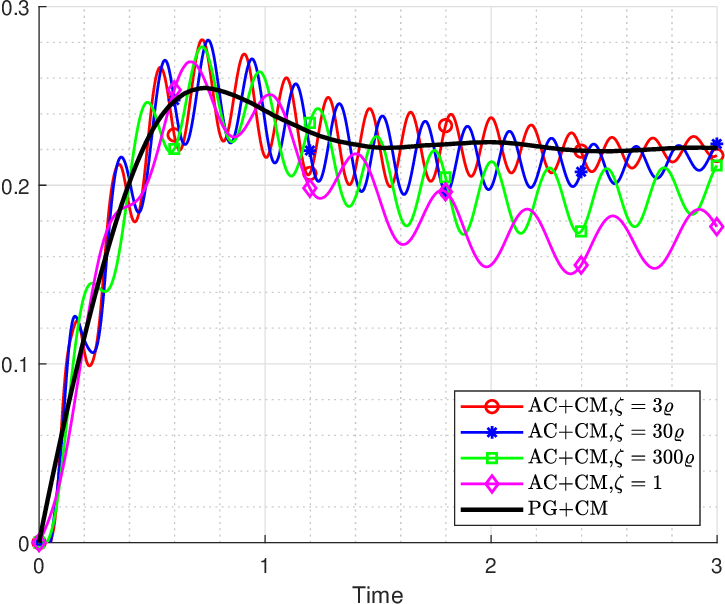}
                \caption{Time evolution of rise velocity, computed by $h=2^{-6}$, with different $\zeta$ for test case 2.}
                \label{rise_velocity_refined}
            \end{figure}
            Instead of the numerical boundary layer introduced by pressure projection method, the artificial compressibility approach may introduce non-physical acoustic waves, with wave speed $c\simeq\frac{\sqrt{\gamma_0\varrho}}{\tau}$, as observed in the time evolution of rise velocity in test case 2. The manifestation of this phenomenon is particularly conspicuous in the presence of a large density ratio; the velocity oscillation was not observed in test case 1. In fact, as mentioned in the work of DeCaria et al. \cite{2017_decaria}, increasing $\zeta$ could slow the nonphysical acoustic waves to the point that the CFL condition is satisfied. However, this treatment does not yield prominent improvement in the proposed schemes in this article, as demonstrated in Figure~\ref{rise_velocity_refined}. We will explore potential remedies for this issue in the AC formulation in the future work.

            Now we summarize several approaches in the existing literature to overcome this formidable challenge. The time filter technique was advised in \cite{2017_decaria} to control non-physical acoustics. In the work of Sitompul and Aoki \cite{2019_sitompul}, the authors observed that the oscillation amplitude could be suppressed by reducing the Mach number, which corresponds to a decrease of time step size; this phenomenon has also been observed in our results. Furthermore, the curve exhibits a gradual smoothing trend over time, indicating the suppression of non-physical acoustic waves. In the subsequent works~\cite{2021_matsushita, 2021_yang}, the authors provided additional evidence that these oscillations arise from the initial pressure imbalance. They successfully eliminated these oscillations by employing the evolving pressure projection method \cite{2021_yang}, or solving the pressure Poisson equation in the first several steps \cite{2021_matsushita} to balance the initial pressure with the gravity. Meanwhile, incorporating their approaches into the structure-preserving design of artificial compressibility scheme poses significant challenges.

        \subsection{Rayleigh-Taylor instability}
            \begin{table}[ht]
                \centering
                \caption{Physical parameters and corresponding dimensionless numbers for the Rayleigh-Taylor instability benchmark problem.}
                \begin{tabular}{lllllll}
                \hline
                $\mathrm{At}$ & $\rho_1(\rho_r)$ & $\eta_1(\eta_r)$       & $g$ & $\lambda$ & $\mathrm{Re}$ & $\mathrm{We}$             \\
                \hline
                3             & 3                &$\frac{\sqrt{2}}{1000}$ & 2   & $0.01$    & $3000$        & $400\sqrt{2}$             \\
                7             & 7                &$\frac{\sqrt{2}}{1000}$ & 2   & $0.01$    & $7000$        & $\frac{2800\sqrt{2}}{3}$  \\
                \hline
                \end{tabular}
                \label{parameters_rayleigh_taylor}
            \end{table}
        
            In this subsection, we compute the benchmark problem of Rayleigh-Taylor instability \cite{2000_guermond}, which is characterized by large interface deformation and high $\mathrm{Reynolds}$ numbers. Assuming negligible viscosity and surface tension, buoyancy predominantly dominates its evolution. The Atwood number, defined as $\mathrm{At}=\frac{\rho_1-\rho_2}{\rho_1+\rho_2},$ is commonly used in this benchmark problem to parametrize the dependence on density ratio. We choose two density ratio ($\rho_1/\rho_2$) values, 3 and 7, which correspond to $\mathrm{At}=0.5$ and 0.75, while the viscosity ratio is taken to be 1. The domain $\Omega$ is set as $(-0.5,0.5) \times (-1.5,1)$,  and the initial interface is described by the following profile:
            \begin{equation}
                \phi^0 (x,y)=\tanh\left(\frac{y+0.1\cos{2\pi{x}}}{\mathrm{Cn}}\right).
            \end{equation}
            The other physical variables are set to zero at the initial time step. The same boundary conditions and the mesh sizes, as in the single rising bubble problem, are used.  The relevant physical parameters used in \cite{2000_guermond} (a very small surface tension is set in our simulations) and corresponding dimensionless numbers are listed in Table \ref{parameters_rayleigh_taylor}. In both cases, we take $h=2^{-7},\tau=1\times{10}^{-3}$ with $\mathrm{Cn}=5\times{10}^{-3}$. To capture the vortex shape, a high spatial resolution is needed, so the Taylor-Hood $\mathbb{P}_2\times\mathbb{P}_1$ element is used for $(\bm{u}_h,P_h)$. In addition, we choose $1/\mathrm{Pe}=0.01$ for $m=(\phi^2-1)^2$. Moreover, it is worth noting that the vortex dynamics are captured in a degenerate mobility with very small surface tension $\lambda$ ($\mathcal{O}(10^{-4})$) and reference mobility $M$ ($\mathcal{O}(10^{-3})$), in many literature \cite{2020_fu, 2021_fu}. Such a parameter  requires a very high resolution, both in space and time, to get a satisfied result. Hence, moderate $\lambda$ and $M$ ($\mathcal{O}(10^{-2})$) are chosen in our numerical implementation.

            \subsubsection{A low \texorpdfstring{$\mathrm{At}$}{At} number case}
           
                The contour plots of $\phi_h$ at different dimensional time instants are depicted in Figure~\ref{profile_RTI_low}. It is observed that the solutions computed by different schemes exhibit similar structure, and the result computed with a degenerate mobility provides more details than the one computed with a constant mobility. A comparison of positions of rising and falling bubbles against reference data, taken from Fu \cite{2020_fu} (another phase-field model) and a previous work \cite{2021_li} (variable density incompressible flows), is depicted in Figure~\ref{position_1}. We observe that the results obtained from current work match the reference data very well.

                \begin{figure}[ht]
                    \centering
                    \subfloat[PG+CM.]{
                        \includegraphics[width=0.3\linewidth]{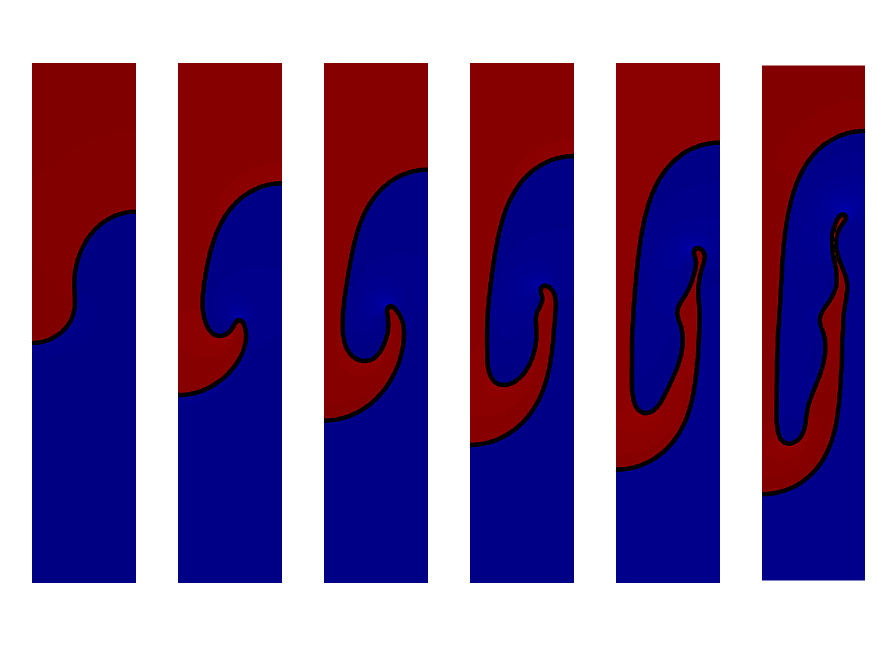}
                    }
                    \\
                    \subfloat[AC+CM.]{
                        \includegraphics[width=0.3\linewidth]{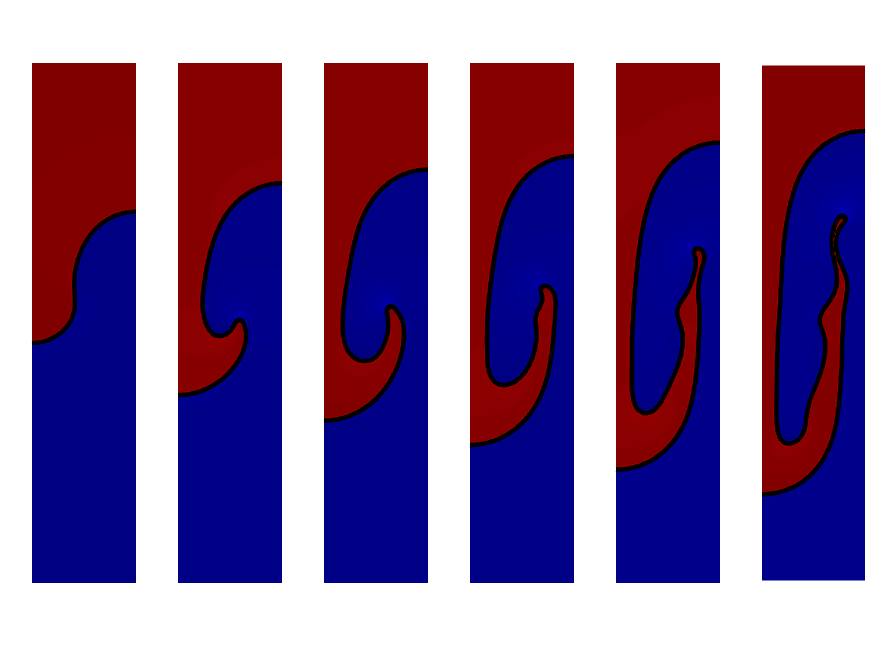}
                    }
                    \\
                    \subfloat[AC+DM.]{
                        \includegraphics[width=0.3\linewidth]{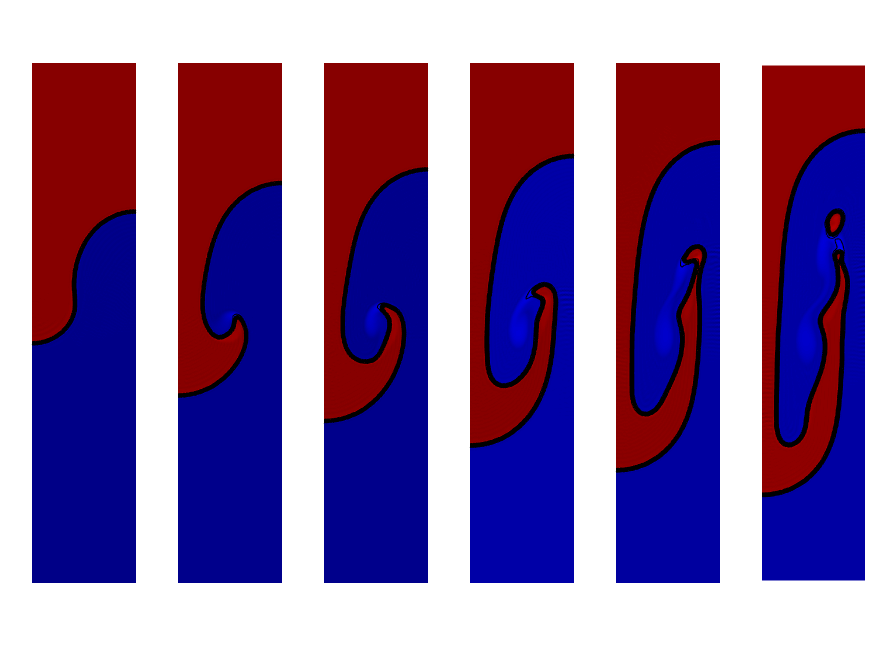}
                    }
                    \caption{Contour plot of the Rayleigh-Taylor instability problem with $\mathrm{At}=0.5$, at a sequence of dimensional time instants, $t^*=1,1.5,1.75,2,2.25,2.5$ (from left to right).}
                    \label{profile_RTI_low}
                \end{figure}

                \begin{figure}[ht]
                    \centering
                    \includegraphics[width=0.4\linewidth]{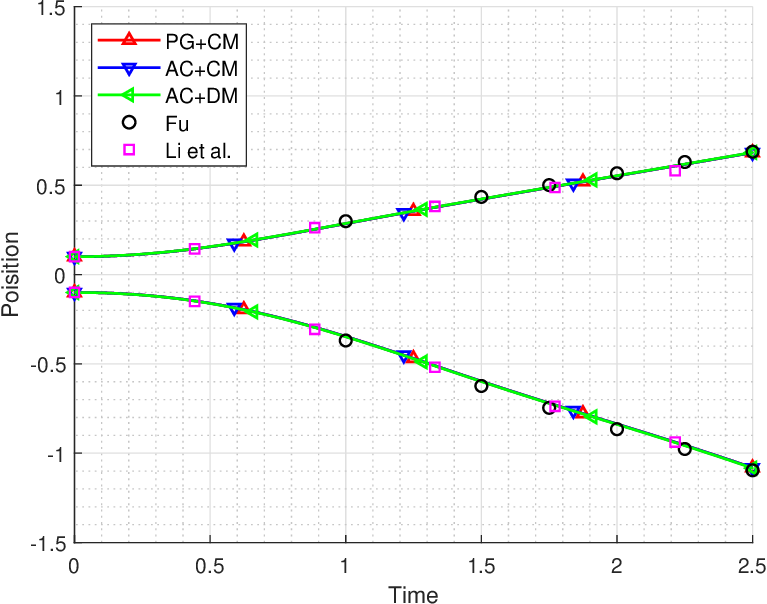}
                    \caption{Comparison of rising and falling bubble positions versus dimensional time with $\mathrm{At}=0.5$.}
                    \label{position_1}
                \end{figure}

            \subsubsection{A high \texorpdfstring{$\mathrm{At}$}{At} number case}
            
                The contour plots of $\phi_h$ at different dimensional time instants are depicted in Figure~\ref{profile_RTI_high}. Similarly, it can be observed that the solutions computed by different schemes indicate similar structures, while the result computed with a degenerate mobility exhibit more details than the one computed with a constant mobility. A comparison in terms of rising and falling bubble positions is depicted in Figure~\ref{position_2}, in which the reference data are taken from a previous work \cite{2021_li}. Again, the results obtained from the current work match the reference data very well.

                \begin{figure}[ht]
                    \centering
                    \subfloat[PG+CM.]{
                        \includegraphics[width=0.3\linewidth]{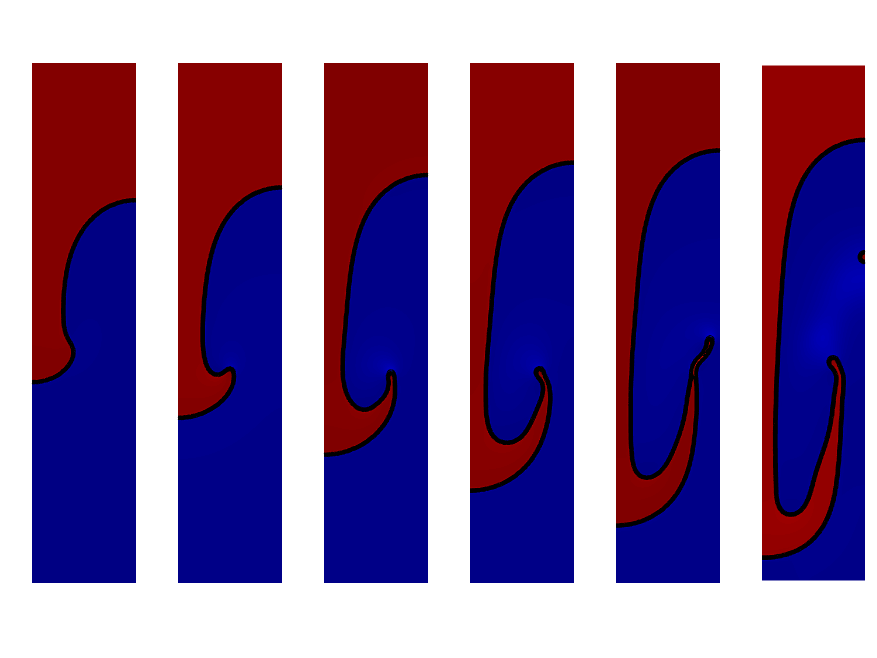}
                    }
                    \\
                    \subfloat[AC+CM.]{
                        \includegraphics[width=0.3\linewidth]{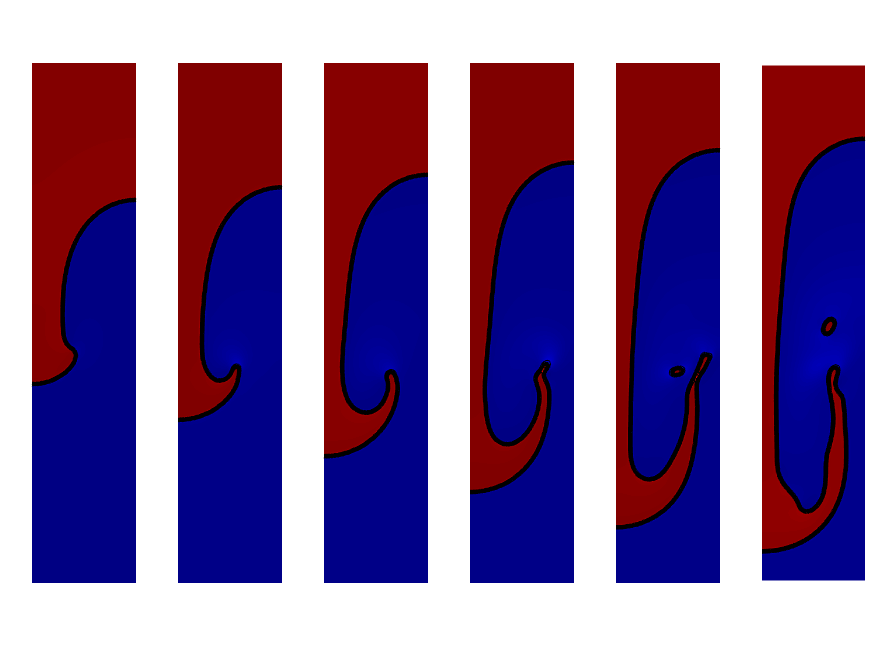}
                    }
                    \\
                    \subfloat[AC+DM.]{
                        \includegraphics[width=0.3\linewidth]{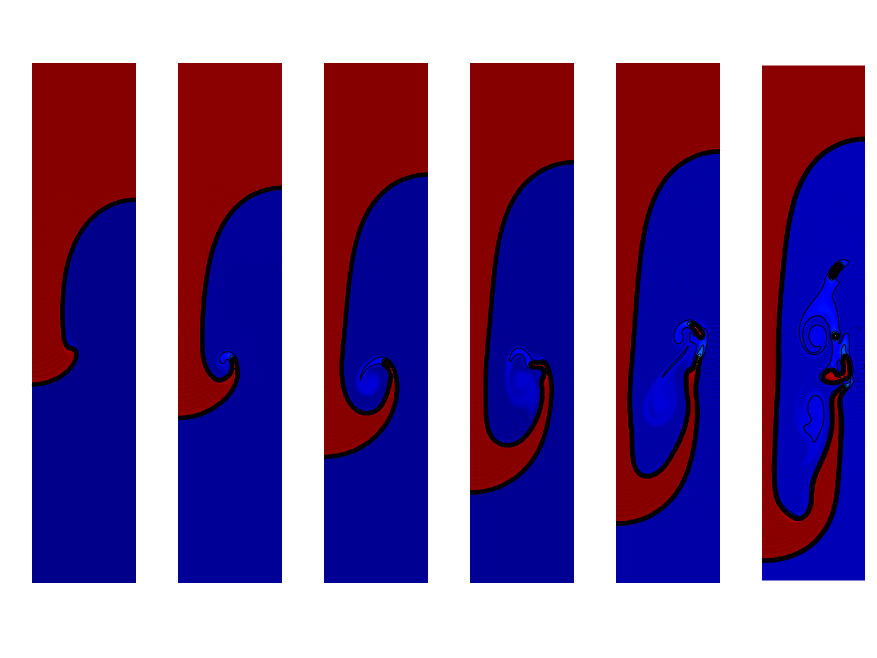}
                    }
                    \caption{Contour plot of the Rayleigh-Taylor instability problem with $\mathrm{At}=0.75$, at a sequence of dimensional time instants, $t^*=1,1.2,1.4,1.6,1.8,2$ (from left to right).}
                    \label{profile_RTI_high}
                \end{figure}
                \begin{figure}[ht]
                    \centering
                    \includegraphics[width=0.4\linewidth]{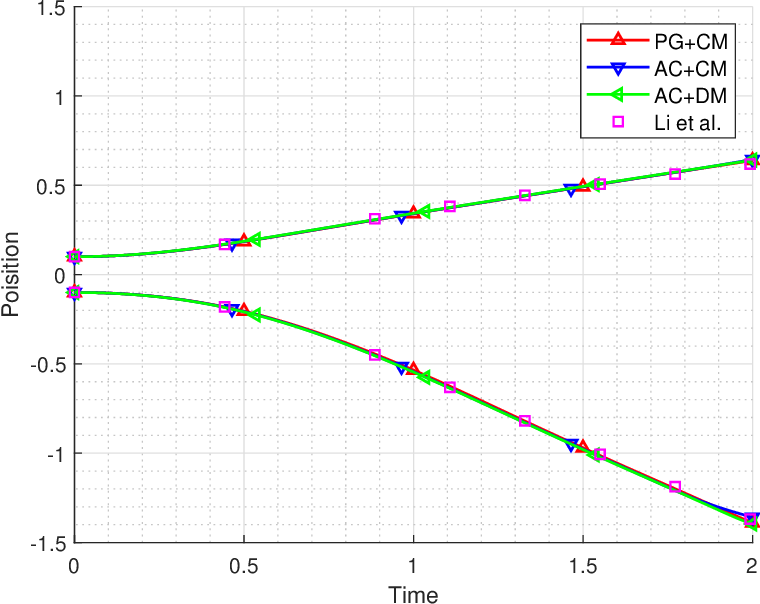}
                    \caption{Comparison of rising and falling bubble positions versus dimensional time with $\mathrm{At}=0.75$.}
                    \label{position_2}
                \end{figure}

        \subsection{Single rising bubble in three dimensions}
            In this subsection, we consider a natural 3-D analogue of the single rising bubble problem, as presented in section \ref{benckmark_bubble}. The domain is set as $\Omega=(0,1)^2 \times (0,2)$, with no-slip horizontal walls and free-slip vertical walls. The corresponding dimensionless numbers are the same as in section \ref{benckmark_bubble} (see Table \ref{parameters_bubble_rising}), and the initial interface is described by the 3-D profile
            \begin{equation}
                \phi^0 (x,y,z) =\tanh\Big(\frac{\sqrt{(x-0.5)^2+(y-0.5)^2+(z-0.5)^2}-0.3}{\mathrm{Cn}}\Big).
            \end{equation}
            To compare with the results presented in \cite{2015_barrett}, which solves the sharp interface model in the Barrett-Garcke-N\"{u}rnberg formulation using unfitted finite element approximations, the bubble radius is set to $0.3$ but not $0.25$. Again, the other physical variables are set to zero at the initial time step. We preform the numerical simulation on a very coarse mesh: $h=2^{-5}$ with $\tau=4\times{10}^{-3}$ and $\mathrm{Cn}=2\times{10}^{-2}$.

            The contour plots of $\phi_h$ at different time instants are presented in Figures~\ref{3D_bubble_test_case}. It is observed that the passage of time leads to the development of bubbles into ellipsoidal and skirted shapes, respectively. Moreover, the accuracy of the presented simulations is demonstrated by a comparison with the results in \cite{2015_barrett}, in terms of bubble shape and position, as displayed in Figures~\ref{3D_bubble_shape_1} and \ref{3D_bubble_shape_2}. And it can be observed that our simulation results are similar to the reference data, in spite of a very coarse mesh is performed. Meanwhile, since computing the benchmark quantities presented in section \ref{benckmark_bubble} on reconstructed interface are very complicated in three dimensions, they are not presented in this section. Instead, the evolution (with respect to dimensional time) of $\xi_1$ and $\xi_2$ are depicted in Figure~\ref{3D_xi}, and we can observe that they both remain close to 1 in both cases.
            \begin{figure}[ht]
                \centering
                \subfloat[PG+CM for test case 1, at a sequence of time instants, $t=0,1,1.8,2.5,\frac{21}{10}\sqrt{2}$.]{
                    \includegraphics[width=0.8\linewidth]{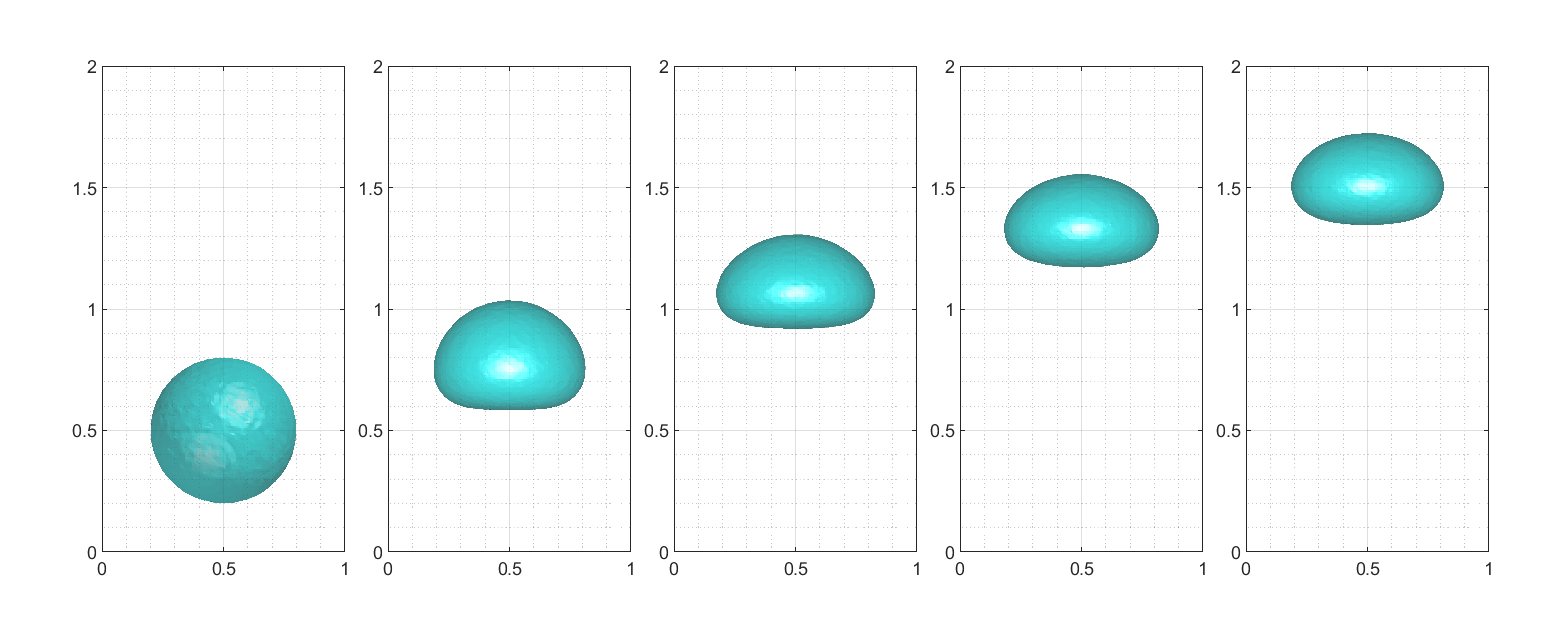}
                }
                \\
                \subfloat[AC+CM for test case 2, at a sequence of time instants, $t=0,0.5,0.8,1.1,\frac{21}{20}\sqrt{2}$.]{
                    \includegraphics[width=0.8\linewidth]{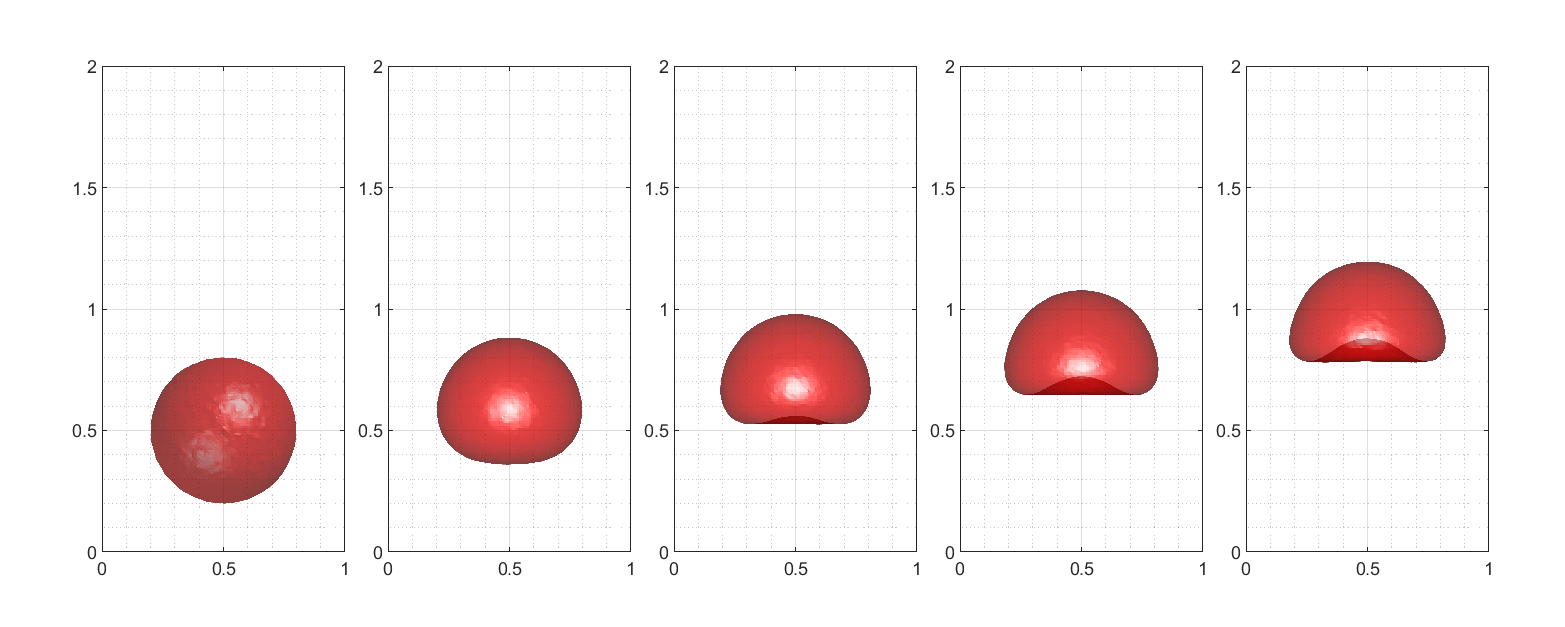}
                }
                \caption{Views of the bubble shape for the 3D single rising bubble problem from the front.}
                \label{3D_bubble_test_case}
            \end{figure}
            \begin{figure}[ht]
                \centering
                \subfloat[PG+CM for test case 1.]{
                    \includegraphics[width=0.4\linewidth]{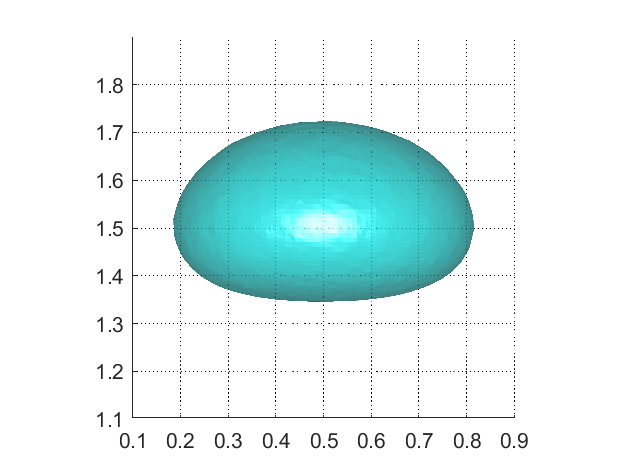}
                }
                \subfloat[Barrett et al. \cite{2015_barrett}.]{
                    \includegraphics[width=0.3\linewidth]{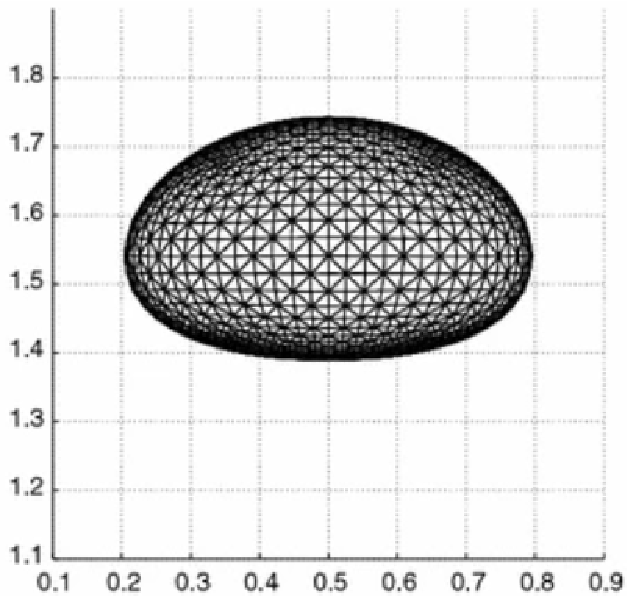}
                }
                \caption{Views of bubble shape from the front at dimensional time $t^*=3$.}
                \label{3D_bubble_shape_1}
            \end{figure}
            \begin{figure}[ht]
                \centering
                \subfloat[AC+CM for test case 2.]{
                    \includegraphics[width=0.34\linewidth]{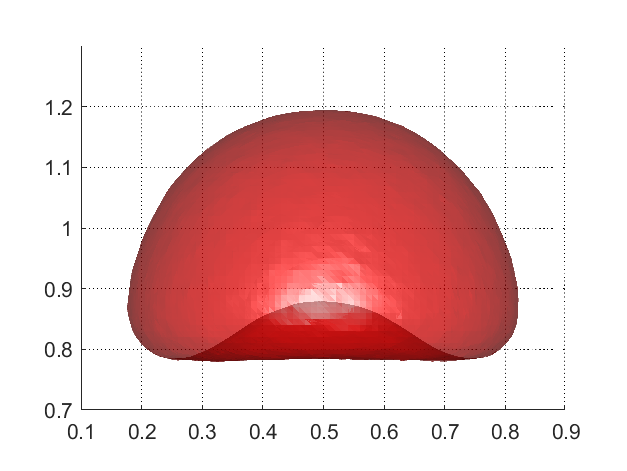}
                }
                \subfloat[Barrett et al. \cite{2015_barrett}.]{
                    \includegraphics[width=0.3\linewidth]{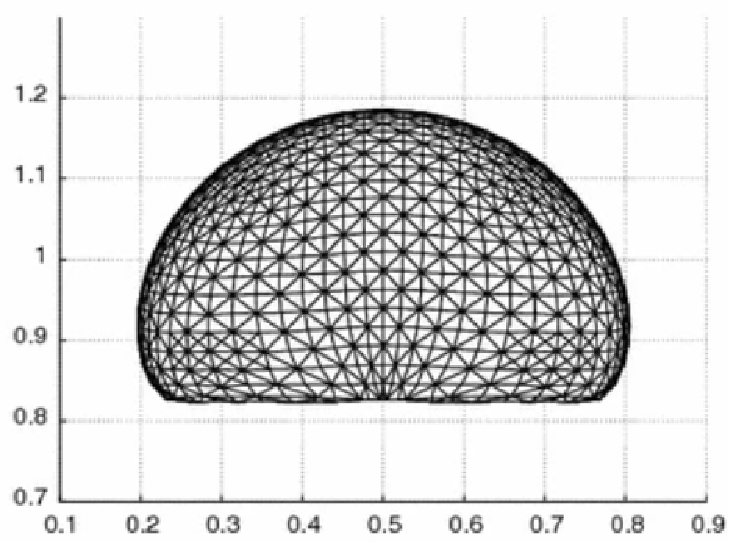}
                }
                \caption{Views of bubble shape from the front at dimensional time $t^*=1.5$.}
                \label{3D_bubble_shape_2}
            \end{figure}
            \begin{figure}[ht]
                \centering
                \subfloat[PG+CM for test case 1.]{
                    \includegraphics[width=0.3\linewidth]{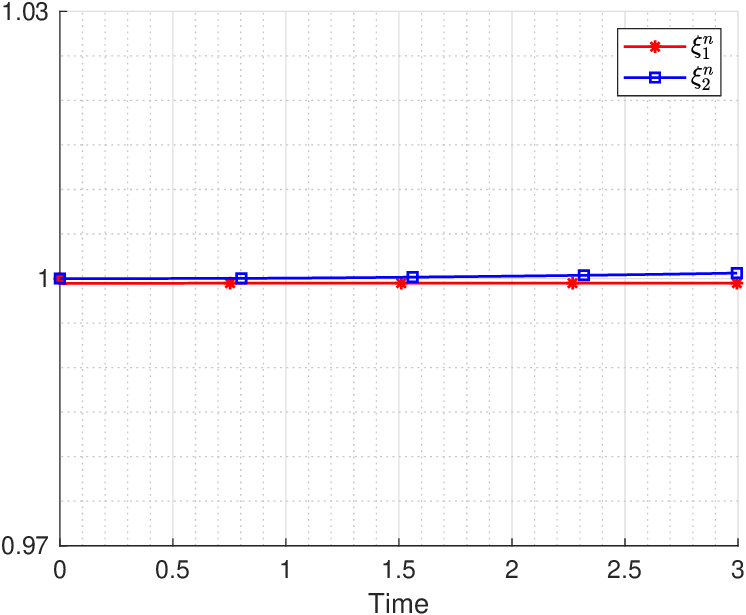}
                }
                \subfloat[AC+CM for test case 2.]{
                    \includegraphics[width=0.3\linewidth]{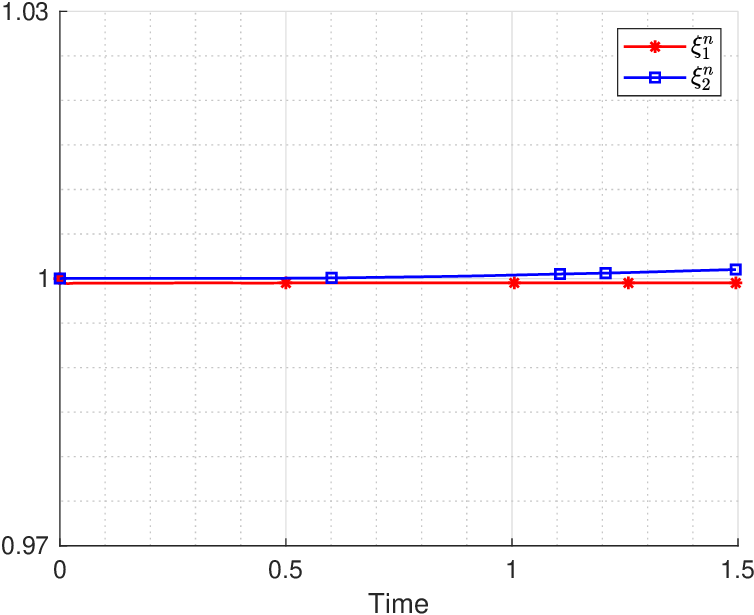}
                }
                \caption{Time evolution of the auxiliary variables.}
                \label{3D_xi}
            \end{figure}

    \section{Concluding remarks}
    \label{section_5}
        In this work, we have developed two efficient second-order BDF-type, finite element numerical schemes for the Abels-Garcke-Gr{\"u}n (AGG) model based on the MSAV approach. Both schemes are proved to be uniquely solvable and unconditionally energy stable. The accuracy and robustness of the proposed schemes are confirmed by several representative numerical simulation results.

        Irrespective of the exhibited advantages of the proposed schemes, there are still several issues that require further explorations. (i) The cut-off approach is still utilized to preserve the bound of $\phi_h^n$ when evaluating $\rho_h^n$ and $\eta_h^n$. Meanwhile, a rigorous justification of the maximum bound principle is not available for the Cahn-Hilliard equation, in contrast to that of the Allen-Cahn equation. In turn, we firmly believe that a theoretical analysis of the maximum norm of the numerical solution becomes a very important issue~\cite{2023_huang}, especially for two-phase incompressible flows with different densities. In particular, it is  anticipated that the method proposed in \cite{2022_huang} may successfully accomplish this objective in the future research. (ii) The decoupled structure attributed to the SAV approach raises the computational cost for extra load vectors and linear systems. Although substantial improvements have been achieved~\cite{2020_huang}, the associated implementation of this complicated system necessitates a further study. (iii) A thorough investigation is required to address the challenge of eliminating non-physical acoustic waves when implementing the scheme \eqref{MCHNSt_2} in the presence of a large density ratio.

        Finally, several future research objectives are presented: (i) a theoretical justification of convergence and error analysis for the proposed schemes, and (ii) exploration of efficient and structure-preserving schemes for more complex two-phase incompressible flows with different densities, such as magnetohydrodynamic and ferrohydrodynamic systems.

    \section*{Acknowledgments}
    The research of M. Li was supported in part by the NSFC 12271082, 62231016.
    % The research of C. Wang was supported in part by NSF DMS-2012269, DMS 2309548.

%     \section*{Declaration of Interest Statement}
%    The authors declare that they have no known competing interests or personal relationships that could have appeared to influence the work reported in this paper.

    \small
    \bibliographystyle{abbrv}
    \bibliography{ref}

\begin{thebibliography}{10}

\bibitem{2023_abels}
H.~Abels, H.~Garcke, and A.~Giorgini.
\newblock {Global regularity and asymptotic stabilization for the
  incompressible Navier--Stokes-Cahn--Hilliard model with unmatched densities}.
\newblock {\em Math. Ann.}, pages 1--55, 2023.

\bibitem{2012_abels}
H.~Abels, H.~Garcke, and G.~Gr{\"u}n.
\newblock {Thermodynamically consistent, frame indifferent diffuse interface
  models for incompressible two-phase flows with different densities}.
\newblock {\em Math. Models Methods Appl. Sci.}, 22(03):1150013, 2012.

\bibitem{2014_aki}
G.~L. Aki, W.~Dreyer, J.~Giesselmann, and C.~Kraus.
\newblock {A quasi-incompressible diffuse interface model with phase
  transition}.
\newblock {\em Math. Models Methods Appl. Sci.}, 24(05):827--861, 2014.

\bibitem{2012_aland}
S.~Aland and A.~Voigt.
\newblock {Benchmark computations of diffuse interface models for
  two-dimensional bubble dynamics}.
\newblock {\em Internat. J. Numer. Methods Fluids}, 69(3):747--761, 2012.

\bibitem{2015_barrett}
J.~W. Barrett, H.~Garcke, and R.~N{\"u}rnberg.
\newblock {A stable parametric finite element discretization of two-phase
  Navier--Stokes flow}.
\newblock {\em J. Sci. Comput.}, 63:78--117, 2015.

\bibitem{2023_cai}
W.~Cai, W.~Sun, J.~Wang, and Z.~Yang.
\newblock {Optimal Error Estimates of Unconditionally Stable Finite Element
  Schemes for the Cahn--Hilliard--Navier--Stokes System}.
\newblock {\em SIAM J. Numer. Anal.}, 61(3):1218--1245, 2023.

\bibitem{2012_campillo}
E.~Campillo-Funollet, G.~Gr\"un, and F.~Klingbeil.
\newblock {On modeling and simulation of electrokinetic phenomena in two-phase
  flow with general mass densities}.
\newblock {\em SIAM J. Appl. Math.}, 72(6):1899--1925, 2012.

\bibitem{2021_chen}
C.~Chen and X.~Yang.
\newblock {Fully-discrete finite element numerical scheme with decoupling
  structure and energy stability for the Cahn--Hilliard phase-field model of
  two-phase incompressible flow system with variable density and viscosity}.
\newblock {\em ESAIM Math. Model. Numer. Anal.}, 55(5):2323--2347, 2021.

\bibitem{2022_chen}
C.~Chen and X.~Yang.
\newblock {Highly efficient and unconditionally energy stable semi-discrete
  time-marching numerical scheme for the two-phase incompressible flow
  phase-field system with variable-density and viscosity}.
\newblock {\em Sci. China Math.}, 65(12):2631--2656, 2022.

\bibitem{chen19a}
W.~Chen, W.~Feng, Y.~Liu, C.~Wang, and S.~Wise.
\newblock A second order energy stable scheme for the {Cahn-Hilliard-Hele-Shaw}
  equation.
\newblock {\em Discrete Contin. Dyn. Syst. Ser. B}, 24(1):149--182, 2019.

\bibitem{chen22c}
W.~Chen, D.~Han, C.~Wang, S.~Wang, X.~Wang, and Y.~Zhang.
\newblock Error estimate of a decoupled numerical scheme for the
  {Cahn-Hilliard-Stokes-Darcy} system.
\newblock {\em IMA J. Numer. Anal.}, 42(3):2621--2655, 2022.

\bibitem{chen24c}
W.~Chen, J.~Jing, Q.~Liu, C.~Wang, and X.~Wang.
\newblock Convergence analysis of a second order numerical scheme for the
  {Flory-Huggins-Cahn-Hilliard-Navier-Stokes} system.
\newblock {\em J. Comput. Appl. Math.}, 450:115981, 2024.

\bibitem{chen24b}
W.~Chen, J.~Jing, Q.~Liu, C.~Wang, and X.~Wang.
\newblock A second order accurate, positivity-preserving numerical scheme of
  the {Cahn-Hilliard-Navier-Stokes} system with {Flory-Huggins} potential.
\newblock {\em Commun. Comput. Phys.}, 35:633--661, 2024.

\bibitem{chen22b}
W.~Chen, J.~Jing, C.~Wang, and X.~Wang.
\newblock A positivity preserving, energy stable finite difference scheme for
  the {Flory-Huggins-Cahn-Hilliard-Navier-Stokes} system.
\newblock {\em J. Sci. Comput.}, 92:31, 2022.

\bibitem{chen16}
W.~Chen, Y.~Liu, C.~Wang, and S.~Wise.
\newblock An optimal-rate convergence analysis of a fully discrete finite
  difference scheme for {Cahn-Hilliard-Hele-Shaw} equation.
\newblock {\em Math. Comp.}, 85:2231--2257, 2016.

\bibitem{2016_chen}
Y.~Chen and J.~Shen.
\newblock {Efficient, adaptive energy stable schemes for the incompressible
  Cahn--Hilliard Navier--Stokes phase-field models}.
\newblock {\em J. Comput. Phys.}, 308:40--56, 2016.

\bibitem{2018_cheng}
Q.~Cheng and J.~Shen.
\newblock {Multiple scalar auxiliary variable (MSAV) approach and its
  application to the phase-field vesicle membrane model}.
\newblock {\em SIAM J. Sci. Comput.}, 40(6):A3982--A4006, 2018.

\bibitem{ChengQ2021a}
Q.~Cheng and C.~Wang.
\newblock Error estimate of a second order accurate scalar auxiliary variable
  {(SAV)} scheme for the thin film epitaxial equation.
\newblock {\em Adv. Appl. Math. Mech.}, 13:1318--1354, 2021.

\bibitem{2022_decaria}
V.~DeCaria, S.~Gottlieb, Z.~J. Grant, and W.~J. Layton.
\newblock {A general linear method approach to the design and optimization of
  efficient, accurate, and easily implemented time-stepping methods in CFD}.
\newblock {\em J. Comput. Phys.}, 455:110927, 2022.

\bibitem{2017_decaria}
V.~DeCaria, W.~Layton, and M.~McLaughlin.
\newblock {A conservative, second order, unconditionally stable artificial
  compression method}.
\newblock {\em Comput. Methods Appl. Mech. Engrg.}, 325:733--747, 2017.

\bibitem{2021_decaria}
V.~DeCaria and M.~Schneier.
\newblock {An embedded variable step IMEX scheme for the incompressible
  Navier--Stokes equations}.
\newblock {\em Comput. Methods Appl. Mech. Engrg.}, 376:113661, 2021.

\bibitem{2017_diegel}
A.~E. Diegel, C.~Wang, X.~Wang, and S.~M. Wise.
\newblock {Convergence analysis and error estimates for a second order accurate
  finite element method for the Cahn--Hilliard--Navier--Stokes system}.
\newblock {\em Numer. Math.}, 137:495--534, 2017.

\bibitem{2018_dong}
S.~Dong.
\newblock {Multiphase flows of N immiscible incompressible fluids: a
  reduction-consistent and thermodynamically-consistent formulation and
  associated algorithm}.
\newblock {\em J. Comput. Phys.}, 361:1--49, 2018.

\bibitem{2022_duan}
B.~Duan, B.~Li, and Z.~Yang.
\newblock {An energy diminishing arbitrary Lagrangian--Eulerian finite element
  method for two-phase Navier--Stokes flow}.
\newblock {\em J. Comput. Phys.}, 461:111215, 2022.

\bibitem{2022_haddad}
M.~El~Haddad and G.~Tierra.
\newblock {A thermodynamically consistent model for two-phase incompressible
  flows with different densities. Derivation and efficient energy-stable
  numerical schemes}.
\newblock {\em Comput. Methods Appl. Mech. Engrg.}, 389:114328, 2022.

\bibitem{2021_Ern_1}
A.~Ern and J.-L. Guermond.
\newblock {\em {Finite elements II---Galerkin approximation, elliptic and mixed
  PDEs}}, volume~73.
\newblock Springer, 2021.

\bibitem{2021_Ern_2}
A.~Ern and J.-L. Guermond.
\newblock {\em {Finite Elements III---First-Order and Time-Dependent PDEs}},
  volume~74.
\newblock Springer, 2021.

\bibitem{2006_feng}
X.~Feng.
\newblock {Fully Discrete Finite Element Approximations of the
  Navier--Stokes--Cahn-Hilliard Diffuse Interface Model for Two-Phase Fluid
  Flows}.
\newblock {\em SIAM J. Numer. Anal.}, 44(3):1049--1072, 2006.

\bibitem{2023_feng}
X.~Feng, Z.~Qiao, S.~Sun, and X.~Wang.
\newblock {An energy-stable Smoothed Particle Hydrodynamics discretization of
  the Navier-Stokes-Cahn-Hilliard model for incompressible two-phase flows}.
\newblock {\em J. Comput. Phys.}, 479:111997, 2023.

\bibitem{2020_fu}
G.~Fu.
\newblock {A divergence-free HDG scheme for the Cahn-Hilliard phase-field model
  for two-phase incompressible flow}.
\newblock {\em J. Comput. Phys.}, 419:109671, 2020.

\bibitem{2021_fu}
G.~Fu and D.~Han.
\newblock {A linear second-order in time unconditionally energy stable finite
  element scheme for a Cahn--Hilliard phase-field model for two-phase
  incompressible flow of variable densities}.
\newblock {\em Comput. Methods Appl. Mech. Engrg.}, 387:114186, 2021.

\bibitem{2022_gao}
Y.~Gao, D.~Han, X.~He, and U.~R{\"u}de.
\newblock {Unconditionally stable numerical methods for
  Cahn-Hilliard-Navier-Stokes-Darcy system with different densities and
  viscosities}.
\newblock {\em J. Comput. Phys.}, 454:110968, 2022.

\bibitem{2018_gong}
Y.~Gong, J.~Zhao, X.~Yang, and Q.~Wang.
\newblock {Fully discrete second-order linear schemes for hydrodynamic phase
  field models of binary viscous fluid flows with variable densities}.
\newblock {\em SIAM J. Sci. Comput.}, 40(1):B138--B167, 2018.

\bibitem{2011_gross}
S.~Gross and A.~Reusken.
\newblock {\em {Numerical methods for two-phase incompressible flows}},
  volume~40.
\newblock Springer, 2011.

\bibitem{2013_grun}
G.~Gr\"{u}n.
\newblock {On convergent schemes for diffuse interface models for two-phase
  flow of incompressible fluids with general mass densities}.
\newblock {\em SIAM J. Numer. Anal.}, 51(6):3036--3061, 2013.

\bibitem{2014_grun}
G.~Gr\"{u}n and F.~Klingbeil.
\newblock {Two-phase flow with mass density contrast: stable schemes for a
  thermodynamic consistent and frame-indifferent diffuse-interface model}.
\newblock {\em J. Comput. Phys.}, 257:708--725, 2014.

\bibitem{2016_grun}
G.~Gr\"un and S.~Metzger.
\newblock {On micro--macro-models for two-phase flow with dilute polymeric
  solutions—modeling and analysis}.
\newblock {\em Math. Models Methods Appl. Sci.}, 26(05):823--866, 2016.

\bibitem{2006_guermond}
J.~L. Guermond, P.~Minev, and J.~Shen.
\newblock {An overview of projection methods for incompressible flows}.
\newblock {\em Comput. Methods Appl. Mech. Engrg.}, 195(44-47):6011--6045,
  2006.

\bibitem{2000_guermond}
J.-L. Guermond and L.~Quartapelle.
\newblock {A projection FEM for variable density incompressible flows}.
\newblock {\em J. Comput. Phys.}, 165(1):167--188, 2000.

\bibitem{2009_guermond}
J.-L. Guermond and A.~Salgado.
\newblock {A splitting method for incompressible flows with variable density
  based on a pressure Poisson equation}.
\newblock {\em J. Comput. Phys.}, 228(8):2834--2846, 2009.

\bibitem{2011_guermond}
J.-L. Guermond and A.~J. Salgado.
\newblock {Error analysis of a fractional time-stepping technique for
  incompressible flows with variable density}.
\newblock {\em SIAM J. Numer. Anal.}, 49(3):917--944, 2011.

\bibitem{2021_guo}
Z.~Guo, F.~Yu, P.~Lin, S.~Wise, and J.~Lowengrub.
\newblock {A diffuse domain method for two-phase flows with large density ratio
  in complex geometries}.
\newblock {\em J. Fluid Mech.}, 907:A38, 2021.

\bibitem{1996_gurtin}
M.~E. Gurtin, D.~Polignone, and J.~Vi\~{n}als.
\newblock {Two-phase binary fluids and immiscible fluids described by an order
  parameter}.
\newblock {\em Math. Models Methods Appl. Sci.}, 6(06):815--831, 1996.

\bibitem{2017_han}
D.~Han, A.~Brylev, X.~Yang, and Z.~Tan.
\newblock {Numerical analysis of second order, fully discrete energy stable
  schemes for phase field models of two-phase incompressible flows}.
\newblock {\em J. Sci. Comput.}, 70:965--989, 2017.

\bibitem{2015_han}
D.~Han and X.~Wang.
\newblock {A second order in time, uniquely solvable, unconditionally stable
  numerical scheme for Cahn--Hilliard--Navier--Stokes equation}.
\newblock {\em J. Comput. Phys.}, 290:139--156, 2015.

\bibitem{1981_hirt}
C.~Hirt and B.~Nichols.
\newblock {Volume of fluid (VOF) method for the dynamics of free boundaries}.
\newblock {\em J. Comput. Phys.}, 39(1):201--225, 1981.

\bibitem{1977_hohenberg}
P.~C. Hohenberg and B.~I. Halperin.
\newblock {Theory of dynamic critical phenomena}.
\newblock {\em Rev. Mod. Phys.}, 49(3):435--479, 1977.

\bibitem{2022_huang}
F.~Huang, J.~Shen, and K.~Wu.
\newblock {Bound/positivity preserving and unconditionally stable schemes for a
  class of fourth order nonlinear equations}.
\newblock {\em J. Comput. Phys.}, 460:111177, 2022.

\bibitem{2020_huang}
F.~Huang, J.~Shen, and Z.~Yang.
\newblock {A highly efficient and accurate new scalar auxiliary variable
  approach for gradient flows}.
\newblock {\em SIAM J. Sci. Comput.}, 42(4):A2514--A2536, 2020.

\bibitem{2023_huang}
Q.-A. Huang, W.~Jiang, J.~Z. Yang, and C.~Yuan.
\newblock {A structure-preserving, upwind-SAV scheme for the degenerate
  Cahn--Hilliard equation with applications to simulating surface diffusion}.
\newblock {\em J. Sci. Comput.}, 97(3):64, 2023.

\bibitem{2009_hysing}
S.~Hysing, S.~Turek, D.~Kuzmin, N.~Parolini, E.~Burman, S.~Ganesan, and
  L.~Tobiska.
\newblock {Quantitative benchmark computations of two-dimensional bubble
  dynamics}.
\newblock {\em Internat. J. Numer. Methods Fluids}, 60(11):1259--1288, 2009.

\bibitem{2022_jiang}
M.~Jiang, Z.~Zhang, and J.~Zhao.
\newblock {Improving the accuracy and consistency of the scalar auxiliary
  variable (SAV) method with relaxation}.
\newblock {\em J. Comput. Phys.}, 456:110954, 2022.

\bibitem{2020_khanwale}
M.~A. Khanwale, A.~D. Lofquist, H.~Sundar, J.~A. Rossmanith, and
  B.~Ganapathysubramanian.
\newblock {Simulating two-phase flows with thermodynamically consistent energy
  stable Cahn-Hilliard Navier-Stokes equations on parallel adaptive octree
  based meshes}.
\newblock {\em J. Comput. Phys.}, 419:109674, 2020.

\bibitem{2023_khanwale}
M.~A. Khanwale, K.~Saurabh, M.~Ishii, H.~Sundar, J.~A. Rossmanith, and
  B.~Ganapathysubramanian.
\newblock {A projection-based, semi-implicit time-stepping approach for the
  Cahn-Hilliard Navier-Stokes equations on adaptive octree meshes}.
\newblock {\em J. Comput. Phys.}, 475:111874, 2023.

\bibitem{2021_li}
M.~Li, Y.~Cheng, J.~Shen, and X.~Zhang.
\newblock {A bound-preserving high order scheme for variable density
  incompressible Navier-Stokes equations}.
\newblock {\em J. Comput. Phys.}, 425:109906, 2021.

\bibitem{2023_li}
N.~Li, J.~Wu, and X.~Feng.
\newblock {Filtered time-stepping method for incompressible Navier-Stokes
  equations with variable density}.
\newblock {\em J. Comput. Phys.}, 473:111764, 2023.

\bibitem{2022_li}
X.~Li and J.~Shen.
\newblock {On fully decoupled MSAV schemes for the
  Cahn--Hilliard--Navier--Stokes model of two-phase incompressible flows}.
\newblock {\em Math. Models Methods Appl. Sci.}, 32(03):457--495, 2022.

\bibitem{2003_liu}
C.~Liu and J.~Shen.
\newblock {A phase field model for the mixture of two incompressible fluids and
  its approximation by a Fourier-spectral method}.
\newblock {\em Phys. D}, 179(3-4):211--228, 2003.

\bibitem{2015_liu}
C.~Liu, J.~Shen, and X.~Yang.
\newblock {Decoupled energy stable schemes for a phase-field model of two-phase
  incompressible flows with variable density}.
\newblock {\em J. Sci. Comput.}, 62:601--622, 2015.

\bibitem{LiuY17}
Y.~Liu, W.~Chen, C.~Wang, and S.~Wise.
\newblock Error analysis of a mixed finite element method for a
  {Cahn-Hilliard-Hele-Shaw} system.
\newblock {\em Numer. Math.}, 135:679--709, 2017.

\bibitem{1998_lowengrub}
J.~Lowengrub and L.~Truskinovsky.
\newblock {Quasi-incompressible Cahn-Hilliard fluids and topological
  transitions}.
\newblock {\em R. Soc. Lond. Proc. Ser. A Math. Phys. Eng. Sci.},
  454(1978):2617--2654, 1998.

\bibitem{2013_magaletti}
F.~Magaletti, F.~Picano, M.~Chinappi, L.~Marino, and C.~M. Casciola.
\newblock {The sharp-interface limit of the Cahn--Hilliard/Navier--Stokes model
  for binary fluids}.
\newblock {\em J. Fluid Mech.}, 714:95--126, 2013.

\bibitem{2021_matsushita}
S.~Matsushita and T.~Aoki.
\newblock {Gas-liquid two-phase flows simulation based on weakly compressible
  scheme with interface-adapted AMR method}.
\newblock {\em J. Comput. Phys.}, 445:110605, 2021.

\bibitem{1981_prosperetti}
A.~Prosperetti.
\newblock {Motion of two superposed viscous fluids}.
\newblock {\em Phys. Fluids}, 24(7):1217--1223, 07 1981.

\bibitem{2007_pyo}
J.-H. Pyo and J.~Shen.
\newblock {Gauge--Uzawa methods for incompressible flows with variable
  density}.
\newblock {\em J. Comput. Phys.}, 221(1):181--197, 2007.

\bibitem{2021_rohde}
C.~Rohde and L.~von Wolff.
\newblock {A ternary Cahn--Hilliard--Navier--Stokes model for two-phase flow
  with precipitation and dissolution}.
\newblock {\em Math. Models Methods Appl. Sci.}, 31(01):1--35, 2021.

\bibitem{2018_shen_1}
J.~Shen and J.~Xu.
\newblock {Convergence and error analysis for the scalar auxiliary variable
  (SAV) schemes to gradient flows}.
\newblock {\em SIAM J. Numer. Anal.}, 56(5):2895--2912, 2018.

\bibitem{2018_shen_2}
J.~Shen, J.~Xu, and J.~Yang.
\newblock {The scalar auxiliary variable (SAV) approach for gradient flows}.
\newblock {\em J. Comput. Phys.}, 353:407--416, 2018.

\bibitem{2015_shen}
J.~Shen and X.~Yang.
\newblock {Decoupled, energy stable schemes for phase-field models of two-phase
  incompressible flows}.
\newblock {\em SIAM J. Numer. Anal.}, 53(1):279--296, 2015.

\bibitem{2018_roudbari}
M.~Shokrpour~Roudbari, G.~{\c{S}}im{\c{s}}ek, E.~H. van Brummelen, and K.~G.
  van~der Zee.
\newblock {Diffuse-interface two-phase flow models with different densities: A
  new quasi-incompressible form and a linear energy-stable method}.
\newblock {\em Math. Models Methods Appl. Sci.}, 28(04):733--770, 2018.

\bibitem{2019_sitompul}
Y.~P. Sitompul and T.~Aoki.
\newblock {A filtered cumulant lattice Boltzmann method for violent two-phase
  flows}.
\newblock {\em J. Comput. Phys.}, 390:93--120, 2019.

\bibitem{1994_sussman}
M.~Sussman, P.~Smereka, and S.~Osher.
\newblock {A level set approach for computing solutions to incompressible
  two-phase flow}.
\newblock {\em J. Comput. Phys.}, 114(1):146--159, 1994.

\bibitem{2023_ten}
M.~F.~P. ten Eikelder, K.~G. van~der Zee, I.~Akkerman, and D.~Schillinger.
\newblock {A unified framework for Navier--Stokes Cahn--Hilliard models with
  non-matching densities}.
\newblock {\em Math. Models Methods Appl. Sci.}, 33(01):175--221, 2023.

\bibitem{WangM2021}
M.~Wang, Q.~Huang, and C.~Wang.
\newblock A second order accurate scalar auxiliary variable {(SAV)} numerical
  method for the square phase field crystal equation.
\newblock {\em J. Sci. Comput.}, 88(2):33, 2021.

\bibitem{2019_wang}
Z.~Wang, S.~Dong, M.~S. Triantafyllou, Y.~Constantinides, and G.~E.
  Karniadakis.
\newblock {A stabilized phase-field method for two-phase flow at high Reynolds
  number and large density/viscosity ratio}.
\newblock {\em J. Comput. Phys.}, 397:108832, 2019.

\bibitem{2017_wu}
J.~Wu, J.~Shen, and X.~Feng.
\newblock {Unconditionally stable Gauge--Uzawa finite element schemes for
  incompressible natural convection problems with variable density}.
\newblock {\em J. Comput. Phys.}, 348:776--789, 2017.

\bibitem{2021_yang}
K.~Yang and T.~Aoki.
\newblock {Weakly compressible Navier-Stokes solver based on evolving pressure
  projection method for two-phase flow simulations}.
\newblock {\em J. Comput. Phys.}, 431:110113, 2021.

\bibitem{2021_yang_1}
X.~Yang.
\newblock {A novel fully-decoupled, second-order and energy stable numerical
  scheme of the conserved Allen--Cahn type flow-coupled binary surfactant
  model}.
\newblock {\em Comput. Methods Appl. Mech. Engrg.}, 373:113502, 2021.

\bibitem{2020_yang}
X.~Yang and G.-D. Zhang.
\newblock {Convergence analysis for the invariant energy quadratization (IEQ)
  schemes for solving the Cahn--Hilliard and Allen--Cahn equations with general
  nonlinear potential}.
\newblock {\em J. Sci. Comput.}, 82:1--28, 2020.

\bibitem{2019_yang}
Z.~Yang and S.~Dong.
\newblock {An unconditionally energy-stable scheme based on an implicit
  auxiliary energy variable for incompressible two-phase flows with different
  densities involving only precomputable coefficient matrices}.
\newblock {\em J. Comput. Phys.}, 393:229--257, 2019.

\bibitem{2014_zhang}
J.~Zhang and M.-J. Ni.
\newblock {Direct simulation of multi-phase MHD flows on an unstructured
  Cartesian adaptive system}.
\newblock {\em J. Comput. Phys.}, 270:345--365, 2014.

\bibitem{2022_zhang}
Y.~Zhang and J.~Shen.
\newblock {A generalized SAV approach with relaxation for dissipative systems}.
\newblock {\em J. Comput. Phys.}, 464:111311, 2022.

\end{thebibliography}
\end{document}